\documentclass[12pt]{article}
\usepackage{amsmath,amsthm,amssymb,latexsym,color}
\usepackage{bbm}
\usepackage{hyperref}

\date{}

\theoremstyle{plain}
\newtheorem{theor}{Theorem}[section]
\newtheorem{claim}[theor]{Claim}
\newtheorem{prop}[theor]{Proposition}

\newtheorem{cor}[theor]{Corollary}
\newtheorem{lemma}[theor]{Lemma}
\newtheorem{rem}[theor]{Remark}
\theoremstyle{remark}

\def\R{{\mathbb R}}
\def\N{{\mathbb N}}
\def\Z{{\mathbb Z}}
\def\C{\mathbb{C}}

\newcommand{\lc}{\lceil}
\newcommand{\rc}{\rceil}
\def\la{\left\langle}
\def\r{\right}
\def\ra{\r\rangle}
\def\Re{\mbox{Re}}
\def\Im{\mbox{Im}}
\def\LL{L}

\def\Prob{{\mathbb P}}
\def\Exp{{\mathbb E}}
\def\Event{{\mathcal E}}
\def\cf{\mathcal{Q}}
\newcommand{\p}{\mathbb{P}}

\def\col{C}
\def\row{R}

\def\cproj{{\rm P}}
\def\spn{{\rm span}\,}
\def\supp{{\rm supp }}

\def\idmat{{\rm Id}}
\def\d{{\rm dist}}
\def\Net{{\mathcal N}}

\newcommand{\lam}{\lambda}
\newcommand{\eps}{\varepsilon}

\def\MSet{{\mathcal M}}
\def\kset{{\mathcal A}}
\def\lset{{\mathcal L}}
\def\ls{{\mathcal L} {\mathcal S}}
\def\lr{{\mathcal L} {\mathcal R}}
\def\hset{{\mathcal H}}
\def\kapset{{\mathcal K}}
\def\rhoset{{\mathcal P}}
\def\est{{\rm SB}}
\def\triv{{\rm TE}}
\def\RSet{{\mathcal R}}

\def\of{{\eta}}
\def\wset{{\mathcal W}}

\def\wbl{\wset_{b}^1}
\def\wbg{\wset_{b}^2}
\def\hset{{\mathcal H}}
\def\RR{Q}
\newcommand{\Mc}{\mathcal{M}_{n,d}}
\newcommand{\hq}{h_{q}}
\def\cards{{\rm cs}}
\def\cardr{{\rm cr}}
\def\heis{{\rm hs}}
\def\heir{{\rm hr}}

\def\Om0{\Omega_0}
\def\omep{\Omega_{2m,\eps}}

\def\xyzv{V}

\def\nx{\| x \|}

\newcommand{\st}{\mathcal{T}}
\newcommand{\il}{I^\ell}
\newcommand{\ir}{I^r}

\def\f{{\mathcal F}}

\def\ii{{\mathcal I}}
\def\iii{\mbox{\bf i}}

\def\nn{n_3}
\def\ww{z}
\def\WW{W}

\def\BB{\mathcal{B}}

\def\ko{\ell_0}

\def\aaa{a_3}
\def\KK{K}
\def\CC{\mathcal{K}}

\def\D{{\mathcal{D}_{n,d}}}

\def\nrangr{G}

\def\outnbr{{\mathcal N}_{\nrangr}^{out}}

\def\outnbrpr{{\mathcal N}_{\nrangr'}^{out}}
\def\innbr{{\mathcal N}_{\nrangr}^{in}}

\def\innbrpr{{\mathcal N}_{\nrangr'}^{in}}
\def\edg{{E}_{\nrangr}}

\def\edgpr{{E}_{\nrangr'}}
\def\inedg{{E}_{\nrangr}^{in}}

\def\outedg{{E}_{\nrangr}^{out}}

\title{Structure of eigenvectors of random regular digraphs}
\author{
Alexander E. Litvak
\and
Anna Lytova
\and
Konstantin Tikhomirov
\and
Nicole Tomczak-Jaegermann
\and
Pierre Youssef
}
\newcommand\address{\noindent\leavevmode

\medskip
\noindent
Alexander E. Litvak
and Nicole Tomczak-Jaegermann,\\
Dept.~of Math.~and Stat.~Sciences,\\
University of Alberta, \\
Edmonton, AB, Canada, T6G 2G1.\\
\texttt{\small
e-mails:  aelitvak@gmail.com \, \, and \, \,
nicole.tomczak@ualberta.ca}\\

\medskip

\noindent
Anna Lytova,\\
Faculty of Math., Physics, and Comp. Science,\\
University of Opole,\\
plac Kopernika 11A, 45-040,\\
Opole, Poland.\\
\texttt{\small
e-mail:
alytova@math.uni.opole.pl
}\\

\medskip

\noindent
 Konstantin Tikhomirov,\\
Dept.~of Math.,
Princeton University,\\
Fine Hall, Washington road,\\
Princeton, NJ 08544.\\
\texttt{\small
e-mail:   kt12@math.princeton.edu}\\

\medskip

\noindent
Pierre Youssef,\\
Universit\'e Paris Diderot,\\
Laboratoire de Probabilit\'es, Statistiques et Mod\'elisation,\\
75013 Paris, France.\\
\texttt{\small
e-mail:  youssef@lpsm.paris}
}

\begin{document}

\maketitle

\abstract{
Let $d$ and $n$ be integers satisfying
  $C\leq d\leq \exp(c\sqrt{\ln n})$ for some universal constants
  $c, C>0$, and let $z\in \C$.  Denote by $M$ the adjacency
  matrix of a random $d$-regular directed graph on $n$ vertices.  In
  this paper, we study the structure of the kernel of submatrices of
  $M-z\,\idmat$, formed by removing a subset of rows.  We show that
  with large probability the kernel consists of two non-intersecting
  types of vectors, which we call {\it very steep} and {\it gradual
    with many levels}.  As a corollary, we show, in particular, that
  every eigenvector of $M$, except for constant multiples of
  $(1,1,\dots,1)$, possesses a weak delocalization property:
  its level sets have cardinality less
  than $Cn\ln^2 d/\ln n$.   For a large constant $d$
  this  provides a   principally
new structural information on eigenvectors, implying that
the number of their level sets grows to infinity with $n$.
As a key technical ingredient of our proofs we introduce a  decomposition of $\C^n$
into vectors of different degrees of ``structuredness,"
which is  an alternative to the decomposition based on the {\it least common denominator}
in the regime when the underlying random matrix is very sparse.
}

\smallskip

\noindent
{\small \bf AMS 2010 Classification:}
{\small
primary: 60B20, 15B52, 46B06, 05C80;
secondary: 46B09, 60C05
}


\tableofcontents

\section{Introduction}
\label{s:intro}

Fix a large integer $n$ and an integer $d$ in the range $\{3,4,\dots,n-3\}$.
Denote by $\Mc$ the collection of all $n\times n$ matrices with entries taking values in $\{0,1\}$
such that in any row and any column there are exactly $d$ ones.
Every matrix from this set can be viewed as the adjacency matrix of a $d$-regular directed graph on $n$-vertices,
where we allow loops but no multiple edges.
For a subset $K\subset[n]:=\{1,2,\dots,n\}$ and a matrix $B$, by $B^K$ we denote
the submatrix of $B$ formed by rows $\row_i(B)$, $i\in K$.
In this paper, we consider the structure of the kernel of random
linear operators of the form $(M-z\idmat)^K$, where $M$ is a random
element of $\Mc$ (with respect to the uniform measure) and $z$ is a
fixed complex number.

Our motivation for this study is multifold.
We obtain new results regarding delocalization properties
of approximate eigenvectors for very sparse random matrices.
Apart of being of an independent interest, these results provide
new insights into spectral properties of random graphs.
Furthermore, as is shown in \cite{LLTTY third part}, our results
are key to understanding the intermediate singular values
of the matrix $M-z\idmat$, which in turn are crucial for establishing
the limiting spectral distribution of appropriately
rescaled adjacency matrices, when the dimension $n$ tends to infinity.

Spectral properties of random graphs, in particular, graphs with predefined degree sequences,
have been an object of active research.
In the case of $d$-regular undirected graphs, the magnitude of the second largest eigenvalue
as well as the limiting spectral distribution of the adjacency matrix
have been considered in various regimes and for different models of randomness
(uniform, permutation, and configuration models).
In particular, the study of the second largest eigenvalue has been motivated by the well known
relation between the magnitude of the spectral gap and the graph expansion properties \cite{Alon Milman, Dodziuk, HLW}.
We refer, in particular, to \cite{BFSU, FKS, Friedman, DJPP, Puder, Bordenave, CGJ, TY} and references therein
as well as to the survey \cite{HLW} for more information on spectral expanders.
The limiting spectral distribution of an (appropriately rescaled) adjacency matrix of an undirected $d$-regular graph
follows the Kesten--McKay law \cite{Kesten walks, McKay} which, for degree $d$ converging to infinity with $n$,
coincides with the classical semi-circle law \cite{TVW}. We refer, in particular, to \cite{Dumitriu Pal, BKY, BHY}
for recent results in this direction.

In the case of directed $d$-regular graphs,
establishing the limiting spectral distribution for constant $d$ presents a major challenge
not resolved as of this writing. It is conjectured that
the limiting spectral distribution of the
appropriately rescaled adjacency matrix follows the {\it oriented Kesten--McKay law} (see, for example, \cite[p.~52]{BC}).
For $\min(d,n-d)\to\infty$, the limiting distribution has been conjectured to follow the circular law,
thus matching (up to rescaling) the standard i.i.d.~non-Hermitian models.
Very recently, this conjecture has been partially resolved in the uniform and permutation models of randomness
under the assumption that the degree $d$ grows with $n$ at least poly-logarithmically
\cite{Cook circular, BCZ}.
However, the case of very slowly growing $d$ has remained open.
This (very sparse) regime is in certain respects fundamentally different
as it requires special handling of not only the smallest but also the {\it intermediate} singular values
of the shifted adjacency matrix $M-z\idmat$.

In \cite{LLTTY first part}, building upon arguments
in  \cite{Cook expansion, Cook adjacency, LLTTY:15, LLTTY-cras},
we have established lower bounds for the smallest singular value of the matrix $M-z\idmat$
which work for all $d$ larger than a large absolute constant (with probability estimates depending on $d$).
In this paper, we consider a more technical
(and more difficult) aspect of the study by establishing a structural theorem for the kernel of random operators
$(M-z\idmat)^K$.

This paper is an autonomous part in the series
of works in which we resolve the conjecture for the limiting spectral distribution
for any function $d=d(n)$ growing to infinity with the dimension $n$. In \cite{LLTTY third part},
we  use the main result of this paper
together with additional probabilistic arguments to derive bounds for intermediate singular values
of $M-z\idmat$ and, applying the standard argument of Girko \cite{Girko}, to
establish the circular law for the spectrum in the regime $d\to\infty$, resolving the
corresponding conjecture (see \cite[p.~5]{Cook circular}).
In order to avoid repetitions, we leave further discussion of Girko's approach and, more generally, the historical overview
of the circular law to \cite{LLTTY third part}.
As mentioned before, while being very useful in proving the limiting
law, the structural theorem is of interest on its own and in the case
$K=[n]$ can be viewed as a {\it delocalization} statement about
approximate eigenvectors of very sparse random matrices.

We start with formulating a ``soft'' version of the main result. In what follows,
given a vector $x\in\C^n$, by $x^{*}=(x_i^*)_i$ we denote the non-increasing
rearrangement of $(|x_i|)_i$. We also recall that for an $n\times n$ matrix $B$ and a subset $K\subset[n]$,
$B^K$ denotes the $|K|\times n$ matrix with rows $\row_i(B)$, $i\in K$.

\begin{theor}[Structural theorem]\label{ker th}
There are universal constants $c,c', C>0$ with the following property.
Let $n\geq C$, $C\leq d\leq \exp(c\sqrt{\ln n})$, and let $M$ be uniformly distributed on $\Mc$.
Let $z\in\C$ be such that $\vert z\vert \leq \sqrt{d}\, \ln d$.
Fix $d^{-1/2}\leq a\leq 1$ and any subset $\KK\subset[n]$ with $0\leq \vert \KK^c\vert\leq n/d^3$.
Set
\begin{align*}
\rho:=\max\Big(n^{-c}, \exp\Big(-\big(n/(1+\vert \KK^c\vert)\big)^{\frac{c\ln\ln d}{\ln d}}\Big)\Big),\;
\delta:=\frac{C\ln^2 d}{\ln (1/\rho)},\;
q:=\max(a\vert \KK^c\vert, 1).
\end{align*}
Then with probability at least $1-1/n$ every unit complex vector $x$ such that
$$
 \|(M-z\,\idmat)^K\,x\|_2\leq \vert\KK^c\vert^3\, n^{-6},
$$ satisfies one of the following two conditions:
\begin{itemize}
\item (Gradual with many levels) One has $\, \, \, x_i^*\leq (n/i)^{3} x_{q}^*\, \, $ for all $\, \, i\leq q$,\,
$$
 x_{i}^*\leq d^3(n/i)^6 x_{\lfloor c' n\rfloor}^*\quad \quad \mbox{ for all }
 \,\,\, q\leq i\leq \lfloor c' n\rfloor,
$$
 and
$$
\Big|\Big\{i\leq n:\,|x_i-\lambda |\leq \rho x_{\lfloor c' n\rfloor}^* \Big\}\Big|\leq  \delta n\quad\quad \mbox{for all $\,\,\, \lambda\in\C$}.
$$
\item (Very steep) $x_i^* > 0.9(n/i)^3 x_{q}^*$ for some  $i\leq q$.
\end{itemize}
\end{theor}

We note that the probability bound in the theorem can be replaced with
$1-n^{-r}$ for any fixed $r\geq 1$ by adjusting the absolute constants.
The terms ``steep'' and ``gradual'' vectors will be discussed in detail below.
In full generality, the theorem will be given at the end of the paper, see Theorem~\ref{ker th gen}.
The above statement asserts that, given a vector $x\in\C^n$ which is close to the kernel of $(M-z\idmat)^K$,
either the coordinates of $x^*$ decrease fast for small indices ($x$ is ``very steep'') or,
if this is not the case, the vector $x$ is spread (has many non-zero components) and, moreover,
for any complex $\lambda$ very few coordinates of $x$ are concentrated around $\lambda$.
Note that when $|K^c|=1$, i.e., when we consider normal vectors to a linear subspace spanned by
$n-1$ matrix rows, the second assertion never holds (since $q=1$), and the theorem says that
the normal vectors are all spread and have many levels. Moreover, for $K=\{1,2,\dots,n\}$,
combining the theorem with a simple covering argument, we obtain the following delocalization
property.

\begin{cor}[Delocalization properties of eigenvectors]\label{cor: deloc}
There are universal constants $c, C >0$ such that the following holds.
Let $n,d$ and $M$ be as in Theorem~\ref{ker th}.
Then with probability at least $1-1/n$
any eigenvector $x$ of $M$, which is not parallel to $(1,1,\dots,1)$, satisfies
$x_i^*\leq d^3(n/i)^6 x_{\lfloor c n\rfloor}^*$ for all $1\leq i\leq  \lfloor c n\rfloor$, and, moreover,
$$
 \forall \lambda \in \C \quad\quad   \Big|\Big\{i\leq n:\,|x_i-\lambda |\leq n^{- c}
   x_{\lfloor c n\rfloor}^* \Big\}\Big|\leq   C n\ln^2 d/\ln n.
$$
\end{cor}

For completeness, we give a proof of Corollary~\ref{cor: deloc} at the end of the paper.
We would like to notice that delocalization properties of eigenvectors for various models have been
a focus of active research. In the case of sparse matrices we refer to \cite{BHY2} for eigenvector statistics for Erd\H{o}s--R\'enyi graphs and to \cite{AM,AM2,BS,G} for delocalization properties of eigenvectors (and almost eigenvectors) of undirected regular graphs.  In the non-Hermitian setting (relevant to our present work)
we refer to \cite{RV deloc,RV no gaps}.
The term {\it delocalization} usually refers to upper bounds on the $\ell_\infty$-norm of a random vector or,
more generally, upper bounds for the scalar product with a fixed unit vector.
The delocalization statement provided by Corollary~\ref{cor: deloc},
is closer to the concept of {\it no-gaps delocalization}, introduced in \cite{RV no gaps},
which bounds the $\ell_2$-mass of the vector supported on every subset of coordinates of a given size.
At the same time, Corollary~\ref{cor: deloc} not only provides lower bounds for the order statistics of eigenvectors
but also measures cardinalities of sets of almost equal coordinates, thus giving an additional structural information
(very important in our context).
Corollary~\ref{cor: deloc}, to our best knowledge,
is the first statement which provides quantitative information on the delocalization
for non-Hermitian random matrices with a {\it constant} number of non-zero elements in rows/columns.

A fundamental feature of the main theorem is that it provides information on the kernel of the matrix
for large constant $d$, with the ``unstructuredness'' measured in terms of the dimension $n$.
Here, ``unstructuredness'' refers to the allowed number of approximately equal
components of a vector.
For example,
we show that for $M^{\{2,\dots,n\}}$ with large probability
any vector in the kernel can have at most $O(n/\ln n)$
equal components.
We expect that the theorem and the argument used in its proof will turn to be useful
in the study of the spectrum of random directed graphs of constant degree.
In fact, combined with some known arguments, our result implies that
the random adjacency matrix $M$ has rank {\it at least} $n-1$ with probability going to one
with $n$ ($d$ being a large constant) \cite{LLTTY rank n-1}.

Theorem~\ref{ker th} can be interpreted as follows:
with probability very close to one we have
$$
 \inf\limits_{x\in T}\|(M-z\idmat)^K x\|_2>0,
$$
where the infimum is taken over the set $T\subset\C^n$ of all non-zero vectors,
which are neither ``very steep'' and
not ``gradual with many levels.''
Estimates of this type fall into a large body of research dealing with
matrix singularity and structural properties of null vectors for various models of randomness.
A possible strategy in estimating the infimum $\inf\limits_{x\in S}\|Bx\|_2$ (for a random matrix $B$
and a subset $S$, say, the unit Euclidean sphere)
consists in representing $S$ as a union of subsets $\bigcup_{\alpha}S_\alpha$ grouping together
vectors with a similar structure, and then combining bounds
for $\inf\limits_{x\in S_\alpha}\|Bx\|_2$. In turn,
each of the infima is bounded using the structural information about
vectors in $S_\alpha$ and may involve, depending on the problem, a discretization of the subset
(i.e., a version of a covering argument).
A very incomplete list of works involving this approach in the study of square non-Hermitian
or ``almost square'' matrices is
\cite{KKS95, TV bernoulli, TVW} (singularity of random Bernoulli matrices),
\cite{LPRT, RV, RV09, RT} (matrices with i.i.d. entries with a tail decay condition),
\cite{TV sparse, {TV ann math}, GT, BasRud, BasRud circ} (sparse matrices with i.i.d. entries, see also \cite{LiRi}
for the non-i.i.d. case), \cite{Cook adjacency, LLTTY:15, Cook circular, BCZ}
(adjacency matrices of directed $d$-regular graphs).
For a detailed exposition of this method we refer to \cite{RV2010} and \cite{V12}.
The decomposition into subsets differs significantly, depending on the randomness model.
In particular, estimating the singularity probability for Bernoulli matrices in \cite{KKS95,TV bernoulli, TVW} involved defining
the {\it combinatorial dimension} of certain discrete vectors.
Further, the idea from \cite{LPRT} of splitting the Euclidean sphere into sets of ``close to sparse''
and ``far from sparse'' vectors was developed in \cite{RV, RV09}, where the notion of {\it compressible}
and {\it incompressible} vectors appeared. In \cite{RV,RV09}, building upon earlier works on the
Littlewood--Offord theory (see, in particular, \cite{TV ann math}),
the concept of the {\it least common denominator} (LCD) of a vector was introduced,
and the Euclidean sphere was partitioned into subsets according to the magnitude of the LCD.
Further, in works \cite{Cook circular, BCZ, LLTTY first part, BasRud, BasRud circ} dealing with
adjacency matrices of sparse directed random graphs, the crucial structural property of a vector
was statistics of ``jumps'' in its non-increasing rearrangement, i.e., the magnitude of ratios of
the form $x_i^*/x_{Li}^*$,
where $i\leq n$ and $L>0$ is a scaling factor.
In works \cite{BasRud, BasRud circ}, this analysis of jumps in the rearrangement
was combined with the LCD-based approach of \cite{RV, RV09}.

Despite the progress in this research direction in the past years, an efficient estimate of the smallest singular value
and, more generally, of quantities of the form $\inf\limits_{x\in T}\|Bx\|_2$ for a very sparse random matrix $B$
(with a constant average number of non-zero elements in a row/column) seems to require essentially new
arguments. In this work, we propose such a new argument for $0/1$ random matrices with prescribed row/column sums.
We believe that our approach can be extended to more general sparse models.

A crucial new ingredient of our paper is the concept of the {\it $\ell$-decomposition} of a vector which is a partition
of $[n]$ into subsets encoding useful structural properties of a complex $n$-dimensional vector $x$.
We combine the $\ell$-decomposition with new tensorization arguments (which allow to pass from individual
small ball probability estimates for $\langle\row_i(M-z\,\idmat),x\rangle$ to the matrix-vector product $(M-z\,\idmat)^K x$)
and a discretization (covering) procedure to get a characterization of {\it gradual} vectors in the kernel.
On the other hand, the {\it steep} vectors are treated by a combination of covering arguments and
a procedure utilizing expansion properties of the graph, thus augmenting the approach in \cite{LLTTY first part}.
In the remainder of the introduction, we will discuss in detail the three main features of the proof: steep and gradual vectors,
$\ell$-decomposition, and tensorization.

\medskip

\noindent {\bf Steep, almost constant, and gradual vectors.}
The notions of steep, almost constant, and gradual vectors appeared in \cite{LLTTY first part}
in the context of bounding the smallest singular value of the shifted adjacency matrix.
Naturally, these notions play an important role in the present paper as well.
For technical reasons, we slightly modified the definitions, compared to \cite{LLTTY first part}.

A full description of the class of steep vectors in $\C^n$ is provided in Subsection~\ref{subs: steep vectors}.
Since the precise formulas are long and involve many parameters, we omit them in the introduction. We just
 loosely describe this class as the collection of vectors $x\in\C^n$ such that
for some indices $1\leq i<j\ll n$, the ratio $x_i^*/x_j^*$ is {\it very large} compared to the ratio $j/i$.
The basic example of a steep vector is $(1,1,\dots,1,0,0,\dots,0)$, with less than $cn$ ones (for a small constant $c>0$).
At the same time, the steep vectors are not necessarily close to sparse in the Euclidean metric -- in particular,
the steep vectors cannot be identified with {\it compressible} vectors introduced in \cite{LPRT, RV, RV09}.

The second class of vectors which we call {\it almost constant}, is much easier to describe explicitly --
those are all the vectors $x\in \C^n$ such that
$$
\exists \lambda \in \C \quad \quad
   |\{ i\leq n \, : \, |x_i - \lambda|\leq \theta\, x_{cn}^{*}\}| >n- cn,
$$
where $\theta$ is a negative constant power of the degree $d$
and $c>0$ is a small universal constant.
Vectors which are neither steep  nor almost constant are called {\it gradual}.
Thus, a gradual vector has many pairs of distinct coordinates and controlled ratios $x^*_i/x^*_j$
for all indices $1\leq i<j\ll n$.

In our study of the kernel of the random operator $(M-z\,\idmat)^K$, we consider its intersection
with the classes of steep and almost constant vectors in Section~\ref{steep} and with gradual vectors
-- in Sections~\ref{s:k-vektors}--\ref{s:proof}.
Theorem~\ref{ker th} combines the information about the intersections.

For large enough $K^c$ and small $d$,
simultaneous existence of very steep and gradual (with many levels) vectors in the kernel of the matrix $M^K$ is an objective fact.
We can consider the following informal argument. For an integer $p\geq 1$,
the kernel of $M^{\{p+1,\dots,n\}}$
contains $\ker M^{\{2,\dots,n\}}$, which,
in view of Theorem~\ref{ker th} and the above remark,
typically consists of gradual vectors with many levels.
At the same time, the columns of $M$, $\col_i(M)$, are ``locally'' almost independent,
in the sense that for every small subset $Q\subset[n]$, the joint distribution of $\col_i(M)$, $i\in Q$,
is ``close'' to the joint distribution of independent vectors uniform on the set $\{y\in\{0,1\}^n:\;|\supp y|=d\}$
(in order not to expand the paper we prefer not to discuss quantitative aspects of this observation).
Thus, for fixed integers $p\gg d$ and  $r\ll n$,
we have
\begin{align*}
\Prob\big\{&\exists i\leq r:\,\supp\, \col_i(M)\cap\{p+1,\dots,n\}=\emptyset
\big\}\\
&\approx1-\prod_{i=1}^r \Prob\big\{\supp\,\col_i(M)\cap\{p+1,\dots,n\}\neq\emptyset\big\}\\
&=1-\bigg(1-{p\choose d}{n\choose d}^{-1}\bigg)^r\approx 1-\exp\big(-(p/n)^d r\big).
\end{align*}
Accordingly, when $r\gg (n/p)^d$, with large probability there exists a null-column in $M^{\{p+1,\dots,n\}}$, so that a typical realization of $M^{\{p+1,\dots,n\}}$
contains a coordinate vector in its kernel, which is ``very steep.''

The analysis of the null steep and almost constant vectors, which occupies Section~\ref{steep},
is a development of an argument from \cite{LLTTY first part} with
some important technical additions.
It combines deterministic estimates for $Mx$ (assuming certain expansion properties of the underlying graph)
with covering arguments for non-constant vectors with large support.
For almost constant vectors $x$,
a satisfactory bound for the L\'evy concentration function of the product $Mx$ is generally impossible.
For example, if $x=(1,1,\dots,1)$ then $Mx=(d,d,\dots,d)$ {\it deterministically}.
A key step in bounding the Euclidean norm of $(M-z\,\idmat)x$ (or, more generally, $(M-z\,\idmat)^K x$) from below
for almost constant vectors is a non-probabilistic argument which utilizes $d$-regularity (see Lemma~\ref{lem: a-c}).
A related method is used for steep vectors if the jump occurs at the beginning of the non-increasing rearrangement
(Lemma~\ref{l:T0}).
For the remaining vectors $x$, anti-concentration estimates for individual row-vector products are tensorized
to obtain individual estimates for $Mx$, which are combined with a covering argument, using specially constructed nets
(see Subsection~\ref{subs: nets}).

\medskip

\noindent {\bf The least common denominator (LCD), and the $\ell$-decomposition.}
The correspondence between small ball probability and arithmetic
properties of the vector of coefficients goes back to Hal\'asz \cite{Hal, Hal2}, and was
refined in the work \cite{TV bernoulli} on the singularity probability, before being
developed for the least singular value problem in \cite{R-ann, TV ann math}. Then
in \cite{RV,RV09}, the LCD with parameters $\gamma, \alpha>0$ of a vector $x\in\R^n$
was defined as
$$
 LCD(x):=\inf\big\{\theta>0,\;\d(\theta x,\Z^n)<\min(\gamma\theta\|x\|_2,\alpha)\big\}.
$$
It was used there with $\gamma$ being a small positive constant and
$\alpha$ being a small constant multiple of $\sqrt{n}$.
Thus, LCD encapsulates information on the distance of a rescaled vector to the integer lattice.
The fundamental correspondence between the magnitude of the LCD of a vector $x$ and the small ball probability
for the random sum $\sum_{i=1}^n x_i\xi_i$ (for independent sufficiently ``spread'' random variables $\xi_i$)
was the key ingredient employed in \cite{RV,RV09}.
Specifically, it was shown in \cite{RV} that a normal vector to the hyperplane spanned by $n-1$ columns typically has
LCD which is exponentially large in dimension. This fact was then applied to estimate the small ball probability
for the least singular value of the random matrix. However,
in the sparse regime this approach presents certain challenges.
Note that the definition of the LCD does not allow to distinguish a vector having a (small) proportion of coordinates of the
same value  from a vector with almost all components distinct. For example, the vectors
$x^1:=(1/n,2/n,3/n,\dots,n/n)$ and $x^2:=(0,0,\dots,0,0, 0.01+1/n, 0.01+2/n,\dots,0.01+0.99n/n)$ (with the first $0.01n$ components equal to zero)
would have comparable LCDs, whereas behaviour of the corresponding scalar products with a row of a random $d$-regular
matrix is completely different. Specifically, if $\xi=(\xi_1,\xi_2,\dots,\xi_n)$ is the random $0/1$
vector uniformly distributed on
$$
 \{y\in\{0,1\}^n:\;|\supp\,y|=d\}
$$
then for any number $r\in\R$ the scalar product $\langle \xi,x^1\rangle$ is equal to $r$ with probability
$O(1/n)$ while $\langle \xi,x^2\rangle=0$ with probability {\it at least} $0.01^d$.
Indeed, it is the absence (or presence) of large constant blocks of coordinates which becomes crucial
in the sparse regime.

The notion of the $\ell$-decomposition developed in this paper is in  particular designed to deal with the above issue.
More importantly, the $\ell$-decomposition carries
very detailed information about the structure
of a vector.  This information turns out extremely useful
in estimating cardinalities of coverings as well as in bounding small ball probabilities.

The precise definition of the $\ell$-decomposition is rather involved and includes an iterative procedure
for constructing a partition of $[n]$ associated to a vector.
We skip the technical details in the introduction (see Subsection~\ref{subs: ell decomposition} for the actual construction procedure) and describe the $\ell$-decomposition of a vector $y$ in the lattice $\frac{1}{k}\big(\Z^2\big)^n$
(where $k$ is a large integer) as a partition of $[n]$ into non-empty subsets $(\lset^{(q)})_{q=1}^m$ (called {\it $\ell$-parts})
which satisfy, among some other conditions, the following:
For any $q\leq m$ and any numbers $a,b\in \frac{1}{k}\Z^2$, setting
$$
 J_{a,q}:=\{i\in \lset^{(q)}:\;y_i=a\} \quad \mbox{ and } \quad  J_{b,q}:=\{i\in \lset^{(q)}:\;y_i=b\},
$$
one has that either one of the sets is empty or $|J_{a,q}|\leq 4|J_{b,q}|$.
In other words, considering the levels (blocks of coordinates having the same value) of the vector $y$ restricted to $\lset^{(q)}$,
all the blocks have approximately the same cardinality. If the vector $y$ is real and non-increasing,
its restriction to $\lset^{(q)}$ can be viewed as a ladder or staircase, where all the stairs are of about the same size,
whereas the gaps between the stairs are allowed to differ significantly.
The number of ``stairs'' within the $\ell$-part is called {\it the height} of the $\ell$-part. Note that the height
is {\it not} determined by the magnitude of $y_i$, $i\in \lset^{(q)}$, but instead by the number of levels in $\lset^{(q)}$.
The position of the $\ell$-parts and their heights provide  essential information on anti-concentration properties of the vector,
specifically, on anti-concentration of the scalar product with a row of our random matrix.
The information contained in the $\ell$-decomposition  allows
us to compute {\it conditional} small ball probability
with imposed restrictions on the distribution of the row. This becomes useful when studying anti-concentration
for the matrix-vector product $M^Ky$ (see Section~\ref{s:small ball}).
Observe that the $\ell$-decomposition is defined for a discrete subset of $\C^n$;
in fact, given a gradual vector $x$ we construct its approximations by vectors in $\frac{1}{k}\big(\Z^2\big)^n$
for various values of $k$ ({\it{}$k$-approximations}) and consider the $\ell$-decomposition of each of the approximations.

At a high level, the way we apply the $\ell$-decomposition to the original problem is similar to the way
the least common denominator was used in \cite{RV}. In \cite{RV}, the set of unit incompressible vectors
was split into subsets according to the magnitude of the LCD. Then, with the help of a covering argument
combined with small ball probability estimates, the subsets of vectors with small (subexponential in dimension)
LCD were excluded from the set of potentially null vectors of the matrix, leaving only those with very large LCD.

In our setting, we partition the collection of gradual vectors according to properties of the $\ell$-decompositions
of their lattice approximations. Using a combination of small ball probability estimates and coverings,
we exclude those gradual vectors with ``not that many'' levels, leaving only those satisfying the first condition
in Theorem~\ref{ker th}. The partitioning scheme
is rather complicated because of the ``multidimensional'' nature of the $\ell$-decomposition,
i.e., due to the necessity to take into consideration a set of parameters rather than a single number.
The crucial notions used in the partitioning are those of {\it regular} and {\it spread} $\ell$-parts.
Given a vector $y\in\frac{1}{k}\big(\Z^2\big)^n$ with the $\ell$-decomposition $(\lset^{(q)})_{q=1}^m$,
we say that the $\ell$-part $\lset^{(q)}$ is {\it spread} if it contains at least two distinct levels and
for every pair of coordinates $y_i,y_j$ ($i,j\in \lset^{(q)}$, $y_i\neq y_j$), we have $|y_i-y_j|\geq d/k$.
In the case of a vector with ordered real components, we may say that
the gaps between the ``stairs'' within the spread $\ell$-parts are $d$ times larger than their absolute minimum $1/k$.
The $\ell$-parts which are not spread are called {\it regular}.
Naturally, the matrix-vector products $M^K y$, with $y$ having spread $\ell$-parts of large cardinality,
enjoy relatively better anti-concentration properties.
At a more detailed level, the following procedure is applied:

\begin{itemize}

\item We isolate the set of gradual vectors, whose $k$-approximation (for some $k$ within a specific range)
has large spread $\ell$-parts. The sets of vectors are denoted by $\kapset_u$ in the text
(see Subsection~\ref{s:decomposition of S}), where $u\geq 5$ is related to $k$ by $k=d^u$.

\item For every vector from $\bigcup\kapset_u$, its $k$-approximation $y$ has the structure which guarantees
strong bounds for the small ball probability
$$
 \Prob\{\|(M-z\,\idmat)^K\,y\|_2\leq t\}.
$$
 These bounds,
combined with a covering procedure for the vectors in $\kapset_u$, implies that
with probability close to one no vector in $\bigcup\kapset_u$ is approximately a null vector.

\item Thus, it remains to deal with vectors in the complement $\mathcal S\setminus \bigcup\kapset_u$,
where $\mathcal S$ denotes the set of gradual vectors in $\C^n$
(we emphasize again that the union is taken over $u$ within a specific range determined by $n$, $d$,
and the cardinality of the set $K$).
We show (see Subsection~\ref{s:decomposition of S}) that the complement $\mathcal S\setminus \bigcup\kapset_u$
consists of gradual vectors, whose $k$-approximation (for a specially chosen $k$) has $\ell$-parts of very large heights
(see the definition of sets $\rhoset_v$ and Propositon~\ref{kappa and rho}).
At an elementary level, the coordinates of those $k$-approximations $y$ take many distinct values.
Naturally, this property provides a fine small ball probability estimate for $(M-z\,\idmat)^K\,y$,
however, for a different reason than in the case of large spread $\ell$-parts.

\item As the final step, we make the following observation: the set of gradual vectors from
$\rhoset_v$, which have relatively large constant blocks of coordinates, has much smaller
complexity than the entire $\rhoset_v$. In a sense, it is possible to construct a net on the
set of such vectors with cardinality balanced by individual small ball probabilities, thus
excluding the set from the collection of potentially null vectors. The remaining set --
gradual vectors without large constant blocks -- are the ``gradual vectors with many levels''
from the first assertion of Theorem~\ref{ker th}.

\end{itemize}

The procedure described above occupies Section~\ref{s:k-vektors} and a considerable part of
Section~\ref{s:proof}. Passing from small ball probability bounds for scalar products with
individual matrix rows to the entire matrix-vector product is a crucial step, with technical
complexities arising because of the lack of independence between the rows. This {\it tensorization}
part of the argument
spans Section~\ref{s:small ball} and partially continues into Section~\ref{s:proof}.
We would like to describe it in  more detail.

\medskip

\noindent {\bf Tensorization.} Our goal is to represent the small ball probability
$$
 \Prob\{\|(M-z\,\idmat)^K y\|_2\leq t\}
$$
for a given $k$-vector $y$ (i.e., $y\in \frac{1}{k}\big(\Z^2\big)^n$) and $t>0$ in terms of the structure of $y$, i.e, in terms of the
$\ell$-decomposition with respect to $y$.
A bound is obtained via a series of reduction steps. At each step, we replace our random model or
quantities  with objects that are simpler to analyze.

As the first step, given a vector $y$ with its $\ell$-decomposition $(\lset^{(q)})_{q=1}^m$,
we condition on realizations of $\sum_{j\in \lset^{(q)}}\col_j(M)$, $q\leq m$ (recall that $\col_i(M)$
denote the $i$-th column of $M$).
Specifically, we define a collection of $n\times m$ matrices $\RR$ with integer valued entries
and study the small ball probability within the event
$$
 \Event_\RR=\big\{\sum_{j\in \lset^{(q)}}\col_j(M)=\col_q(\RR),\;\;q\leq m\big\}
$$
for some fixed matrix $\RR$.
In particular, this forces the distributions of the columns of $M$ from different $\ell$-parts to be independent.
Moreover, using the expansion properties of the underlying graph, we impose additional assumptions on $\RR$ thus getting statistics of
the number of non-zero entries of the matrix $M$ ``restricted'' to each $\ell$-part.
The conditioning on $\Event_\RR$ is described at the beginning of Section~\ref{s:small ball}, however,
 additional structural assumptions on matrices $\RR$ are introduced later in Section~\ref{s:proof}.

The next -- crucial -- step consists in replacing the given random model with another one having independent components.
More precisely, we define a random $n$-dimensional vector $Z$ with {\it jointly independent} components
such that, conditioned on a certain event, its distribution matches the conditional distribution of $M y$ given $\Event_\RR$.
In its essence, every component $Z_i$ is a sum of independent random variables, where each variable
indicates the level of $y$ ``hit'' by a non-zero entry in the $i$-th row of $M$.
We connect the (conditional) distribution of $Z$ with the distribution of $My$ by introducing an intermediate
random model involving bipartite multigraphs. We start by showing that the distribution of the adjacency matrix
of that multigraph, conditioned on the event that the graph is simple, coincides with the distribution of $M$.
Results of this type are known in the random graph literature. In our setting we apply a result from
\cite{McKay-simple}
to get the correspondence. In its turn, the distribution of the adjacency matrix of the multigraph
is directly related to the distribution of $Z$ conditioned on a certain event that can be viewed as a
sort of ``$d$-regularity'' property.

Anti-concentration estimates for $My$ thus can be obtained by multiplying bounds for individual components $Z_i$.
Those, in turn, can be written in terms of the $\ell$-decom\-position of $y$ and the structure of the $i$-th row of the matrix $\RR$.
The functions which encapsulate this information are called {\it the small ball probability estimators}.
As the main result of Section~\ref{s:small ball}, we estimate $\Prob\{\|(M-z\,\idmat)^K y\|_2\leq t\}$
in terms of the product $\prod_{i=1}^n \est_i$, where $\est_i$ is the small ball probability estimator for the $i$-th row/$i$-th component
(see Theorem~\ref{sbp th}).
Computing  the product $\prod_{i=1}^n \est_i$ is not straightforward
as it involves analysis of both the $\ell$-decomposition of $y$ and the matrix $\RR$.
In the first part of Section~\ref{s:proof}, we introduce other  {\it estimators}
related to $\est_i$, which are easier to study.
Once we obtain an explicit upper bound for the small ball probability, we combine it with covering arguments,
which were briefly mentioned above.


The arguments in the paper are largely self-contained, although we employ several external results.
This includes (mostly standard) bounds on concentration and anti-concentration of the sum of independent variables;
certain expansion properties of the underlying random $d$-regular digraph; an upper bound on the
second largest singular value of the adjacency matrix; some estimates regarding the configuration
model for random graphs with predefined degree sequences.

\section{Preliminaries}
\label{preliminaries}

By {\it universal} or {\it absolute} constants we always mean numbers independent of all involved parameters,
in particular independent of $d$ and $n$. Given positive integers $\ell<k$ we denote sets
$\{1, 2, \ldots , \ell\}$ and $\{\ell, \ell + 1,  \ldots , k\}$ by $[\ell]$ and $[\ell, k]$ correspondingly.
Having two functions $f$ and $g$ we write $f\approx g$ if there are two absolute positive constants
$c$ and $C$ such that $cf\le g\le Cf$.
Given $z\in \C$, we denote by $\Re\, z$ (resp., $\Im\, z$) the real (resp., imaginary) part of $z$.
We define a lexicographical order on $\C$ in the following way.
Given $x,y\in \C$, we have $x\geq y$ if either $\Re\, x> \Re\, y$ or $\Re\, x= \Re\, y$ and $\Im\, x\geq \Im\, y$.
The lexicographical ordering will be useful when defining maximum over a finite
subset of the complex plane.

By $\idmat$ we denote the identity $n\times  n$ matrix, we use $\bf 1$ for a vector of ones.
For $I\subset [n]$, by $\cproj_I$ we denote the orthogonal projection on the coordinate subspace $\R^I$ (or $\C^I$),
and denote the complement of $I$ inside $[n]$ by $I^c$ ($n$ is always clear from the context).
Given a vector $x=(x_i)_{i=1}^n\in \C^n$, we denote  $x^\dagger=\bar x =(\bar x_i)_{i= 1}^n$, where $\bar z$ is the complex conjugate of $z\in \C$,
and by $(x_i^*)_{i=1}^n$ we denote the non-increasing rearrangement of the sequence $(|x_i|)_{i=1}^n$.
We use $\la\cdot, \cdot \ra$ for the standard inner product on $\C^n$, that is
$\la x, y \ra = \sum _{i= 1}^{n} x_i \bar y_i$.
Further, we write $\|x\|_{\infty}=\max_i |x_i|$ for the $\ell_{\infty}$-norm of $x$.
For a list of notation related to matrices and graphs we refer to the beginning of Section~\ref{s:graphs}.

Working with classes of vectors, we often consider the Minkowski sum for two subsets $V$ and $W$
of $\C^n$, which is defined as
$$
  V+W = \{v+w \, : \, v\in V, \, w\in W\}.
$$

To obtain probability bounds we often consider certain relations between sets and
use the following simple claims to estimate their probabilities.
Let $A$, $B$ be sets, and $R\subset A\times B$ be a relation.
Given $a\in A$ and $b\in B$, the image of $a$ and preimage of $b$ are defined by
$$
     R(a) = \{ y \in B \, : \, (a,y)  \in  R\} \quad \mbox{ and } \quad
     R^{-1}(b) = \{ x \in A \, : \, (x,b)  \in  R\}.
$$
We also set $R(A)=\cup _{a\in A} R(a)$.
We use the following standard estimate (see e.g. Claim~2.1 in \cite{LLTTY:15}).

\begin{claim} \label{multi-al}
Let $s, t >0$.
Let $R$ be a relation between two finite sets $A$ and $B$ such that for
every $a\in A$ and every $b\in B$ one has $|R(a)|\geq s$ and $|R^{-1}(b)|\leq t$.
Then
$
        s |A|\leq  t |B|.
$
\end{claim}

\subsection{Anti-concentration}

For a random vector $X$ distributed over a real or complex inner product space $E$,
its {\it L\'evy concentration function} $\cf(X,t)$ is defined as
$$
\cf(X,t):=\sup\limits_{\lambda\in E}\Prob\bigl\{\|X-\lambda\|_2\leq t\bigr\},\;\;t>0.
$$
In particular, if $X$ is a complex random variable then
$$
 \cf(X,t)=\sup\limits_{\lambda\in\C}\Prob\bigl\{|X-\lambda|\leq t\bigr\}.
$$
Dealing with the L\'evy concentration function of a complex random variable, we often
identify $\C$ with $\R^2$, which allows us  to apply
Propositions~\ref{prop: esseen} and \ref{prop: complex-kesten} formulated below for $E=\R^2$.

Upper bounds on the concentration function of a sum of independent random variables
is a standard subject, with many results available in the literature.
In our setting, we primarily deal with complex-valued
random variables, which in some situations require more delicate arguments.
In this subsection, we combine classical estimates of the concentration function
with some (not complicated) computations for vector-valued variables.

We will  use  (a particular version of) a theorem of Esseen
\cite{Esseen}
for sums of random vectors (Corollary 1 of Theorem~6.1 in \cite{Esseen} applied with $\rho _i=t_0$, $\rho=t$).

\begin{prop}[Esseen]\label{prop: esseen}
Let $m\geq 1$ and $\xi_1,\dots,\xi_m$ be independent random vectors in $\R^2$.
Then for any $t\geq t_0>0$ one has
$$
  \cf\Big(\sum_{i=1}^m \xi_i,t\Big)\leq
  \frac{C_{\ref{prop: esseen}}(t/t_0)^2}{\sqrt{m-\sum_{i=1}^m\cf(\xi_i,t_0)}},
$$
where $C_{\ref{prop: esseen}}>0$ is a universal constant. In particular, if
$\alpha \geq \max_{i\leq m} \cf(\xi_i,t_0)$ then
$$
  \cf\Big(\sum_{i=1}^m \xi_i,t_0\Big)\leq
  \frac{C_{\ref{prop: esseen}}}{\sqrt{m(1-\alpha)}}.
$$
\end{prop}

We will also need  a result of Miroshnikov \cite{miroshnikov}, which extends
 estimates on the concentration function due to Kesten \cite{kesten}
to the multi-dimensional setting. We state below the two dimensional version of
the  Corollary following Theorem~1 in \cite{miroshnikov} (note that the letter $E$
in that paper stands for the two-dimensional cube $B^2_\infty$, while we deal with the unit disc
$B_2^2$, so that  $B_2^2\subset E=B^2_\infty\subset \sqrt{2} B_2^2$).

\begin{prop}[Miroshnikov]\label{prop: complex-kesten}
Let $m\geq 1$ and  $\xi_1,\ldots, \xi_m$ be independent random vectors in $\R^2$.
Let $t_0>0$ be such that $\max_{i\leq m} \cf(\xi_i,t_0)\leq 1/2$. Then for any $t\geq t_0$ one has
$$
\cf\Big(\sum_{i=1}^m \xi_i,t\Big)\leq  \frac{C\, t }{t_0 \sqrt{m}}\, \max_{i\leq m}
\cf(\xi_i, \sqrt{2}\, t ),
$$
where $C$ is a positive universal constant.
\end{prop}

In general, the factor $1/\sqrt m$ in the above estimates is the best possible and is attained
for example on $\xi_i$'s with $\Re \, \xi_i=\Im \, \xi_i$ being Bernoulli random variables.
But if for example for every $i\le m$, $\Re \, \xi_i$ and $\Im\, \xi_i$ are independent
Bernoulli random variables,  or if $\xi_i$ is uniformly distributed
over the unit square, then Theorem~2 from  \cite{kesten} implies  respectively bounds  $C/m$
and $Ct^2/m$ for all $t\in (0, 1/2]$.
Some other cases when the factor $1/\sqrt m$ can be improved to $1/m$ were considered in
\cite{Esseen-type} (for distributions satisfying a certain symmetry condition) and in
\cite{Esseen} (for, in a sense,  well spread distributions).
We will need the following statement which is known to specialists. We provide its proof
for the sake of completeness at the end of this section.

\begin{prop}\label{density}
Let $m\geq 1$ and  $\xi_1,\ldots, \xi_m$ be independent random vectors in $\R^2$
with densities bounded by $1$. Then the density of $\xi_1 + \ldots + \xi_m$ is
bounded by $C/m$, where $C$ is a universal constant.
\end{prop}

Our next proposition is another case where the factor $1/\sqrt{m}$ can be improved.

\begin{prop}
\label{p:complexLevy}
Let $u,\varepsilon>0$, $m\geq 1$,  and let $\xi_1,\xi_2,\dots,\xi_m$ be i.i.d. discrete complex random variables
taking values on an $\varepsilon$-separated (in the Euclidean metric)
subset of the complex plane satisfying $\sup_{a\in\C}\Prob\{\xi_j=a\}\leq u$.
Then for any  $t>0$ one has
$$
  \cf\Big(\sum\limits_{j=1}^m \xi_j,t\Big)\leq
  C\, \max\big(ut^2/(m\varepsilon^2),u\big),
$$
where $C$ is a positive absolute constant.
\end{prop}

\begin{proof}
As above, in this proof it will be convenient to identify $\C$ with $\R^2$ and to present
complex  random variables as  two dimensional real random vectors. We also denote the Euclidean
unit ball (centered at $0$) on $\R^2$ by $B$. Let $\eta_j$, $j\leq m$,  be i.i.d. random vectors in $\R^2$ uniformly
distributed on $(\eps/2) B$
and jointly independent with $\xi_j$'s.
Note that by standard concentration estimates (e.g., one can apply Hoeffding's inequality \cite{Hoef} for
the first and second coordinates), we have for a large enough absolute positive constant $C$,
$$
  q:=\Prob\Big\{\Big\vert\sum\limits_{j=1}^m \eta_j\Big\vert\leq C\sqrt{m}\varepsilon\Big\}\geq 1/2.
$$

Define smoothed i.i.d. random variables $\xi_j'$, $j\leq m$, by setting
$\xi_j':=\xi_j+\eta_j$.
Since interiors of the discs of radius $\varepsilon/2$ centered at atoms of $\xi_j$'s
are disjoint, we get that the  densities of $\xi_j'$,  $j\leq m$,  are uniformly
bounded by $4u/(\pi\varepsilon^2)$.
By independence of $\eta_j$ and $\xi_j$ we observe that for every $t>0$,
\begin{align*}
 \cf\Big(\sum\limits_{j=1}^m \xi_j,t\Big)&=
 \max_{a\in \C} \,\,  \frac{1}{q}\,\, \Prob\Big\{\Big\vert\sum\limits_{j=1}^m \xi_j-a\Big\vert\leq t\; \mbox{ and }
  \Big\vert\sum\limits_{j=1}^m \eta_j\Big\vert\leq C\sqrt{m}\varepsilon\Big\}
  \\ &
  \leq
  2 \cf\Big(\sum\limits_{j=1}^m \xi_j', t+C\sqrt{m}\varepsilon\Big)
 = 2 \cf\Big(\sum\limits_{j=1}^m \xi_j'',\tau\Big),
\end{align*}
where $\xi_j'':=(2/\eps) \xi_j' \sqrt{u/\pi}$ and  $\tau=(2/\eps)t \sqrt{u/\pi}+2C\sqrt{u m/\pi}$.
Note that the densities of $\xi_j''$, $j\leq m$, are uniformly bounded by $1$. Thus Proposition~\ref{density}
implies the desired result.
\end{proof}

To prove Proposition~\ref{density} we need the following lemma (which, in a sense, similar to
the  proof of Theorem~1 in \cite{Esseen-type}).
In this lemma $\la\cdot , \cdot \ra$ denotes the canonical inner product on $\R^2$.

\begin{lemma}\label{lem-conc-prop}
 Let $p$ be a probability density on $\R^2$ bounded by $1$. Let $f=\hat p$, that is
 $$
    f(x)=\hat p(x) = \int _{\R^2} \exp(-2\pi i\la x, y\ra ) p(y) dy.
 $$
 Then for every $q\geq 2$ one has $\int _{\R^2} |f(x)|^q dx \leq 47/q$.
\end{lemma}

\begin{proof}
 Denote $\tilde p(x)=p(-x)$, $P(x):=p*\tilde p$, where $*$ denotes the convolution,
  and $F(x)=|f(x)|^2$. Then,
  $$F= f\cdot \bar f= \hat P.$$
 Observe that the function $P$ satisfies $P(x) \leq 1$, $P(x)=P(-x)$ for every $x\in \R^2$,
 and $\int _{\R^2} P(x) dx = 1$. Therefore, for every $x\in \R^2$,
 $$
  F(x)= \int _{\R^2}  \cos (2\pi \la x, y\ra )P(y) dy =  \int _{\R^2}(1-  2\sin^2 (\pi \la x, y\ra )P(y) dy
    = 1- 2 \int _{\R^2} \sin^2 (\pi \la x, y\ra ) P(y) dy.
 $$
 Consider the sets
$$A_\delta:=\{x\in \R^2\, : \, F(x)\geq 1-\delta^2\}$$
for $\delta \in (0, 1/2]$.
 Note that for every integer $k\geq 1$ one has
$k^2\sin ^2 t \geq \sin^2 (kt)$.
Given $\delta \in (0, 1/2]$, let $k=\lfloor 1/(2\delta)\rfloor$.  Then
  for every $x\in A_\delta$ we have
 $$
  F(kx)\geq 1- 2 k^2 \int _{\R^2} \sin^2 (\pi \la x, y\ra ) P(y) dy =
  1 - k^2(1-F(x))\geq 1- (k\delta)^2\geq  3/4,
 $$
that is on the set $kA_\delta = \{kx \, : \, x\in A_\delta\}$ we have  $F\geq 3/4$.
On the other hand, by the Plancherel theorem we have
$$
  \int _{\R^2} F(x) dx = \int _{\R^2} |f(x)|^2 dx = \int _{\R^2}  p^2(x) dx\leq  \int _{\R^2}  p(x) dx =1.
$$
This implies $|k A_\delta|\leq 4/3$,  hence $|A_\delta|\leq 4/(3k^2)\leq 64\delta^2 /3$, in particular,
$|A_{1/2}|\leq 4/3$.  Finally we estimate
$$
 \int _{\R^2} |f(x)|^q dx = \int _{\R^2} (F(x))^{q/2} dx.
$$
Then for $q\geq 2$ we have
$$
  I_1:=\int _{A_{1/2}^c} (F(x))^{q/2} dx \leq (3/4)^{q/2-1} \, \int _{\R^2} F(x) dx \leq (3/4)^{q/2-1} \, \leq 3/q,
$$
and
\begin{align*}
  I_2&:=\int _{A_{1/2}} (F(x))^{q/2} dx =\int _0^1 (q/2) s^{q/2-1} |\{F\geq \max(s, 3/4)\}| ds
  \\ &=
  \int _0^{3/4} |\{F\geq  3/4 \}| d s^{q/2} +  \int _{3/4}^1 (q/2) s^{q/2-1} |\{F\geq s\}| ds
  \\ &=
  (3/4)^{q/2} |A_{1/2}| +  \int _{3/4}^1 (q/2) s^{q/2-1} |A_{\sqrt{1- s}}| \, ds
 \leq
  9/(16q) + (64/3)\, \int _{3/4}^1 (q/2) s^{q/2-1} (1-s) ds
\end{align*}
Using integration by parts, we have
$$
  \int _{3/4}^1 (q/2) s^{q/2-1} (1-s)\, ds\leq \int _{0}^1 (1-s)\, d s^{q/2} =
  \int _{0}^1  s^{q/2}\,  ds = \frac{2}{q+2}.
$$
Therefore,
$$
   I_2\leq 9/(16q) + 128/(3q) \leq 44/q.
$$
Since $ \int _{\R^2} |f(x)|^q dx =I_1 +I_2$, this completes the proof.
\end{proof}

\begin{proof}[Proof of Proposition~\ref{density}]
The case $m=1$ is trivial, so we assume $m\geq 2$. As in Lemma~\ref{lem-conc-prop}, set
$f_i=\hat p_i$, and denote the density of  $\xi_1,\xi_2,\dots,\xi_m$ by $p$. Then, applying
the H\"older inequality and Lemma~\ref{lem-conc-prop}, we obtain
$$
  p(x) = \Bigg| \int _{\R^2} \prod_{i=1}^m f_i(y) \exp(2\pi i\la x, y\ra )  dy\Bigg|
  \leq\int _{\R^2} \prod_{i=1}^m \big| f_i(y) \big| dy  \leq
   \prod_{i=1}^m \left(\int _{\R^2} \big| f_i(y) \big|^m  dy \right)^{1/m} \leq C/m.
$$
\end{proof}

\subsection{Concentration}

The next lemma is the tensorization argument. It is a variant of Lemma~2.2 in \cite{RV} and its proof follows the same lines.
We include it for the sake of completeness.
\begin{lemma}\label{lem-tensorization}
Let $\xi_1,\ldots,\xi_n$ be independent complex random variables and $\varepsilon_0, p_1,\ldots,p_n$ be non-negative
real numbers.
Assume that for every $i\leq n$ and every $\varepsilon\geq \varepsilon_0$ one has
$$
 \Prob\{ \vert \xi_i\vert\leq \varepsilon\}\leq \varepsilon^2 p_i.
$$
Then for every $\varepsilon\geq \varepsilon_0$ one has
$$
 \Prob\Big\{\sum_{i=1}^n \vert \xi_i\vert^2 \leq \varepsilon^2 n\Big\}\leq \left(C_{\ref{lem-tensorization}}\varepsilon\right)^{2n} \,
 \prod_{i=1}^n p_i,
$$
where $C_{\ref{lem-tensorization}}>0$ is a universal constant.
\end{lemma}
\begin{proof}
Let $\varepsilon \geq \varepsilon_0$.
Using the hypothesis of the lemma and the distribution integral formula, we have
\begin{align*}
\Exp \exp(-\vert \xi_i\vert^2/\varepsilon^2)&= \int_{0}^1\Prob\{\exp(-\vert \xi_i\vert^2/\varepsilon^2)>s\}\,ds
=\int_0^\infty 2ue^{-u^2} \Prob\{|\xi_i|< u\varepsilon\}\,du\nonumber\\
&\leq
p_i \varepsilon^2\int_{\varepsilon_0/\varepsilon}^\infty 2u^3e^{-u^2}\,du+
p_i \varepsilon_0^2\int_0^{\varepsilon_0/\varepsilon} 2ue^{-u^2}\,du
\leq C p_i\varepsilon^2.
\end{align*}
By Markov's inequality, we obtain
\begin{align*}
 \Prob\Big\{\sum_{i=1}^n\vert \xi_i\vert^2\leq\varepsilon^2 n\Big\}
&= \Prob\Big\{\exp\Big(-\frac{1}{\varepsilon^2}\sum_{i=1}^n\vert \xi_i\vert^2\Big)
\geq e^{-n}\Big\}\nonumber \\
&\leq e^n \Exp \exp\Big(-\frac{1}{\varepsilon^2}\sum_{i=1}^n\vert \xi_i\vert^2\Big)
=e^{n} \prod_{i=1}^n \Exp \exp(-\vert \xi_i\vert^2/\varepsilon^2),
\end{align*}
which implies the desired result.
\end{proof}

\smallskip

The following statement is a non-Hermitian counterpart of the spectral gap estimates
for undirected random $d$-regular graphs -- a characteristic of major importance in connection with the graph expansion properties.
We refer to \cite{Alon Milman, Bordenave, BFSU, CGJ, Dodziuk, DJPP, Friedman, FKS, HLW, Puder, TY}
for more information on random expanders. The following statement was first proved in \cite{BFSU} for
$d\leq C\sqrt{n}$ (which is enough in this paper), then it was extended to the range $d\leq Cn^{2/3}$ in
\cite{CGJ} and to $d\leq n/2$ in \cite{TY}.
\begin{theor}\label{norm lemma}
There exists a universal constant $C_{\ref{norm lemma}}>0$ such that for $1\leq d\leq \frac{n}{2}$ one has
$\Prob(\Event_{\ref{norm lemma}})\geq 1-1/n^{100}$, where
$$
\Event_{\ref{norm lemma}}:=\Bigl\{M\in \Mc:\,\Big\|M- \frac{d}{n} {\bf 1}{\bf 1}^t \Big\|\leq C_{\ref{norm lemma}}\sqrt{d}\Bigr\}.
$$
\end{theor}

In fact one can replace the term $1/n^{100}$ in the above probability bound
by any negative power of $n$ at the expense of increasing the constant $C_{\ref{norm lemma}}$.

\section{Edge count statistics of $d$-regular digraphs}
\label{s:graphs}

A combination of probabilistic arguments shows that the edge counting statistics of random $d$-regular digraphs
(i.e., the number of edges connecting subsets of vertices of given cardinalities) concentrate around their
average values. In this section, we collect some estimates of the number of edges connecting given subsets
of vertices (equivalently, the number of non-zero elements in a given submatrix) and of
the number of in- or out-neighbors of a given vertex subset. While some of the statements
are borrowed from earlier works, others are new. We would like to note that properties of this type
were considered in the random setting in \cite{Cook expansion, LLTTY:15,LLTTY first part}.

First we introduce some notation. Denote by $\D$ the set of directed $d$-regular graphs on $n$ vertices,
where we allow loops but no multiple edges. This way, there is a natural bijection between $\D$ and $\Mc$.
We endow $\D$ with the uniform probability measure also denoted by $\Prob$.
Given a graph $\nrangr\in\D$ with an edge set $E$ and a subset $I\subset [n]$ of its vertices,
define sets of out- and in-neighbors as
$$
\outnbr(I)=\bigl\{v\leq n:\,  \exists i\in I \, \, (i,v) \in E \bigr\} \,\,\, \mbox{ and } \,\,\,
\innbr(I)=\bigl\{v\leq n:\,  \exists i\in I \, \, (v,i) \in E\bigr\} .
$$
Similarly, we define the out-edges and the in-edges as
$$
\outedg(I)=\bigl\{e\in E:\, e=(i,j) \text{ for some } i\in I \text{ and } j\leq n \bigr\}
$$
and
$$
\inedg(I)=\bigl\{e\in E:\, e=(i,j) \text{ for some } i\leq n \text{ and } j\in I\bigr\}.
$$
If $I=\{i\}$ we use lighter notation $\outnbr(i)$, $\innbr(i)$, $\outedg(i)$, and $\inedg(i)$.
Given a graph $\nrangr=([n],E)$, for every $I,J\subset [n]$ the set of all edges departing from $I$ and landing in $J$ is denoted by
$$
 \edg(I,J):=\bigl\{ e\in E:\, e=(i,j) \text{ for some } i\in I \text{ and } j\in J\bigr\}.
$$

 Let $M=\{\mu_{ij}\}\in \Mc$ and let $R_i=R_i(M)$ be the $i$'s row of $M$, $i=1,...,n$.
For
every subset  $J\subset [n]$, let
\begin{align*}
& S_J:=\{i\le n\,:\, \supp\,R_i\cap J\neq\emptyset \}
\end{align*}
be the union of supports of columns indexed by $J$ (the matrix will be clear from the context).
Given an $n\times n$ matrix $M$ and a set $\KK\subset [n]$, we use notation $M^\KK$ for a $|K|\times n$ matrix obtained from $M$ by removing
rows $R_i(M)$ with indices $i\not\in\KK$.

\smallskip

We start with the statement which essentially says that given a typical $d$-regular digraph
and a set of vertices $J$, which is not too large,
the set of all in-neighbors of $J$ has cardinality close
to the largest possible, i.e., $d|J|$. To formulate the statement, given $k\leq n$ and
$\eps \in (0,1)$ we introduce the set
\begin{equation}\label{Oo}
\Omega_{k, \eps}:=\Big\{M\in \Mc \,:\,\forall J\subset[n] \, \, \text{ with }\, \,
|J|=k
 \, \, \text{ one has }\, \,|S_J|\geq
(1-\varepsilon) d k \Big\}.
\end{equation}
Clearly, if $k=1$ then $\Omega_{k, \eps}=\Mc$.
The following theorem
is essentially  Theorem~2.2 of \cite{LLTTY:15} (see also Theorem 3.1 there).
\begin{theor}
\label{graph th known}
Let $e^8<d\leq n$, $\varepsilon_0=  \sqrt{\ln d/d}$, and  $\varepsilon\in [\varepsilon_0,1)$.
Let $k\leq c_{\ref{graph th known}}\eps n/d$,
where $c_{\ref{graph th known}}\in (0, 1)$ is a sufficiently small absolute positive constant.
Then
\begin{equation*}
\Prob(\Omega_{k,\varepsilon})\geq 1-\exp\left(-\frac{\eps^2 dk}{8}
\ln \left(\frac{e\eps c_{\ref{graph th known}} n}{k d}\right)\right).
\end{equation*}
In particular,
\begin{equation*}
\Prob\Big(\bigcap\limits_{k=1}^{\lfloor c_{\ref{graph th known}}\eps n/d\rfloor}\Omega_{k,\varepsilon}\Big)
\geq 1-\left( C d/\eps n\right)^{\eps^2 d/8 },
\end{equation*}
where $C$ is an absolute positive constant.
\end{theor}

In this paper, we prove the following auxiliary theorem, which states that given a large set of columns (of a typical matrix from $\Mc$),
there are many rows having many ones in this set.

\begin{theor}\label{th-graph-new}
Let $d\leq n$ be large enough integers and let $\ko\geq d+24 e n/d.$
For every $k\geq \ko$, denote
$$
\alpha_k:= \frac{d(k-d)}{8 e n}-1 \quad \mbox{ and } \quad
\beta_k:= \max\Big(e n\, \exp{(-\alpha_k/2)}, \frac{4k\ln (en/k)}{\alpha_k}\Big).
$$
Let $\Event_{\ref{th-graph-new}}$ be the set of all $M\in\Mc$
such that for every $J\subset [n]$ with $\vert J\vert\geq \ko$ one has
$$
  \bigl|\bigl\{i\le n :\,|\supp\,\row_i(M)\cap J|<\alpha_{\vert J\vert}\bigr\}\bigr|
\leq \beta_{\vert J\vert}.
$$
Then
$$
 \Prob(\Event_{\ref{th-graph-new}})\ge 1-4e^{-\ko}.
$$
\end{theor}
Before passing to the proof of Theorem~\ref{th-graph-new} we mention an immediate
corollary which will be used in Section~\ref{s:proof}.

\begin{cor}\label{graph th to prove}
There exist positive absolute constants $C_{\ref{graph th to prove}},c_{\ref{graph th to prove}}$
such that the following holds. Let $C_{\ref{graph th to prove}}\le d\le c_{\ref{graph th to prove}}\sqrt{n}$ and let
$\Event_{\ref{graph th to prove}}$ be the set of all $M\in\Mc$
such that for every $J\subset [n]$ with $\vert J\vert\geq n/\sqrt{d}$ one has
$$
 \big|\big\{i\le n :\,|\supp\,\row_i(M)\cap J|<{c_{\ref{graph th to prove}}d|J|}/{n}\big\}\big|
\leq n/\sqrt{d}.
$$
Then
$$
 \Prob(\Event_{\ref{graph th to prove}})\ge 1-4\exp\big(-n/\sqrt{d}\big).
$$
\end{cor}

\begin{proof}[Proof of Corollary~\ref{graph th to prove}.]  Let $J\subset [n]$ be such that $k:=|J|\ge n/\sqrt{d}$.
By the conditions on $n$ and $d$ we have $k\geq 2d$ so that, with $\alpha_k,\beta_k$
defined as in Theorem~\ref{th-graph-new}, we have $\alpha_k\geq dk/(32en)\geq \sqrt{d}/(32e)$
and $\beta_k\leq C n\ln d/d$, for some positive constant $C$.
Adjusting the choice of the constant $c_{\ref{graph th to prove}}$ and
applying Theorem \ref{th-graph-new} with $\ell_0= \lceil n/\sqrt{d}\rceil\geq d+24 e n/d$ we obtain the result.
\end{proof}

In order to prove Theorem~\ref{th-graph-new}, we  need the lemma below
(it will be more convenient for us to formulate it in the graph language).
For every $S,J\subset [n]$ we introduce
$$
  \Gamma_S^J:=\{ \nrangr\in\D: \, \forall i\in S\text{ one has } \vert \edg(i, J)\vert < \alpha_{|J|}\},
$$
where $\alpha_k$'s were defined in Theorem~\ref{th-graph-new}.
When $S=[s]$ and $J=[k]$ (we postulate that $[0]=\emptyset$) we will denote the above set by $\Gamma_s^k$.
For every $\ell\leq d$, denote
$$
\Gamma_{s,\ell}^k:=\Gamma_{s-1}^k\cap \{ \nrangr\in\D: \vert \edg(s, [k])\vert =\ell \}.
$$
With these notations, we have $\Gamma_{0}^k=\D$ for every $k$.
Clearly
\begin{equation}\label{eq-decomposition-gamma}
\Gamma_s^k\subseteq\Gamma^k_{s-1}\quad \quad \text{ and }\quad \quad
\Gamma_s^k=\bigsqcup_{\ell=0}^{\lfloor \alpha_k\rfloor} \Gamma_{s,\ell}^k.
\end{equation}

\begin{lemma}\label{lem-prob-gamma-s-k}
Let $d,n,\ko$, and $\alpha_k$, $k\geq \ko$, be as in Theorem~\ref{th-graph-new}
and let $S,J\subset [n]$ be such that $|J|\ge \ko$.  Then
$$
 \Prob(\Gamma_S^J)\leq \exp(-\alpha_{|J|} |S|).
$$
\end{lemma}
\begin{proof}
Without loss of generality we may assume that $S=[s]$ and $J=[k]$ for some $s\geq 1$ and $k\geq \ko$.
Let $q$ be a parameter in the interval $\alpha_k<q\leq d$, which  will be chosen later.
We first compare the cardinalities of $\Gamma_{s,\ell}^k$ and $\Gamma_{s,\ell+1}^k$ for
every $\ell<q$. To this end we construct a relation
$R_\ell\subset \Gamma_{s,\ell}^k\times \Gamma_{s,\ell+1}^k$.
Let $\nrangr\in \Gamma_{s,\ell}^k$. For  $j>k$ denote
$$
 E_j:= \edg([s]^c,[k])\setminus \Big(\inedg\big(\outnbr(s)\cap[k]\big)\cup \outedg\big(\innbr(j)\big)\Big).
$$
Since $\nrangr\in \Gamma_{s,\ell}^k$,
$$
 \vert\edg([s]^c,[k])\vert = kd - \vert\edg([s], [k])\vert \geq kd-\alpha_k(s-1)-\ell .
$$
On the other hand, since $\vert\outnbr(s)\cap[k]\vert =\ell$, then
$\vert\edg\big([s]^c,\outnbr(s)\cap[k]\big)\vert\leq \ell(d-1)$.
Using that
$$
 \vert \edg\big(\innbr(j)\setminus[s], [k]\big)\vert \leq  d (d-1)
$$
we obtain
\begin{equation}\label{eq-size-Ej}
\vert E_j\vert \geq kd- \alpha_k(s-1)-\ell d-d (d-1) \geq (k-d+1)d- q (s+d-1).
\end{equation}

Now we are ready to define the relation $R_\ell$.
We let a pair $(\nrangr,\nrangr')$ belong to $R_\ell$ for some $\nrangr'\in\Gamma_{s,\ell+1}^k$ if $\nrangr'$
can be obtained from $\nrangr$ in the following way. Choose $j\in \outnbr(s)\cap [k]^c$ and an edge $(u,v)\in E_j$.
We destroy the edge $(s,j)$ and create the edge $(s,v)$, then we destroy the edge $(u,v)$ and create the edge $(u,j)$
(in other words, we perform the simple switching on the vertices $s,u,j,v$). Note that the conditions
$u\not\in\innbr(j)$ and $v\not\in\outnbr(s)$, which are implied by the definition of $E_j$, guarantee that the simple
switching does not create multiple edges, and we obtain a valid graph in $\Gamma_{s,\ell+1}^k$.
Using (\ref{eq-size-Ej}) and assuming
\begin{equation}\label{cond-al}
q \leq \frac{d(k-d)}{2(n+d)},
\end{equation}
 we deduce that for every $\nrangr\in\Gamma_{s,\ell}^k$ one has
\begin{equation}\label{eq-size-image}
\vert R_\ell(\nrangr)\vert \geq (d-\ell)\Big[ (k-d+1)d- q(s+d-1)\Big]\geq \frac{d(k-d)(d-q)}{2}.
\end{equation}

Now we estimate the cardinalities of preimages. Let $\nrangr'\in R_\ell(\Gamma_{s,\ell}^k)$. In order to reconstruct
a graph $\nrangr$ for which $(\nrangr,\nrangr')\in R_\ell$ we need to perform a simple switching which destroys an
edge in $\edgpr(s,[k])$ and adds an edge in $\edgpr(s,[k]^c)$. To this end, choose
$$
 v\in \outnbrpr(s)\cap [k] \quad \mbox{ and } \quad
 j\in [k]^c\setminus \outnbrpr(s).
$$
Since $\vert \innbrpr(j)\vert =d$, there are at most $d$
simple switchings which destroy the edge $(s,v)$ and create the edge $(s,j)$.
Using that $|\outnbrpr(s)\cap [k]|=\ell+1$, we observe
\begin{equation}\label{eq-size-preimage}
\vert R_\ell^{-1}(\nrangr')\vert \leq d(\ell+1) (n-k-(d-\ell-1))\leq dq(n-k).
\end{equation}

Using Claim~\ref{multi-al} together with inequalities
\eqref{eq-size-image} and \eqref{eq-size-preimage}, we obtain that for every $\ell<q$,
$$
\vert\Gamma_{s,\ell}^k\vert \leq \frac{2q(n-k)}{(k-d)(d-q)}\vert\Gamma_{s,\ell+1}^k\vert.
$$
Therefore
$$
\vert\Gamma_{s,\ell}^k\vert \leq \left(\frac{2q(n-k)}{(k-d)(d-q)}\right)^{q-\ell}\vert\Gamma_{s,q}^k\vert\leq
\left(\frac{2q(n-k)}{(k-d)(d-q)}\right)^{q-\ell}\vert\Gamma_{s-1}^k\vert.
$$
This together with (\ref{eq-decomposition-gamma}) implies that
$$
\vert \Gamma_s^k\vert\leq \sum_{\ell=0}^{\lfloor \alpha_k\rfloor}\left(\frac{2q(n-k)}{(k-d)(d-q)}\right)^{q-\ell}\vert\Gamma_{s-1}^k\vert
\leq  e^{\lfloor\alpha_k\rfloor +1-q}\, \vert\Gamma_{s-1}^k\vert ,
$$
provided that
$$
\frac{2q(n-k)}{(k-d)(d-q)}\leq e^{-1}.
$$
We choose $q=2\lfloor \alpha_k\rfloor+2$ which satisfies  the above condition and the condition
(\ref{cond-al})
by the definition of $\alpha_k$. Therefore
we have
$$
\vert \Gamma_s^k\vert\leq e^{-\alpha_k}\vert \Gamma_{s-1}^k\vert.
$$
Since we do not impose any restrictions on $s$, we conclude that
$$
\vert \Gamma_s^k\vert=\vert \D\vert \prod_{p=1}^s \frac{\vert \Gamma_p^k\vert}{\vert \Gamma_{p-1}^k\vert}
\leq e^{-s\alpha_k}\, \vert \D\vert.
$$
\end{proof}

\begin{proof}[Proof of Theorem~\ref{th-graph-new}.]
We start by defining
\begin{align*}
 \Gamma:= \Big\{\nrangr\in\D:\, \exists J\subset [n],\, |J|\geq \ko\, \,\, \text{with}\,\,
\bigl|\bigl\{i\le n :\,|\edg(i,J)|<\alpha_{\vert J\vert}\bigr\}\bigr|
> \beta_{\vert J\vert}\Big\}.
\end{align*}
It is not difficult to see that $\Gamma$ is the graph counterpart of event
$\Event_{\ref{th-graph-new}}$,
in particular $\Prob(\Event_{\ref{th-graph-new}})=\Prob(\Gamma)$.
Note that
$$
 \Gamma=\bigcup_{\underset{\vert J\vert\geq \ko}{J\subset [n]}}
 \bigcup_{\underset{\vert S\vert> \beta_{\vert J\vert} }{S\subset [n]}} \Gamma_{S}^J.
$$
Therefore, applying Lemma~\ref{lem-prob-gamma-s-k} and taking the union bound, we obtain
$$
\Prob(\Gamma)\leq \sum_{\underset{\vert J\vert\geq \ko}{J\subset [n]}}\sum_{\underset{\vert S\vert> \beta_{\vert J\vert} }{S\subset [n]}} \Prob(\Gamma_{S}^J)\leq \sum_{\underset{\vert J\vert\geq \ko}{J\subset [n]}}\sum_{\underset{\vert S\vert\geq \beta_{\vert J\vert} }{S\subset [n]}} e^{-\alpha_{\vert J\vert}\vert S\vert}
=\sum_{k\geq \ko}\sum_{s\geq \beta_k} {n\choose k} {n\choose s} e^{-\alpha_k s}.
$$
Further, by the choice of $\beta_k$, we have
$
{n\choose s} e^{-\alpha_k s}\leq (en/s)^s \, e^{-\alpha_k s}\leq e^{-\alpha_k s/2}
$
for all $s\geq \beta_k$ and since $\alpha_k\geq 2$, then
$$
\sum_{s\geq \beta_k} {n\choose s} e^{-\alpha_k s}\leq 2e^{-\alpha_k\beta_k/2}.
$$
By the choice of $\beta_k$'s, this implies
$$
\Prob(\Gamma)\leq 2\sum_{k\geq \ko}{n\choose k}e^{-\alpha_k\beta_k/2}\leq 2\sum_{k\geq \ko}\left(\frac{k}{en}\right)^k\leq 4e^{-\ko},
$$
which completes the proof.
\end{proof}

Combining Theorems~\ref{graph th known} and~\ref{th-graph-new}, we prove the following proposition.

\begin{prop}\label{graph prop}
There exist absolute positive  constants $C$ and $c_{\ref{graph prop}}$  such that the following holds.
Let $C\leq d\leq c_{\ref{graph prop}}\sqrt{n}/\ln n$ and let
$\Event_{\ref{graph prop}}$ be the set of all $M\in\Mc$
such that for all $J\subset [n]$ one has
\begin{align*}
 \big|\big\{i\leq n:\,&|\supp\,\row_i(M)\cap J|\geq {c_{\ref{graph prop}} d|J|}/{n}\, \, \, \mbox{ and }\, \, \,
\\
&|\supp\,\row_i(M)\cap J^c|\geq {c_{\ref{graph prop}} d|J^c|}/{n}\big\}\big|\geq c_{\ref{graph prop}} \min(d|J|,d|J^c|,  n).
\end{align*}
Then $$\Prob(\Event_{\ref{graph prop}})\geq 1-n^{-c_{\ref{graph prop}} d}.$$
\end{prop}

\begin{proof}
Let $c>0$ be small enough absolute constant and let $d\leq c\sqrt{n}$.
We first treat subsets $J$ satisfying
$|J|< 2d$. Recall that $S_J$ denotes the union of supports of columns indexed by $J$. Therefore,
using $d\leq c\sqrt{n}$ and applying
Theorem~\ref{graph th known}, we observe that there exists an absolute constant $c'>0$ such that
with probability at least $1-n^{-c' d}$ one has $|S_J|\ge 0.9d|J|$ for all $J$ with $|J|< 2d$.
Now fix $J\subset [n]$ satisfying both  $|J|< 2d$ and $|S_J|\ge 0.9d|J|$.
Since $d\leq c\sqrt{n}$, then $cd |J|/n\leq 2c^3<1$ for small enough $c$. Therefore, the
condition
$$
 |\supp\,\row_i(M)\cap J|\geq {c d|J|}/{n}
$$
means that $i\in S_{J}$.
Note also that in this case $\min(d|J|,d|J^c|,  n)= d|J|$. Thus, it is enough to show that
$$
 |\{i\in S_{J}:\, |\supp R_i(M)\cap J^c|\ge c d \}|\ge c d|J|.
$$
Let $\ell$ denote the number of rows  having at least $d-1$ ones in $J^c$. Counting ones
in the columns indexed by $J$ we have
$$
   d|J^c|= dn-d|J| \leq d \ell + (d-2) (n-\ell) = 2 \ell + dn - 2n.
$$
This implies $\ell \geq n - d|J|/2$. Therefore there are at least
$
    \ell  +|S_J|-n\geq 0.4 d|J|
$
rows indexed by $S_J$ and having at least $d-1$ ones in $J^c$.
This proves that the set of all matrices in $\Mc$ satisfying the condition of the proposition
for subsets $J\subset[n]$ with $|J|<2d$, has measure at least $1-n^{-c' d}$. Interchanging
the role of $J$ and $J^c$ we obtain the same bound for subsets $J$ satisfying $|J|> n - 2d$.

For the rest of the proof we deal only with sets $J$ satisfying $|J|\in [2d, n-2d]$
for which the quantities $d(\vert J\vert-d)/n$ and $d\vert J\vert/n$ are equivalent
up to a constant multiple (and similarly for $J^c$).
We will prove a more precise relation, which is convenient to formulate in the graph language.
Namely, denoting by
$\Event $  the set of all digraphs $\nrangr\in\D$ such that for every $J\subset [n]$ with
$|J|\in [2d, n-2d]$
\begin{align*}
   \Big\vert\Big\{&i\le n :\,
 |\edg(i,J)| \geq \frac{c_1 d(\vert J\vert-d)}{n}
 \, \, \mbox{ and } \, \,
 |\edg(i,J^c)| \geq \frac{c_1 d(\vert J^c \vert-d)}{n}
  \Big\}\Big\vert\\
&\geq c_1 \min(d\vert J\vert, d\vert J^c\vert, n),
\end{align*}
we show that the event $\Event$ has probability at least  $1-n^{-c_1 d}$,
where $c_1>0$ is a sufficiently small universal constant.

Given $\ko\geq  d+24en/d$ and $k\ge \ko$, we consider parameters $\alpha_k$ and $\beta _k$ introduced
in Theorem~\ref{th-graph-new}. Additionally, for $k<\ko$ we set $\alpha_k:=\alpha _{\ko}$ and let $\beta_{k}$
be defined by the same formula as $\beta _k$ for $k\geq \ko$.  Note that
$(\alpha_k)_k$ is a non-decreasing sequence, while  $(\beta_k)_k$ is a non-increasing sequence.
Now set $\ko= \lfloor d+Cn/d\rfloor$, where $C\geq 24e$ is a sufficiently large universal constant
chosen so that $\beta_{\ko}\leq n/4$ (then $\beta _i + \beta _j \leq n/2$ for every $i, j\geq \ko$).
Note that  $2d<\ko< n/2$. Define the event
\begin{align*}
\Event _1:=\Big\{&\nrangr\in\D:\, \forall J\subset [n], \vert J\vert \in [\ko, n-\ko]),\,\,\, \text{ one has }\\
&\big\vert\big\{i\le n :\,|\edg(i,J)|\geq \alpha_{\vert J\vert} \text{ and } |\edg(i,J^c)|\geq \alpha_{\vert J^c\vert} \big\}\big\vert
\geq n-\beta_{\vert J\vert}-\beta_{\vert J^c\vert}\Big\}
\end{align*}
and, for $m=2,3$, the events
\begin{align*}
&\Event _m:=\Big\{\nrangr\in\D:\, \forall J\subset [n], \vert J\vert \in S_m,  \,\,\,  \text{ one has }\\
&\Big\vert\Big\{i\le n :\,|\edg(i,J)|\geq \frac{\alpha_{\vert J\vert}}{\alpha_{\ko}} \text{ and } |\edg(i,J^c)|
\geq \frac{\alpha_{\vert J^c\vert}}{\alpha_{\ko}} \Big\}\Big\vert
\geq c_1\, \min(d\vert J\vert,  d\vert J^c\vert, n)\Big\}.
\end{align*}
where $S_2=[\ko-1]$, $S_3=[n-\ko+1, n]$.
Then we clearly have $\Prob(\Event _2)=\Prob(\Event _3)$ and, moreover,
$\Event _1\cap\Event _2 \cap \Event _3 \subset \Event$, provided that
$c_1$ is sufficiently small. Therefore,
\begin{equation}\label{eq-event-prop-graph}
 \Prob(\Event )\geq \Prob(\Event _1)+\Prob(\Event _2)+ \Prob(\Event _3)-2= \Prob(\Event _1)+ 2 \Prob(\Event _2)-2.
\end{equation}

First, we estimate probability of $\Event _1$.
For any set $J$ with $\vert J\vert \geq \ko$ and $\vert J^c\vert\geq \ko$, the condition
\begin{align*}
  \big\vert\big\{i\le n :\,|\edg(i,J)|\geq \alpha_{\vert J\vert}\big\}\big\vert&\geq n-\beta_{\vert J\vert}
  \text{  and  }
  \big\vert\big\{i\le n :\, |\edg(i,J^c)|\geq \alpha_{\vert J^c\vert} \big\}\big\vert
\geq n-\beta_{\vert J^c\vert}
\end{align*}
for a graph $\nrangr\in\D$ implies
$$
\big\vert\big\{i\le n :\,|\edg(i,J)|\geq \alpha_{\vert J\vert}\,\,\, \text{ and }
\,\,\, |\edg(i,J^c)|\geq \alpha_{\vert J^c\vert} \big\}\big\vert
\geq n-\beta_{\vert J\vert}-\beta_{\vert J^c\vert}.
$$
Therefore by Theorem~\ref{th-graph-new} we obtain
$\Prob(\Event _1)\geq
1-8e^{-\ko}.$

We now turn  to the probability of $\Event _2$. Note that for every $J$ with $\vert J\vert <\ko$, we have
$\alpha_{\vert J\vert}= \alpha_{\ko}$ and $\alpha_{\vert J^c\vert}\leq \alpha_n$. Further, for every $i\leq n$,
$$
 \vert\edg(i,J)\vert+\vert \edg(i,J^c)\vert=d.
$$
Therefore, for every graph $\nrangr\in\D$ we have
\begin{align*}
 \Big\{i\le n :\,|\edg(i,J)|\geq \frac{\alpha_{\vert J\vert}}{\alpha_{\ko}}
 \,\,\,  \text{ and }\,\,\,  &|\edg(i,J^c)|\geq \frac{\alpha_{\vert J^c\vert}}{\alpha_{\ko}} \Big\}
   \\
  & \supset   \Big\{i\le n :\,|\edg(i,J)|\geq 1
 \,\,\,  \text{ and }\,\,\,  |\edg(i,J^c)|\geq \frac{\alpha_{n}}{\alpha_{\ko}} \Big\}
  \\
 & = \innbr(J)\setminus \Big\{i\le n :\, \vert \edg(i,J)\vert > d-\frac{\alpha_n}{\alpha_{\ko}}\Big\}.
\end{align*}
Since $\alpha_{\ko}\geq 2$, $\alpha_{n}\leq d/(8e)$, we observe
$$
 \Big\vert\Big\{i\le n :\, \vert \edg(i,J)\vert > d-\frac{\alpha_n}{\alpha_{\ko}}\Big\}\Big\vert\leq
 \frac{d\vert J\vert}{d-\alpha_n/\alpha_{\ko}}\leq 2\vert J\vert.
$$
Therefore,
$$
  \Event _2\supset \Big\{\nrangr\in\D:\, \forall J\subset [n], \vert J\vert< \ko, \,
   \vert\innbr(J)\vert \geq c_1\min(d\vert J\vert,  n)+2\vert J\vert\Big\}
$$
 We apply Theorem~\ref{graph th known} with $\eps=0.1$. Recall that $c_1$ is small enough.
 Theorem~\ref{graph th known} implies that there exists a universal constant $c_1'>0$,
 such that with probability at least $1-n^{-c_1' d}$ for every $J$ with
$|J|\leq c_{\ref{graph th known}}n/(10 d)$ one has
$$
 \vert\innbr(J)\vert\geq 0.9 d\vert J\vert\geq c_1\min(d\vert J\vert,  n)+2\vert J\vert
$$
  and for every $J$ with $c_{\ref{graph th known}}n/(10 d) \leq |J|<\ko$, passing to
  a subset $J_0\subset J$ with $|J_0|= \lfloor c_{\ref{graph th known}}n/(10 d)\rfloor$
  and using $\ko \leq 2Cn/d$,
  one has
$$
 \vert\innbr(J)\vert\geq \vert\innbr(J_0)\vert\geq 0.9 d\vert J_0 \vert\geq 9 c_{\ref{graph th known}}
  d\vert J \vert/ (400C) \geq c_1\min(d\vert J\vert,  n)+2\vert J\vert.
$$
Thus $\Prob(\Event _2)\geq 1-n^{-c_1' d}$.
By \eqref{eq-event-prop-graph} this implies $\Prob(\Event ) \geq 1 - 8e^{\ko}-2 n^{-c_1' d}$,
which together  with the bounds obtained at the beginning of the proof implies the desired result.
\end{proof}

Finally, we need the following deterministic statement dealing with sets
of in-neighbours of two disjoint sets of vertices. Given two disjoint subsets $J^\ell$, $J^r\subset[n]$
and a matrix $M\in \Mc$, denote
$$
\il=\il(M):=\{i\le n :\,|\supp R_i\cap J^\ell|=1 \, \, \text{ and }\, \,\supp R_i\cap J^r=\emptyset\},
$$
and
$$
 \ir=\ir(M):=\{i\le n :\,\supp R_i\cap J^\ell=\emptyset\,\,\text{ and }\,\,|\supp R_i\cap J^r|=1\}.
$$
Here the upper indices $\ell$ and $r$ refer to {\it left} and {\it right}, since later for a given
vector $x\in \C^n$, denoting by $\sigma$ a permutation  of $[n]$ satisfying
$x_i^*=|x_{\sigma(i)}|$ for all $i\le n$, we will choose $J^\ell=\sigma([k_1])$ and
$J^r=\sigma([k_2, n])$ for some $k_1<k_2$.
The following statement is Lemma 2.7 from \cite{LLTTY first part}.

\begin{lemma}\label{c:SJ}
Let $d$ and $\varepsilon$ be as in Theorem~\ref{graph th known}. Let $p\ge 2$, $m\geq 1$ be
integers satisfying $pm \leq c_{\ref{graph th known}} \eps n/d$
and let  $J^\ell, J^r\subset[n]$ be such that
$J^\ell\cap J^r=\emptyset$, $|J^\ell|=m$, $|J^r|=(p-1)m$. Let $M\in \Omega_{pm, \eps}$.
Then
$$
   |\il|\ge(1-2\eps p)d|J^\ell|.
$$
In particular, if $|J^r|=|J^\ell|=m$ with $m\leq c_{\ref{graph th known}} \eps n/(2d)$ then
$$
   (1-4\eps) dm \leq \min(|\il|, \, |\ir|) \leq \max (|\il|, \, |\ir|) \leq dm .
$$
\end{lemma}

\section{Almost constant and steep vectors}
\label{steep}

As in \cite{LLTTY first part} we split $\C^n$ into three classes of vectors which we call
{\it steep}, {\it gradual}, and {\it almost constant vectors}.
This section is devoted to steep and almost constant vectors. The definition of steep vectors is similar to the one given in
\cite{LLTTY first part}, with slight modifications one of which is quite important.
Note that if a subset $K\subset [n]$ is such that $|K^c|$ is much larger than $n^{1-1/d}$, then
the submatrix $M^K$ will contain null columns with large probability, hence the kernel of $M^K$
will contain very sparse vectors.
Therefore, when studying the kernel of $M^K$ (or $(M-z\,\idmat)^K$),
very sparse vectors and those ``close'' to very sparse should be handled separately,
see the definition of $\st_3$ below.
The set $\st_3$ -- the set of {\it very steep} vectors -- can be viewed as an enlargement of the set of very sparse vectors;
in this sense, our
construction is related to the definition of compressible vectors,
$$
  \mbox{Comp} (m, \rho) := \{x\in \C^n  \,:\,  \exists\, \, m\mbox{-sparse vector}
  \,\, y\in \C^n \, \,  \mbox{ such that } \, \,  \|x-y\|_2 \leq \rho \| x \|_2\},
$$
 introduced in \cite{RV} following ideas from \cite{LPRT} (as usual, $m$-sparse
 means that a vector has at most $m$ non-zero coordinates).
Both classes, $\st_3$ and Comp,
are introduced as classes of vectors close to $m$-sparse vectors (for an appropriate $m$).  An important difference between the two
lies in how the distance to the set of sparse vectors is measured -- instead of the Euclidean
distance used for compressible vectors, we estimate the $\ell_\infty$-norm after some normalization related to
a variant of the weak $\ell_{1/3}$-norm.

After introducing and eliminating very steep vectors we consider other
vectors with a jump in their non-increasing rearrangement.
The goal is to show that such vectors are far from the kernel of $M^K$. We will distinguish two  types of jumps. The first one occurs at the beginning
of the non-increasing rearrangement and is of order $4d$, that is for certain $m<k$, we have
$x^*_m>4d x^*_k$, see the definition of $\st_0$ below.
To treat such vectors $x$ we use
Lemma~\ref{l:T0}, which yields that with large probability
the random matrix distributed in $\Mc$ has many rows with exactly one $1$ in coordinates
corresponding to $m$ largest components of $x$, and all zero coordinates in places corresponding
to $m+1$-st to $k$-th largest component of $x$.
Since the total numbers of ones in every row is $d$ we have that the inner product of every such row with
$x$ is separated from zero. Unfortunately, since this procedure relies on graph expansion properties given by Theorem~\ref{graph th known}, it works only when
$k$ is not too large, namely when  $k\leq c_{\ref{graph th known}} \eps n/d$. For larger values of $k$
we use a different technique, which requires considering a jump of order $d^{3/2}$, see the definitions of $\st_1$ and $\st_2$.
For vectors with such jumps,  using the switching technique,  we estimate the probability that a fixed vector is close to the kernel,  construct a special net of rather
small cardinality in the set of such vectors, then use the union bound.
This scheme is similar to the one used in
\cite{LLTTY first part} with an important difference in the class $\st_2$ -- we use a number of coordinates proportional to $n$ in the definition (meaning $\nn$ is proportional to $n$), while in \cite{LLTTY first part} we used only $n/\ln d$ coordinates. This makes
bounding the probability for individual vectors much more difficult and involved (the method of \cite{LLTTY first part}
does not work). The main novelty of our new method is splitting the set $\Mc$ into many equivalence classes and
working in each class separately.
The construction of nets, which is also rather delicate, comes from
 \cite{LLTTY first part}.

 Finally, we treat almost constant vectors, i.e., vectors having many coordinates which
   are almost equal to each other. Having excluded
 vectors with jumps it only remains to treat those without any.
 However, if a vector $x$ with no jumps has many almost equal coordinates
 (say on a set $J$), then its inner product with a row having many ones
 inside $J$  cannot be close to zero (and we show that there are many such rows).
 In view of this observation, it is important to find a balance between quantitative
 characteristics of the level of jumps and the places where these jumps occur.
 Due to technical reasons, such a balance cannot be achieved directly,
 in particular, to treat almost constant vectors, one needs to consider
 constant jumps (not a power of $d$). Fortunately, it turns out that every almost
 constant vector without a constant jump can be represented as the sum of a vector
 with a big jump and a constant vector, i.e., a vector whose coordinates are equal
 to each other. Moreover, our proof for vectors with big jumps is stable under shifts
 by constant vectors which makes our treatment of almost constant vectors a lot easier.

\smallskip

We now introduce the following parameters, which will be used throughout this section.
First fix $1\leq \LL\leq n/d^3$ (we always assume that $n\geq d^3$).
When considering the minor $M^K$, $L$ will be responsible for the size of the set $K^c$.
In order to use Theorem~\ref{graph th known}, we fix $\eps_0$ and a related parameter
$p$ as follows:
$$
\eps _0 = \sqrt{(\ln d)/{d}}, \quad \quad  p= \lfloor 1/(5 \eps_0) \rfloor=\left\lfloor \tfrac{1}{5} \sqrt{d/\ln d}\r\rfloor
$$
(the choice of $p$ comes from $\eps _0 p < 1$ needed in Lemma~\ref{l:T0} in order to apply  Lemma~\ref{c:SJ}).
 Furthermore, we  fix a sufficiently small positive  absolute constant $\aaa$
(we don't try to estimate the actual value of $\aaa$, the conditions on how small it is appear
in the corresponding proofs).
Set
$$
  n_1:=\lceil  n /d^{3/2} \rceil, \quad n_2 :=\lfloor  n/d^{2/3} \rfloor, \quad \mbox{ and } \quad \nn :=\lfloor \aaa n \rfloor.
$$
We also fix two positive integers $r$ and $r_0=r_0(\LL)$ such that $p^{r}<n_1\leq p^{r+1}$ and
$r_0$ is the smallest non-negative integer satisfying $p^{r_0}\geq 20\LL/d$.
Note that $0\leq r_0<r$. Indeed,
$$
 p^{r-1} \geq n_1/p^2 \geq 25 n \ln d/d^{5/2}>20 L /d,
$$
which implies that $r_0\leq r-1$.

Finally, denote the class of constant vectors by
$$
 \CC : = \{ x=(x_i)_{i=1}^n\in \C^n \, :\, x_1=x_2 =...=x_n\}.
$$

\subsection{Steep vectors}
\label{subs: steep vectors}

The definition of the class of steep vectors consists of few steps at which we define the sets
 $\st_0$, $\st_1$, $\st_{2}$, and  $\st_3$.
We start with $\st_3$, the class of {\it very steep} vectors.
Set
$$
  \st_{3}= \st_{3} (\LL) :=  \{x\in \C^n\,:\,
  \exists i\leq p^{r_0}  \,\, \mbox{ such that } \, \,   x_{i}^{*}> (n/i)^3
    x_{p^{r_0}}^{*}\}.
$$
Note that one can relate this class to the class of vectors close to $m$-sparse
vectors with $m=p^{r_0}-1$. Indeed, consider the following variant of
the weak $\ell_{1/3}$-norm,
$$
  |||x|||  =\frac{1}{n^3}\, \sup _{i\leq m} i^{3} x_i^* .
$$
Then
$$
  \st_{3}= \{x\in \C^n  \,:\, \exists\, \, m\mbox{-sparse vector}  \,\, y\in \C^n \, \,
  \mbox{ such that } \, \,  \|x-y\|_\infty < ||| x |||\}.
$$

 We now define the set $\st_0$. For $r_0\leq i\leq r-1$ set
$$
  \st_{0, i}= \st_{0, i}(\LL):=\{x\in \C^n\,:\,  x\not\in \st_{3} \cup \bigcup_{j=r_0}^{i-1} \st_{0, j}
  \,\, \mbox{ and } \, \,  x_{p^i}^{*}> 4 d x_{p^{i+1}}^{*}\},
$$
where $\cup_{j=r_0}^{r_0-1}\st_{0, j}$ means $\emptyset$,  and for $i=r$ let
$$
  \st_{0, r}=\st_{0, r}(\LL):=\{x\in \C^n\,:\,  x\not\in \st_{3} \cup \bigcup_{j=r_0}^{r-1} \st_{0, j}
  \,\, \mbox{ and } \, \,  x_{\lc n_1/p\rc}^{*}> 4 d x_{n_1}^{*}\},
$$
Let
$$
 \st_0=\st_0(\LL) :=\bigcup _{i=r_0}^{r}\st _{0, i}.
$$

Next we define $\st_1$ and $\st_2$ as
$$
  \st _1 = \st_{1} (\LL): =\{x\in \C^n\,:\,  x\not\in  \st_{0}\cup \st_{3}
  \,\, \mbox{ and } \, \,  x_{n_{1}}^{*}> d^{3/2}\, x_{n_2}^{*}\}
$$
and
$$
  \st _2 = \st_{2} (\LL): =\{x\in \C^n\,:\,  x\not\in  \st_{0}\cup \st_1 \cup \st_{3}
  \,\, \mbox{ and } \, \,  x_{n_{2}}^{*}> d^{3/2}\, x_{\nn}^{*}\} .
$$

Below we  work with constant shifts of steep vectors,
so we also introduce the following sets for  $0\leq i\leq 3$,
$$
   \st _i^\CC  := \{v\in \C^n\,:\, v=x+y \,\, \, \, \mbox{ for some } \, \, \, \,
  x\in  \st_{i} \, \, \mbox{ and } \, \,  y\in \CC \, \,
  \mbox{ with } \, \, |y_1|\leq x_{n_1}^*/10 \}.
$$
Note that
\begin{equation}\label{t3shift}
  \st_{3}^\CC\subset \{v\in \C^n\,:\,
  \exists i\leq p^{r_0}  \,\, \mbox{ such that } \, \,   v_{i}^{*}> 0.9 (n/i)^3
    v_{p^{r_0}}^{*}\}.
\end{equation}

Finally we define sets of {\it steep} and {\it shifted steep} vectors as
$$
   \st:=\st_0\cup \st_1\cup \st_2\cup \st_3 \quad \mbox{ and }
 \quad \st _\CC:=\st_0^\CC\cup \st_1^\CC \cup \st_2^\CC.
$$
Note that the set of steep vectors contains {\it very steep} vectors.

\medskip

Our first goal in this section is to prove  the following theorem.

\begin{theor} \label{t:steep}
Let $d\geq 1$ be large enough, $n\geq d^3$, $1\leq \LL\leq n/d^3$, $\KK\subset [n]$ with $|\KK^c|\leq L$,
and $\ww\in \C$ be such that $|\ww|\leq d/2$. Let
$$
  \Event_{steep}:=\Big\{M\in\Mc\,:\,\exists\;v\in \st_\CC\, \, \, \mbox{ such that }
  \, \, \,
  \|(M-\ww \idmat)^\KK v\|_2 < \frac{\LL^3  d}{ n^{6} } \, \,  \|v\|_2
  \Big\}.
$$
 Then
\begin{equation*}
\label{Psteep}
\Prob(\Event_{steep}) \leq \min\left( \exp\left(-\LL/d\r),\, \exp\big(-(\ln d)(\ln  n)/20\big)\right).
\end{equation*}
\end{theor}

\smallskip

We will now formulate simple properties of steep vectors which will be used later.
The following lemma shows that the vectors from the complement of $\st$
have a rather regular decay of their coordinates.

\begin{lemma} \label{l:decay}
Let $d\geq 1$ be large enough, $n\geq d^3$, and  $x\not \in \st$.
Then
$$
 x_m^*\leq \left\{
\begin{array}{ll}
\left(n/ m\right)^{6} \,  x^*_{\nn} & \mbox{ if }\, 1\leq m\leq p^{r_0},
\\
d \left(n/ m\right)^{3} \,  x^*_{\nn}, & \mbox{ if }\, p^{r_0}\leq m\leq n_1.
\end{array}
\right.
$$
Furthermore, for  every  $n_1\leq j \leq i\leq \nn$ one has
$$
  x_j^*\leq    x^*_{n_1}\leq d^3  x^*_{\nn} \leq        d^3 x^*_{i}.
$$
\end{lemma}

\begin{proof}
Since $p^r\ge \lc n_1/p\rc$ and $x\notin \st$, we have
$
x_{p^r}^* \leq 4d x^* _{n_1}\le 4d^{4}x^*_{\nn}.
$
Therefore,  for every $r_0\leq j\leq r$,
$$
   x_{p^j}^*\leq  (4d) x_{p^{j+1}}^*
  \leq \ldots \leq  (4d)^{r-j} x_{p^r}^*\leq 4d^4(4d)^{r-j}x_{\nn}^*.
 $$
Since for large $d$ one has $4d<p^3$ and $p^{r}\le n_1\le n/d^{3/2}+1$, we deduce for $j=r_0$ that
$$
 x_{p^{r_0}}^*\leq 4d^4(4d)^{r-r_0} x_{\nn}^*\le 4d^4p^{3(r-r_0)} x_{\nn}^*\le 4d^4(n_1/p^{r_0})^3 x_{\nn}^*\le(n/p^{r_0})^3 x_{\nn}^*.
$$
Since $x\not \in \st_{3}$, this implies the bound for every $1\leq m\leq p^{r_0}$.

Now let $p^j \leq m <p^{j+1}$ for some $r_0\leq j<r$. Then
$$
   x_{m}^*\leq   x_{p^{j}}^*\leq    4d^4p^{3(r-j)} x_{\nn}^* \leq
  4d^4(n_1p/m)^{3}x_{\nn}^*\leq d \left(n/m\r)^3  x_{\nn}^* ,
$$
which proves the case $m<p^r$. For $p^r\leq m\leq n_1$ we have $n/m\geq  n/n_1\geq d^{3/2}/2$, hence
$$
  x_{m}^*\leq 4d  x_{n_1}^*\leq 4 d^4 x_{\nn}^* \leq \left(n/m\r)^3  x_{\nn}^* .
$$

The last inequality is trivial.
\end{proof}

The next lemma provides a comparison of the $\ell_2$-norm of a given vector with one of its coordinates.
It is similar to Lemma~3.5 from \cite{LLTTY first part}. Since our choice of
parameters as well as the definition of steep vectors is slightly different we provide the proof
for the sake of completeness.

\begin{lemma} \label{l:norma}
Let $d\geq 1$ be large enough,  $n\geq d^3$, $1\leq \LL\leq n/d^3$, and  $x\in \C^n\setminus \st _{3}$.
Then
$$
  \|x\|_2 \le \frac{n^{6}}{100 \LL^3  d^{3/2}}\, x_m^*,
$$
where $m=p^i$  if  $x\in \st _{0,i}$  for some  $r_0\leq i\leq r$ and
$m=n_1$ if $x\notin \st _{3}  \cup \st _{0}$.
\end{lemma}

\begin{proof} By the definition of $\st_3$, for
$x\not \in \st_{3}$ one has
$$
   \sum _{i=1}^{p^{r_0}} (x^*_i)^2\leq  \sum _{i=1}^{p^{r_0}} (n/i)^{6}\,  (x^*_{p^{r_0}})^2
   \leq \frac{4}{3} \, n^{6}\,  (x^*_{p^{r_0}})^2.
$$
If  $x\in \st _{0, r_0}$, then
$$
  \|x\|_2^2 = \sum _{j=1}^{p^{r_0}-1} (x_{j}^*)^2 + \sum _{j=p^{r_0}}^{n} (x_{j}^*)^2 \leq
  \frac{4}{3} \, n^{6}\,  (x^*_{p^{r_0}})^2 + n  (x_{p^{r_0}}^*)^2 \leq
  2 \, n^{6}\,   (x_{p^{r_0}}^*)^2,
$$
which implies the bound in the case $i=r_0$.
Let $x\in \st _{0,i}$ for some $r_0< i\leq r$.
Then $x\not\in \st_{3}$ and for every $j<i$ one has $x\not\in \st_{0, j}$. Therefore,
assuming without loss of generality  that $x_{p^i}^*=1$, as in the previous lemma
we observe
$$
  x_{p^{r_0}}^*\leq   (4d)^{i-r_0} x_{p^i}^* =  (4d)^{i-r_0} \le   p^{3(i-r_0)}.
$$
Therefore, using again that $4d\le p^3$, we observe
\begin{align*}
      \|x\|_2^2 &=  \sum _{j=1}^{p^{r_0}} (x_{j}^*)^2 + \sum _{j=p^{r_0}+1}^{p^{r_0+1}} (x_{j}^*)^2 + \sum _{j=p^{r_0+1}+1}^{p^{r_0+2}} (x_{j}^*)^2+\dots
      \\
      &\leq 2  \, n^{6}\,  (x^*_{p^{r_0}})^2 + p^{r_0+1} (4d)^{2(i-r_0)} + p^{r_0+2} (4d)^{2(i-r_0-1)} + \ldots +  p^{i} (4d)^{2} + n
      \\
      &\leq 3n^6p^{6(i-r_0)}+(i-r_0)p^{6i+1-5r_0}
      \\
      &\leq 4n^6p^{6(r-r_0)}.
\end{align*}
Recalling that $20\LL/d \leq p^{r_0}\leq 20\LL p/d$ and $p^r\leq n_1=\lceil n/d^{3/2}\rceil \leq p^{r+1}$, we have
$$
 p^{r-r_0}\le \frac{n_1 d}{20L}\leq \frac{ n}{10L \sqrt d},
$$
which, together with the above, implies the desired bound in the case  $r_0<i\leq r$.
Repeating  the above scheme and using that $p^r\leq n_1<p^{r+1}$, we  obtain the result for
$x\not \in \st_{3} \cup \st _0$.
\end{proof}

\subsection{Lower bounds on $\|Mx\|_2$ for  vectors from $\st_0$}

Here we provide lower  bounds on the ratio $\|Mx\|_2/\nx _2$  for vectors $x$ from $\st_0$.
Recall that given $\eps$ and $k$ the set $\Omega_{k, \eps}$  was introduced
before Theorem \ref{graph th known} (see  (\ref{Oo})).

\begin{lemma} \label{l:T0} There exists an absolute positive constant $C$ such that the following holds.
Let $d\geq C$, $n\geq d^3$, $1\leq \LL \leq n/d^3$,  $\KK\subset [n]$ with $|\KK^c|\leq \LL$, and let
$\ww\in \C$ be such that $|\ww|\leq d/2$.
Then for every $v\in \st_0^\CC$ and every
$$
  M\in  \Omega_{n_1,\eps _0}   \cap  \,   \bigcap _{j= r_0+1}^{r}\, \Omega_{p^j,\eps _0}
$$
one has
$$
 \|(M-\ww \idmat)^\KK  v\|_2 \geq  \frac{p^{r_0/2}\,\LL^3\, d^{2}}{n^{6}}\, \|v\|_2.
$$
\end{lemma}

\begin{proof}
Let $v=x+y$, where $x\in \st_{0}$ and $y\in \CC$ with $|y_1|\leq x^*_{n_1}/10$.
Fix $r_0\leq  i\leq r$ such that  $x\in \st_{0, i}$. If $i<r$ set $m=p^i$, if $i=r$ set $m=\lc n_1/p\rc$. Then $x_{m}^{*}> 4 d x_{p m}^{*}$.
Fix a permutation $\sigma=\sigma_x$ of $[n]$ such that $x_i^*=|x_{\sigma(i)}|$ for $i\le n$.
Let
$$
  J^\ell=\sigma([m]), \quad  J^r=\sigma([p m]\setminus[m]), \quad \mbox{ and } \quad
  J=(J^\ell\cup J^r)^c.
$$
Then, for sufficiently large $d$,
$$
  |J^\ell\cup J^r|= pm \leq p \lc n_1/p\rc \le  c_{\ref{graph th known}}\eps_0 n/d \quad \, \, \, \mbox{ and } \quad \, \, \,
 |J^r|=(p-1)|J^\ell|=(p-1)m.
$$
Denote by $I_\ell$ the set of rows  having exactly one 1 in $J^\ell$ and no 1's in $J^r$.
 Lemma~\ref{c:SJ}  implies that
$$
  |I_\ell|\geq  (1-2p\eps_0)m d\ge 3m d/5.
$$
Let $I=(I_\ell\setminus (J^\ell\cup J^r))\cap \KK$ (so that the minor $I\times (J^\ell\cup J^r)$
does not intersect the main diagonal and only rows indexed by $\KK$ are considered). Since
$|\KK^c|\leq \LL$ and $m\geq p^{r_0}\geq 20\LL/d$, we have
$$
 |I|\geq 3md/5-pm - \LL  \geq  md(3/5 - p/d - 1/20)\geq md/2
$$
provided that $d$ is large enough.
By definition, for every $s\in I$  there exists $j(s)\in J^\ell$ such that
$$
   \supp R_{s}\cap J^\ell=\{j(s)\},\quad \supp R_{s}\cap J^r= \emptyset,\quad\text{and}\quad
   \max_{i\in J}|x_i|\le x^*_{m p} .
$$
Since $|y_1|\leq x^*_{n_1}/10 \leq x^*_{pm}/10$,  $s\not\in J^\ell\cup J^r$ (which implies $\vert x_{s}\vert \leq x^*_{pm}$), and $j(s)\in J^{\ell}$
(which implies $|x_{j(s)}|\geq x^*_m> 4 d x^*_{m p}$), we obtain
\begin{align*}
  |\langle R_{s} (M&-\ww \idmat),\, (x+y)^\dagger\rangle|
  =\Big| x_{j(s)}+\sum_{j\in J\cap \supp R_{s}} x_j - \ww   x_s + d y_1 - \ww y_1 \Big| \\&  \geq|x_{j(s)}|- (d-1)\, x_{m p}^*
   - |\ww| \, x_{m p}^* - (d+|z|) |y_1| \ge  x^*_{m}/2 .
\end{align*}
Since the number of such rows is $|I|\geq m d/2$ and $I\subset \KK$,  we obtain
$$
  \|(M-\ww \idmat)^\KK (x+y)\|_2 \geq \sqrt{m d}\, x^*_{m}/ 2\sqrt{2}.
$$
Using Lemma~\ref{l:norma}, we have
$$
 \|x+y\|_2\leq \nx_2 + \|y\| _2\leq \frac{n^{6}}{100 \LL^3 d^{3/2}} \, x^*_{m} + \sqrt{n} |y_1| \leq
 \left(\frac{n^{6}}{100 \LL^3 d^{3/2}} + \frac{n^{1/2}}{10} \r) x^*_{m},
$$
which implies the desired result.
\end{proof}

\subsection{Bounds for vectors  from $\st_1^\CC \cup \st_2^\CC$}\label{subs: nets}

For the  vectors  from $\st_1^\CC \cup \st_2^\CC$ we will
use the union bound together with a covering argument.
We first construct nets for ``normalized'' versions of the sets $\st_i^\CC$ and then provide
individual probability bounds for elements of the nets. The natural normalization for ``non-shifted''
component would be
$x_{n_1}^{*}=1$, which we use for $\st_1^\CC$.
However, for individual probability bounds below and to have the same level of approximation, it is
more convenient to use a slightly different normalization for $\st_2^\CC$.  We construct nets for
the sets
$$
 \st'_i=\{x+y \, \, :\, \, x\in \st_i\,:\, x_{n_i}^{*}=1 \, \, \, \, \mbox{ and } \,\,\,\,
  y\in \CC \, : \, |y_1|\leq x_{n_1}^{*}/10\} , \,\,\, i=1,2.
$$

Then, repeating the proof of Lemma~3.8 from \cite{LLTTY first part} with slight adjustments, we obtain the
following lemma.

\begin{lemma}[Cardinalities of nets]
\label{l:net}
Let $d\leq n^{1/3}$ be large enough and $i=1,2$.
There exists a set $\Net _i= \Net_i' + \Net_i''$, $\Net_i'\subset \C^n$, $\Net_i''\subset \CC$, with the following properties.
The cardinality
$$
  |\Net_i| \le \exp\left( 7  n_{i+1}\ln d \r).
$$
For every $u\in \Net_i'$ one has $u_j^*=0$ for all $j\geq n_{i+1}$.
For every $x\in \st_i$ with  $x_{n_i}^{*}=1$ and every $y\in \CC$ with
$|y_1|\leq x_{n_1}^{*}/10$ there are $u\in \Net_i'$ and $w\in \Net_i''$
  satisfying
  $$
         \|x-u\|_\infty \leq 1/d^{3/2} \, \, \, \, \mbox{ and } \,\,\,\,  \|y-w\|_\infty \leq 1/d^{3/2}.
  $$
\end{lemma}

\smallskip

We now turn to the individual probability bounds where we will work in a more general setting by considering any $n\times n$ complex matrix $\WW$ instead of the shift $\ww\idmat$.
To obtain the lower bounds on $\|(M+\WW) x\|_2$ for vectors $x$ from our nets, we  investigate
the behavior of the inner products $\la R_i(M+\WW), x^\dagger \ra$. One of the tools
that we use is the L\'evi concentration function for  $\la R_i(M+\WW), x^\dagger \ra$. To estimate this
function we, in particular, will use Theorem~\ref{graph th known} for $2m$ columns of $M$ corresponding
to the $m$ biggest and $m$ smallest (in modulus) coordinates of $x$, where $m=n_1$ or $m=n_2$.
The main difficulty in this scheme comes from the restriction $2m \le c_{\ref{graph th known}} \eps n/d$
in Theorem~\ref{graph th known}, which is not satisfied for $m=n_2$. To resolve this problem we split
the set of $2m$ columns into smaller subsets of columns of size at most $c_{\ref{graph th known}} \eps n/d$,
and create independent random variables corresponding to this splitting and such that their sum is
$\la R_i(M+\WW), x^\dagger \ra$ up to a constant. Then we apply Proposition~\ref{prop: esseen},
allowing to deal with L\'evy concentration function for sums of independent random variables.

We first describe subdivisions of $\Mc$ needed for our construction. Given $J\subset [n]$ and
$M\in \Mc$  denote
$$
  I (J, M) = \{i \leq n \, : \, |\supp R_i(M) \cap J |       =1\}
$$
({\it cf.},  the definition of $\il(M)$,  $\ir(M)$  before  Lemma~\ref{c:SJ} -- clearly, if
we split $J$ into $J^\ell$ and $J^r$, then $I (J, M)= \il(M) \cup \ir(M)$). By $\MSet _J$ we denote
the set of $n\times |J|$ matrices obtained from matrices $M\in \Mc$ by taking columns with indices in $J$,
i.e.,
$$
     \MSet _J = \left\{ V = \{v_{ij}\}_{i\leq n, j\in J}\, : \, \exists M\in \Mc \, \, \mbox{ such that }
     \, \, \forall i\leq n\, \forall j\in J \, \,\, \,  v_{ij}=\mu_{ij} \r\}.
$$
Now we fix $q_0\leq n$ and a partition $J_0$, $J_1$, ..., $J_{q_0}$ of $[n]$.
Given subsets $I_1$, ..., $I_{q_0}$ of $[n]$ and $V\in \MSet _{J_0}$, denote
$\ii = (I_1, \ldots, I_{q_0})$ and  consider the class
$$
   \f (\ii, V) =  \left\{M\in \Mc \, :\,  \forall  q\in [q_0] \quad I (J_q, M) =  I_q \,\,
     \mbox{ and } \,\,\forall i\leq n\, \forall j\in J_0  \,\,\,  \mu_{ij}= v_{ij} \right\}
$$
(depending on the choice of $\ii$ such a class can be empty). In words, we fix the columns
indexed by $J_0$ and for each $q\in [q_0]$ we fix the rows having exactly one 1 in columns indexed by $J_q$.
Then $\Mc$ is the disjoint union of classes $\f (\ii, V)$ over all $V\in \MSet _{J_0}$ and
all $\ii \in (\mathcal{P} ([n]))^{q_0}$, where  $\mathcal{P} (\cdot)$ denotes the power set.

Furthermore, given $V$ and $\ii$ as above, we split each class $\f (\ii, V)$
into smaller equivalence classes using the following
equivalence relation. Fix $i\leq n$ and $A\subset [q_0]$.
Denote $A_0:=\{0\}\cup ([q_0]\setminus A)$. We say that two matrices $M, \widetilde M\in \f (\ii, V)$
are equivalent  if
$$
  \forall s< i\, \, \forall j\leq n \quad\quad  \mu _{sj}= \tilde \mu _{sj},
$$
$$
  \forall s\leq n\, \, \forall j\in  J':= \bigcup_{q\in A_0}J_q \quad\quad  \mu _{sj}= \tilde \mu _{sj},
$$
and
$$
  \forall s\leq n\, \, \forall q\in A \quad\quad  \sum _{j\in J_q}\mu _{sj}= \sum _{j\in J_q}\tilde \mu _{sj}.
$$
The collection of equivalence classes corresponding to this relation will be denoted by
$$
 \mathcal{H} = \mathcal{H} (\f (\ii, V), i, A),\quad
 \mbox{ in particular }\quad \f (\ii, V) = \bigcup _{H\in \mathcal{H}} H.
$$
Note that for matrices in a given class $H$, the rows $R_1$, ..., $R_{i-1}$ are fixed and every block $[n]\times J_q$ has a prescribed sum in each row,
thus, in a sense, these blocks are independent of each other on $H$.

Finally, given a vector $x\in \C^n$, an index $i\leq n$,  a class $H\in \mathcal{H}$ (in particular, $V, \ii, i, A$ are fixed), and $q\in A$, we introduce a random variable $\xi_q$
on $H$ by
$$
   \xi _q = \xi_q (M) := \sum _{j\in J_q} \mu _{ij} x_j.
$$
In words, $\xi_q$ represents the dot product of $x$ with the restriction of the $i$-th row to $J_q$.
Later we use this construction in the case when $i\in I_q$ for all $q\in A$,
that is for a specific choice of parameters
defining our classes  (recall here that for $M\in H$, $|\supp R_i(M) \cap J_q |=1$
provided that $i\in I_q$). As we have already mentioned, by construction, for
 matrices in the class $H$ every block $[n]\times J_q$ has a prescribed sum
 in each row, therefore the random variables $\xi_q$, $q\in A$, are independent.
Thus, using that for a fixed matrix $W=\{w_{ij}\}$ and a fixed constant vector
$y\in \CC$, the function
$$
  \xi'  = \xi'(M) := \sum _{j\in  J'} \mu _{ij} x_j + \sum _{j=1}^n
  w_{ij} x_j + y_1 d  + y_1 \sum _{j=1}^n  w_{ij}
$$
is a constant on $H$, we may apply Proposition~\ref{prop: esseen}
(in which we identify $\C$ with $\R^2$) to
$$
   \left|\la R_i(M+\WW), (x+y)^\dagger \ra\right| = \Big| \sum _{q\in A} \xi_q + \xi' \Big|
$$
with some $\alpha >0$ satisfying  $\cf (\xi_q , 1/3)\leq \alpha$ for every $q\in A$. This gives
\begin{equation}\label{conc-inner}
  \Prob \left(\left|\la R_i(M+\WW), (x+y)^\dagger \ra \right| \leq 1/3\right) \leq \frac{C_0}{\sqrt{(1-\alpha) |A|}},
\end{equation}
where $C_0$ is a positive absolute constant.

\smallskip

We are ready now to estimate individual probabilities.

\begin{lemma}[Individual probability]
\label{individual}
There exist absolute constants $C, C'>1>c_1>0$ such that the following holds.
Let $C<d< n$, $\KK\subset [n]$, $\eps \in [\eps _0, 0.01]$.
Set $m_0=   \lfloor c_{\ref{graph th known}}\eps n/(2d)\rfloor$ and
let $m_1$ and $m_2$ be such that  $m_1<m_2\leq n-m_1$. Assume that
$x\in \C^n$ satisfies
$$
   x^*_{m_1}> 2/3 \quad \mbox{ and } \quad x^*_i = 0 \, \, \, \mbox{ for every }\, \, i>  m_2.
$$
Let $\WW$ be a complex $n\times n$ matrix, $y\in \CC$, and denote $m=\min(m_0, m_1)$ and
$$
  E= E(x) = \left\{
     M \in \Mc\, :\, \|(M+\WW)^\KK (x+y) \|_2\leq   \sqrt{ c_1 m d}   \r\}.
$$
Then if $m_1 \leq m_0$ and  $|\KK^c| \leq 3 m_1 d/5$
$$
   \Prob(E\cap  \Omega _{2m_1,\eps} )\leq  \left(5/6 \r)^{m_1 d/2},
$$
if $m_1> C'  m_0$, $\eps=0.01$, and $|\KK^c| \leq 3 m_0 d/5$
$$
   \Prob(E\cap  \Omega _{2m_0,\eps} )\leq  \left(\frac{C n}{m_1 d}\r)
   ^{m_0 d/4}.
$$
\end{lemma}

\begin{rem}\label{rem-ind}
We apply this lemma below twice: first  with $m_1=n_1<m_0$, $m_2=n_2$, $\eps = 0.01$, obtaining
\begin{equation*} \label{individual-one}
   \Prob(E\cap  \Omega _{2n_1,0.01}  )\leq  \left(5/6 \r)^{n_1 d/2};
\end{equation*}
then with $m_1=n_2>m_0$, $m_2=n_3$, $\eps = 0.01$, obtaining
\begin{equation*}\label{individual-two}
  \Prob(E\cap  \Omega _{2m_0, 0.01} )\leq  \left(\frac{C n}{d n_2} \r)
   ^{0.01\, c_{\ref{graph th known}} n/8} \leq  \left(\frac{C_1}{d} \r)^{c n},
\end{equation*}
where $C_1=8 C^3$ and $c= c_{\ref{graph th known}} n/2400$ are positive absolute constants.
\end{rem}

\begin{proof}
Fix $\gamma = 3md/5n$.
Fix $x\in \C^n$ and $y\in \CC$ satisfying the condition of the lemma. Let $\sigma$ be a permutation
of $[n]$ such that
$x_i^*=|x_{\sigma(i)}|$ for all $i\le n$. Denote $q_0=m_1/m$ and without loss of generality
assume that either $q_0=1$ or that $q_0$ is a large enough integer. Let
$J^{\ell}_1, J_2^\ell, \ldots,  J^\ell_{q_0}$ be a partition
of $\sigma ([m_1])$ into sets of cardinality $m$.
Let $J^{r}_1, J_2^r, \ldots,  J^r_{q_0}$ be a partition
of $\sigma ([n-m_1+1, n])$ into sets of cardinality $m$. Denote
$$
 J_q:=J^\ell_q\cup J^r_q  \, \, \, \mbox{ for }\,\,\, q\in [q_0] \quad \mbox{ and } \quad
 J_0:= [n]\setminus \bigcup _{q=1}^{q_0} J_q.
$$
Then $J_0$, $J_1$, ..., $J_{q_0}$ is a partition of $[n]$, which we fix in this proof.
Let $M\in\omep$.  For every pair  $J^\ell_q$, $J^r_q$, let the sets $\il_q(M)$ and  $\ir_q(M)$
be defined as before  Lemma~\ref{c:SJ} and let $I_q=I_q(M)= \il_q(M) \cup \ir_q(M)$.
Since
$$
 |J_q|=2m \le 2m_0\le c_{\ref{graph th known}} \eps n/d,
$$
 Lemma~\ref{c:SJ}  implies  that
$$
 |\il_q(M)|,\,|\ir_q(M)|\in[(1-4\eps)md,\,md],
$$
 in particular,
 \begin{equation}\label{cond-card}
 |I_q|\in[2(1-4\eps)md,\,2md].
 \end{equation}
Now we split $\Mc$ into a disjoint union of classes $\f (\ii, V)$ defined at the beginning of this subsection
with $V\in \MSet _{J_0}$ and  $\ii =(I_1, \ldots, I_q)$ and note that $\omep \cap \f (\ii, V)\ne \emptyset$
implies that
$I_q$ satisfies (\ref{cond-card}) for every $q$.
Thus, to prove our lemma it is enough to prove a uniform upper bound for  such classes,
indeed,
$$
  \Prob(E(x)\cap  \omep )\leq \max  \p (E(x) \cap \omep\,|\, \f (\ii, V))
  \leq \max  \p (E(x) |\, \f (\ii, V))
$$
where the first maximum is taken over all  $\f (\ii, V)$ with $\omep \cap \f (\ii, V)\ne \emptyset$ and
the second maximum is taken over $\f (\ii, V)$ with $I_q$'s satisfying (\ref{cond-card}).

Fix such a class $\f (\ii, V)$ and denote the uniform probability on it
just by $\p_\f$, that is
$$
 \p_\f (\cdot) = \p( \cdot \, | \,  \f (\ii, V)).
$$
 Let
 $$
  I: = \bigcup _{q=1}^{q_0} I_q.
 $$
 Note that $|I|\leq 2 q_0 md$.
We first show that the set of $i$'s belonging to many $I_q$'s is rather large.
More precisely, given $i\in [n]$ denote
$$
 A_i = \{ q\in[q_0]\, : \, i\in I_q\}, \quad
 I_{00}=\{i\in I\, : \, |A_i|\ge \gamma   q_0\},  \quad \mbox{ and }\quad I_0=I_{00}\cap \KK.
$$
Then, using bounds on cardinalities of $I_q$'s, one has
$$
  2(1-4\eps) m d q_0 \leq \sum_{q=1}^{q_0} |I_q| = \sum_{i=1}^n |A_i| \leq |I_{00}| q_0 + (n-|I_{00}|) \gamma q_0
  \leq |I_{00}| q_0 + n \gamma q_0.
$$
Since $\eps \leq 0.01$, $\gamma = 3md/(5n)$ and $|K^c|\leq 3md/5$, we get
$$
   |I_0|\geq |I_{00}| -|\KK^c|\geq  2(1-4\eps) m d - 6md/5   \geq 2md/3.
$$
 Without loss of generality
we assume that $I_0=\{1, 2, \ldots |I_0|\}$ and  only consider the first
$k:=\lc  2md/3 \rc$  indices from it. Then $[k]\subset I_0 \subset \KK$.

Now, by definition, for matrices $M\in E(x)$ we have
$$
  \|(M+\WW)^\KK (x+y)\|_2^2 = \sum _{i\in \KK} | \la R_i(M+\WW), (x+y)^\dagger \ra|^2 \le c_1\, md .
$$
Therefore there are at most $9 c_1 md$ rows  with
$| \langle R_i(M+\WW), (x+y)^\dagger \rangle|\ge 1/3$. Hence,
$$
 |\{i\leq k\, : \,|  \langle R_i(M+\WW), (x+y)^\dagger \rangle|< 1/3\}|\ge
   2md/3 - 9 c_1 md \ge 2(1-14 c_1)md/3
$$
(we used that   $I_0\subset K$).
Let
 $
   k_0:= \lc 2(1-14 c_1) md/3 \rc
 $
and for every $i\leq k$ denote
$$
  \Omega_i:=\{M\in\f (\ii, V) \, :\, |\la R_i(M+\WW), (x+y)^\dagger \ra|< 1/3 \}
  \quad \mbox{ and } \quad \Omega_0= \f (\ii, V) .
$$
Then
\begin{align*}
  \Prob_{\f}(E(x)) &\le  \sum _{B\subset [k]\atop |B|=k_0 } \,
  \Prob_{\f}\Big(\bigcap_{i\in B}\Omega_i\Big)
  \le   {k \choose k_0 }\, \max _{B\subset [k]\atop |B|=k_0 } \, \Prob_{\f}\Big(\bigcap_{i\in B} \Omega_i\Big).
\end{align*}
Without loss of generality we assume that the maximum above is attained at $B=[k_0]$. Then
\begin{equation} \label{ptensor}
  \Prob_{\f}(E(x))  \le
 \left(1/ c_1\r)^{10 c_1 md} \,\, \, \prod_{i=1}^{k_0} \, \Prob_{\f}(\Omega_{i}|\,\Omega_1\cap\ldots\cap \Omega_{i-1}).
\end{equation}

 Next we estimate the factors in the product. Fix $i$ and $A_i= \{ q\, : \, i\in I_q\}$. Since $i\in I_0$, we have
 $|A_i|\geq \gamma q_0$. Consider
 the splitting of $\f (\ii, V)$ into classes $H\in \mathcal{H} = \mathcal{H} (\f (\ii, V), i, A_i)$ as described
 before  the statement of the lemma and let $\p _H$ denote the uniform probability on a class $H$, i.e.,
 $\p_H(\cdot)=\p(\cdot | H)$.  Since in
every class $H$ all matrices have their first $i-1$ rows fixed,  for every $H$ the intersection
 $H_i:=H\cap \Omega_1\cap\ldots\cap \Omega_{i-1}$ is either  $H$ or $\emptyset$.
 Thus
 $$
    \Prob_{\f} (\Omega_{i} | \, \Omega_1\cap\ldots\cap \Omega_{i-1} ) \leq
    \max _{H: H_i\ne \emptyset}   \Prob_{H}(\Omega_{i}).
 $$
Fix $H$ such that $H_i\ne \emptyset$ and consider the random variables $\xi_q$, $q\in A_i$, defined above. Then by
(\ref{conc-inner}) we have
$$
  \Prob_{H}(\Omega_{i}) = \Prob_{H}(|\la R_i(M+\WW), (x+y)^\dagger \ra | \leq 1/3) \leq \frac{C_0 \alpha}{\sqrt{(1-\alpha) |A_i|}}\leq
  \frac{C_0 \alpha}{\sqrt{(1-\alpha) \gamma q_0}}
$$
where
 $\alpha = \max _{q\in A_i} \cf (\xi_q (M), 1/3)$.
Note that
in the case $q_0=1$ we just have
$$
  \Prob_{H}(\Omega_{i}) = \alpha = \cf (\xi_1 (M), 1/3).
$$
Thus it remains to estimate $\cf (\xi_q , 1/3)$ for $q\in A_i$. Fix $q\in A_i$, so that $i\in I_q$.
Recall that, by construction, the intersection of the support of $R_i(M)$ with $J_q$
is a singleton. Denote the corresponding index by $j(q)$.  Then
$$
  \xi _q = \xi_q(M) = \sum _{j\in J_q} \mu _{ij} x_j = x_{j(q)}
$$
and note that $|x_{j(q)}|>2/3$ whenever $j(q)\in J^\ell_q$ and
$x_{j(q)} =0$ whenever $j(q)\in J^r_q$.
Denote
$$
  H^\ell = \{M\in H \, :\, j(q)\in J_q^\ell \} \quad \mbox{ and } \quad
  H^r = \{M\in H \, :\, j(q)\in J_q^r \}.
$$
Since $|\xi _q|$ is either larger than $2/3$ or equals $0$ we observe that
$$
   \cf (\xi_q (M), 1/3) \leq \max\{ \Prob_{H}(H^\ell),\,  \Prob_{H}(H^r)\}.
$$

\begin{claim} \label{cl-one-n} For $i\leq \lc 2md/3 \rc$ one has
$$
   \max\{ \Prob_{H}(H^\ell),\, \Prob_{H}(H^r)\}   \leq 4/5.
$$
\end{claim}
Combining the probability estimates starting with  (\ref{ptensor}),
using that $\gamma = 3md/5n$, and applying  Claim~\ref{cl-one-n},
we obtain in the case $q_0=m_1/m\geq C'$,
\begin{align*}
   \Prob_{\f}(E(z))&\leq  \left( \frac{1}{c_1}\r)^{10 c_1 md} \,\, \,    \left(\frac {4 C_0}{\sqrt{5\gamma q_0}}\r)^{2(1-14 c_1) md/3}
   \\& =
    \left(\frac{1}{c_1}\r)^{10 c_1 md} \,\, \,    \left(\frac{4 C_0 \sqrt{n}}{\sqrt{3 m_1 d}}\r)^{2(1-14 c_1)md/3} \leq
   \left(\frac{C_1 n}{m_1 d}\r)^{md/4},
\end{align*}
provided that $c_1$ is small enough and $C'$ is large enough.
In the case $q_0=1$ we have
$$
   \Prob_{\f}(E(z))\leq \left( \frac{1}{c_1}\r)^{10 c_1 md}   \,\, \, \left( \frac{4}{5}\r)^{2(1-14 c_1) md/3}   \leq
   \left( \frac{5}{6}\r)^{md/2}
$$
provided that $c_1$ is small enough.
This completes the proof.
\end{proof}

\begin{proof}[Proof of Claim~\ref{cl-one-n}]
We show the bound for $\Prob_{H}(H^\ell)$, the other bound is similar.
Note that for matrices in $H$ we have $I(J_q, M) = I_q$ and $I_q$ satisfies
(\ref{cond-card}). Since $|J_q^\ell|=|J_q^r|=m$, we observe that on $H$ one has
$$
  |I_q^\ell (M)|\geq |I_q|- md \geq  (1-8\eps) md \quad \mbox{  and } \quad
  |I_q^r (M)|\leq md.
$$

To compare cardinalities,  define a relation $R\in H^\ell \times H^r$ by
$(M,\,M')\in R$ iff $M\in H^\ell$,  $M'\in H^r$, and
$M'$ can be obtained from $M$ by a simple switching in
$$(I_q\setminus[i-1]) \times J_q$$
(note that the i-th row is necessarily involved in
the switching). It is easy to check that
for every $M\in H^\ell$ and  every $M'\in H^r$ one has
$$
  \vert R(M)\vert  = \vert \ir _q (M)\setminus [i-1]\vert  \quad \mbox{ and } \quad
 \vert R^{-1}(M') \vert=\vert  \il _q (M')\setminus [i-1]\vert ,
$$
hence
$|R(M)|\geq (1-8\eps) md - i+1$  and
$|R^{-1}(M')|  \leq md$.
Claim~\ref{multi-al} yields
$$
      |H^\ell|/|H^r|\leq \frac{md}{(1-8\eps) md - i+1}\leq \frac{1}{1/3-8\eps}.
$$
Therefore,
$$
 |H|/|H^\ell| = (|H^\ell|+|H^r|)/|H^\ell| \geq 4/3-8\eps ,
$$
which completes the proof since $\eps \leq 0.01$.
\end{proof}

\subsection{Proof of Theorem \ref{t:steep}}

We are ready to complete the proof.

\begin{proof}[Proof of Theorem \ref{t:steep}.]
Recall that $d$ is large enough,  $\eps _0=\sqrt{(\ln d) /d}$, $p=\lfloor 1/(5\eps _0)\rfloor$, and let $\eps=0.01$.
Denote $m= m_0=   \lfloor c_{\ref{graph th known}}\eps n/(2d)\rfloor$ and note that
$n/d^{3/2}\leq n_1\leq m_0\leq n_2$ and that $|K^c|\leq \LL\leq 3n_1d/5\leq 3m_0d/5$.
Below we deal with matrices from
$$
  \Omega_0 =  \Omega_{2 n_1,\eps}
 \,  \cap\,    \Omega_{2 m_0,\eps } \cap \Omega_{n_1,\eps _0}  \,  \cap\, \bigcap _{j= r_0}^{r}\, \Omega_{p^j,\eps _0}.
$$

\smallskip

If $v\in \st_0^\CC$ and $M\in \Omega_0$ then Lemma~\ref{l:T0}
implies that
$$
  \|(M-\ww \idmat)^\KK  v\|_2 \geq \frac{\LL^3\, d^{2}}{n^{6}}\, \|v\|_2.
$$

We turn now to the case $v\in \st_i^\CC$ for $i=1,2$. Let
\begin{align}
 &\Event_{i}:=\Big\{M\in\Mc\,:\,\exists\, v\in \st_i^\CC \,\, \, \mbox{such that}\,\,\,\|(M-\ww  \idmat)^\KK v\|_2\le
 \frac{\sqrt{ c_1 m d}}{2 \, b_i }\, \|v\|_2\Big\},\notag
\end{align}
where $c_1$ is the constant from Lemma~\ref{individual}, $b_1=n^{6}/(\LL^3 d^{3/2})$, and $b_2= d^{3/2} b_1=n^{6}/\LL^3$.
For a matrix $M\in \Event_{i}$ there exists $v=v(M)\in \st_i^\CC$
$$
 \|(M-\ww  \idmat)^\KK v)\|_2\le  \frac{\sqrt{ c_1 m d}}{2 \, b_i }\, \|v\|_2.
$$
Write $v=x+y$, where $x \in \st_i$ and $y \in \CC$ such that $|y_1|\leq x^*_{n_1}/10$.
Normalize $v$ so that $x_{n_{i}}^{*}=1$ (that is, $v\in \st_i'$). By Lemma~\ref{l:norma} we have
$$
  \|v\|_2=\|x+y\|_2\leq \frac{n^{6}}{100 \LL^3 d^{3/2}}\, x^*_{n_1} + \frac{\sqrt{n} \,  x^*_{n_1}}{10} \leq
   \frac{ n^{6}}{ \LL^3 d^{3/2}} \, x^*_{n_1}\leq b_i x^*_{n_i}=b_i.
$$
Let $\Net_i=\Net_i'+\Net_i''$ be the net constructed in Lemma~\ref{l:net}. Then there exist
$u\in \Net_i'$ with $u_{n_{i}}^{*}\geq 1-1/2d^{3/2}>2/3$
and $u_{j}^{*}=0$ for $j> n_{i+1}$, and $w\in \Net_i''\subset \CC$, such that for large enough $d$,
$$
 \|v-(u+w)\|_2\leq \sqrt{n}\, ( \|x-u\|_\infty+\|y-w\|_\infty)
 \leq 2 \sqrt{n} \, d^{-3/2}\leq \frac{\sqrt{ c_1 m d}}{4 d}.
$$
Therefore, using that $\|M\|=d$ and $|\ww|\leq d$, we obtain that
for every matrix $M\in \Event_{i}$ there exist $u=u(M)\in \Net_i'$ and $w=w(M)\in \Net_i''\subset \CC$ with
$$
   \|(M-\ww  \idmat)^\KK (u+w)\|_2\le \|(M-\ww  \idmat)^\KK v\|_2 + (\|M\| +|\ww|) \|v-(u+w)\|_2 \leq
  \sqrt{ c_1 m d}.
$$

Using union bound,  our choice of $n_1$, $n_2$, $n_3$, Lemma~\ref{l:net}, and
Lemma~\ref{individual} twice --
first  with $m_1=n_1<m_0$, $m_2=n_2$, $\eps = 0.01$,
then with $m_1=n_2>m_0$, $m_2=n_3$, $\eps = 0.01$  (see Remark~\ref{rem-ind}),
we obtain for small enough $\aaa$ and large enough $d$,
$$
 \p\left(\Event_{1} \cap \Omega _0\r) \leq \exp \left(-(n_1 d/2)\ln (6/5) + 7 n_2 \ln d\r) \leq
 \exp \left(- n_1 d/20 \r)
$$
and
$$
 \p\left(\Event_{2} \cap \Omega _0\r) \leq \exp \left(-c n\ln d + 7 n_3 \ln d\r) \leq
 \exp \left(- c_0 n \ln d\r),
$$
where $c_0>0$ is an absolute constant.

Combining all cases  we obtain that for $x\in \st_{\CC}$ one has
$\|(M-\ww  \idmat)^K x\|_2 \leq \beta \nx_2$, where
$$
  \beta:= \min \left(  \frac{\LL^3\, d^{2}}{n^{6}},\, \,
   \frac{\LL^3 \sqrt{ c_1 m_0 d}}{2 n^{6} } \ \r)
  \geq \frac{\LL^3 d^{2}}{n^{6}}\, \min \left(1, \,
   \frac{\sqrt{ c_1 c_{\ref{graph th known}}\eps n/2  }}{2 d^{2}  } \r)
  \geq \frac{ \LL^3 d}{ n^{6}}
$$
with probability at most
$$
  p_0:=\p\left( \Omega_0^c\r)+ \exp \left(- n_1 d/20 \r) +  \exp \left(- c_0 n \ln d\r).
$$

We now estimate the probability $p_0$.
Using Theorem~\ref{graph th known} and that $n_1\geq n/d^{3/2}$, $\eps_0^2 d= \ln d$, $\eps=0.01$,
 we obtain  for large enough $d$,
$$
   p_1:=   \p\left( \Omega_{2 n_1,\eps }^c\r) +  \p\left( \Omega_{2 n_1,\eps _0 }^c\r) + \exp \left(-n_1 d/20\r)
$$
$$
  \leq
     \exp\left(-\eps^2 d n_1/4 \right) +  \exp\left(-\eps_0^2 d n_1 /8 \right)
    + \exp \left(-n_1 d/20\r) \leq \exp \left(-n/ d^{3/2}\r) ;
$$
$$
   p_2:=   \p\left( \Omega_{2 m_0,\eps }^c\r)  + \exp \left(- c_0 n\ln  d\r)
  \leq
     \exp\left(-\eps^2 d m_0/4 \right) +  \exp \left(- c_0 n\ln  d\r)
     \leq \exp \left(- c_3 n\r) ;
$$
and
\begin{align*}
   p_3:= \sum  _{i=r_0}^{r}\p\Big( \Omega_{p^j,\eps_0}^c\Big) &\leq \sum _{i=r_0}^{r}
   \exp\Big( -\frac{p^i\ln d }{8}\ln \Big(\frac{
   n}{p^i d^{3/2}}\Big)\Big)
   \\
   &\leq \exp\Big(-\frac{p^{r_0}\ln d }{9}
   \ln \Big(\frac{n}{p^{r_0} d^{3/2}}\Big)\Big),
\end{align*}
where $c_3$ is a positive absolute constant.
Since $r_0\geq 0$ and $n\geq d^3$, we have
$$
   p_3\leq \exp\left(-\frac{\ln d }{9}
   \ln \left(\frac{n}{ d^{3/2}}\right)\right)
   \leq \exp\left(-\frac{(\ln d)\ln n }{18}\r).
$$
 Since $p^{r_0}\geq 20\LL/d$,
we also have $p_3\leq \exp(-2 \LL (\ln d/d))$.
Since $p_0\leq p_1+p_2+ p_3$, the desired estimate follows.

\end{proof}

\subsection{Almost constant vectors}
\label{subs: almost constant}

Given $\theta>0$,  we introduce a class of almost constant vectors by
$$
  \BB(\theta) =\{ x\in \C^n\, : \, \exists \lambda \in \C \,  \mbox{ such that }\,
   |\{ i\leq n \, : \, |x_i - \lambda|\leq \theta\, x_{\nn}^{*}\}| >n- \nn\}.
$$
Note that this class slightly differs from the class considered in \cite{LLTTY first part} --
there we compared the error in terms of $\|x\|_2$ instead of $x_{\nn}^{*}$.

\begin{rem}\label{lambda-0}
 Let $x\in \BB(\theta) $.
Fix a permutation $\sigma=\sigma_x$ of $[n]$ such that $x_i^*=|x_{\sigma(i)}|$ for $i\le n$.
Fix $\lam _0 =\lam _0(x) \in \C$ such that the cardinality of
$$
   J_1:= \{ i\leq n \, : \, |x_i - \lambda_0|\leq \theta\, x_{\nn}^* \}
$$
is at least  $n- \nn+1$.
Then there exist positive integers $k, \ell$ with $k\leq \nn <\ell$ such that $\sigma(k), \sigma(\ell) \in J_1$  and
$$
 x_{k}^* - \theta  x_{\nn}^*    \leq  |x_{\sigma(k)}| -  |  x_{\sigma(k)} - \lambda_0|
     \leq  | \lambda_0|\leq | \lambda_0 - x_{\sigma(\ell)}| + |x_{\sigma(\ell)}| \leq
    \theta  x_{\nn}^* +   x_{\ell}^* ,
$$
 which implies
\begin{equation}\label{rhon3}
    (1-\theta)    x_{\nn}^*
     \leq  | \lambda_0| \leq   (1+\theta)  x_{\nn}^* .
\end{equation}
\end{rem}

Define $n_0=\lfloor n/16 d\rfloor$.
Given $t>0$,  consider the following class of vectors
$$
  S (t) : =\{x\in \C^n\,:\,  0< x_{n_{0}}^{*}\leq  t x_{\nn}^{*}\}.
$$

The proof of the next lemma is similar to that of Theorem~3.1
from \cite{LLTTY first part}. We provide it at the end of the section
for the sake of completeness.

\begin{lemma}\label{lem: a-c}
Let  $\theta \in (0, 1/20]$ and $t\geq 12$ be such that $\aaa t \leq 1/100$.
Let $\KK\subset [n]$ with $|\KK^c|\leq n/4$  and $\ww\in \C$ with $|\ww |\leq d/5$.
Then for every $x\in \BB(\theta)\cap S(t)$ and every $M\in \Mc$ one has
$$
  \| (M - \ww \idmat)^\KK x\|_2 \geq \frac{d  \sqrt{n}}{2 \sqrt{2}}\, x^*_{\nn}.
$$
\end{lemma}

We also need the following simple lemma about almost constant vectors not covered
by Lemma~\ref{lem: a-c}.

\begin{lemma}\label{lem: splitt}
Let $d\geq 3$,  $0<\theta \leq 10/d^3$ and $t\geq 12$ be such that $\aaa t\leq 1/100$. Then
 every  $x\in \BB(\theta)\setminus S(t)$ can be represented as $x=w+y$ with
$w\in \st $ and $y\in \CC$ with $|y_1|\leq w^*_{n_1}/10$.
\end{lemma}

\begin{proof} Fix $x\in \BB(\theta)\setminus S(t)$. Then $x_{n_0}^*> t x^*_{\nn}$.
Let $\sigma$, $J_1$, and $\lambda _0$ be as in Remark~\ref{lambda-0}.
Consider $y=(\lambda _0, \lambda _0, ..., \lambda _0)\in \CC$ and $w=x-y$.
Since $x \not \in  S(t)$ and by (\ref{rhon3}),  for every $i\leq n_0$ we have
$$
  |w_{\sigma(i)}| \geq |x_{\sigma(i)}| - |\lambda _0| \geq x^*_{n_0} - (1 +\theta) x^*_{\nn}
  > (t-1-\theta) x^*_{\nn} >10  x^*_{\nn}.
$$
This implies $w_{n_0}^* > 10  x^*_{\nn}.$ On the other hand, for every $i \in J_1$, one has $|w_i|\leq \theta x^*_{\nn}$.
Since $|J_1|> n-\nn$, this implies $w_{\nn}^* \leq \theta x^*_{\nn}$. Using that $\theta \leq 10/d^3$, we obtain
$$
  w_{n_1}^*\geq w_{n_0}^* > d^3 w^*_{\nn},
$$
which shows $w\in \st$.

Using again that $x \not \in  S(t)$ and the inequality (\ref{rhon3}),
we observe that  for every $i\leq n_0$,
$$
  |w_{\sigma(i)}| \geq |x_{\sigma(i)}| - |\lambda _0| \geq |x_{\sigma(n_0)}|
  - |\lambda _0|\geq t x^*_{\nn} - |\lambda _0|
  > (t/(1+\theta) -1)  |\lambda _0|>10 |\lambda _0|,
$$
which implies $|y_1| =  |\lambda _0|\leq w^*_{n_1}/10$ and completes the proof.
\end{proof}

As a consequence of Theorem~\ref{t:steep} and Lemmas~\ref{lem: a-c},
\ref{lem: splitt}, and \ref{l:norma}, we obtain the main theorem of this section.

\begin{theor}\label{th: steep and ac}
Let   $\aaa  \leq 1/1200$, $d\geq 1$ be large enough,
$n\geq d^3$, $1\leq \LL\leq n/d^3$, and $0<\theta \leq 10/d^3$.
Let $\KK\subset [n]$ with $|\KK^c|\leq \LL$  and $\ww\in \C$ with $|\ww |\leq d/5$.
Then with probability at least
$$
  1- \min(\exp(-\LL/d), \exp(-(\ln d)(\ln n)/20),
$$
one has that for
every $x\in (\BB(\theta)\setminus \st_{3}^\CC) \cup \st _\CC$
$$
  \| (M - \ww \idmat)^\KK x\|_2 \geq \frac{ \LL^3 d}{n^{6}}\, \|x\|_2.
$$
\end{theor}

\begin{proof}
Fix $t=12$. Fix $x\in \BB(\theta)\setminus \st_{3}^\CC$. If $x\in \st_\CC$ then the result follows by
Theorem~\ref{t:steep}. Therefore we assume that $x\not\in  \st_\CC\cup \st_{3}^\CC$. Then,
in particular, $x\not\in  \st_0\cup \st_{3}$ and $x\not\in  \st_1\cup \st_2$, hence,
by Lemma~\ref{l:norma} we have
$$
   x_{\nn}^*\geq x_{n_1}^*/d^3 \geq \frac{100\LL^3}{n^{6}  d^{3/2}}\, \|x\|_2.
$$
Since $n\geq d^3$, this and Lemma~\ref{lem: a-c} implies the case when $x\in S(t)$. Note that Lemma~\ref{lem: splitt}
says that
$$
   \BB(\theta) \setminus S(t)\subset \st_\CC\cup \st_{3}^\CC,
$$
therefore we are done.
\end{proof}

\begin{proof}
[Proof of Lemma~\ref{lem: a-c}.]
Since $x\in S(t)$, we have $x_{\nn}^*\ne 0$.
Let $\sigma$, $J_1$, and $\lambda _0$ be as in Remark~\ref{lambda-0} and set
$$
  J_2=\sigma([n_{0}])\setminus J_1 , \quad  J_3 = \sigma([\nn])\setminus (J_1\cup J_2), \quad \mbox{ and }
  \quad J_4=[n]\setminus  (J_1\cup \sigma([\nn])).
$$
Then $|J_3|, |J_4|\leq \nn$, $[n]=J_1 \cup J_2\cup J_3\cup J_4$, and
\begin{equation}\label{condonJ}
\forall j\in J_4 \, \, \, \, |x_j|\leq x_{\nn}^*
 \quad \mbox{ and } \quad
\forall j\in J_3 \, \, \, \, |x_j|\leq x_{n_0}^*\leq  t x_{\nn}^*
\end{equation}

\smallskip

 Now, given a matrix $M\in\Mc$, consider
$$
 I_2 = \{ i\leq n \, : \, \supp R_i (M)\cap J_2  \ne \emptyset \} \, \mbox{ and }\,
 I_\ell = \{ i\leq n \, : \, |\supp R_i (M)\cap J_\ell| \geq   16 \nn d/n \},
$$
for $\ell = 3,4$.
Since $M\in \Mc$, we have $|I_2|\leq d\, n_{0}$ and $(16\nn d/n) |I_\ell| \leq d | J_\ell |$, hence
$$
 |I_2|\leq n/16  \quad  \quad \mbox{ and } \quad \quad
  | I_\ell | \leq n/16 \, \, \mbox{ for } \, \, \ell =3,4.
$$
Set $I:=[n]\setminus(I_2\cup I_3\cup I_4\cup \sigma([\nn])\cup K^c)$. Then for small enough $\aaa$,
$$
  |I| \geq  n - 3 n /16 -  \nn -n/4 \geq   n / 2 \quad \mbox{ and } \quad \forall i\in I \, \, \, \,
  |x_i|\leq x^*_{\nn} \leq |\lam _0|/(1-\theta).
$$
Moreover, for every $i\in I$, denote $\overline{J _\ell}=\overline{J _\ell} (i) = J_\ell \cap \supp R_i (M)$
for $1\leq \ell\leq  4$, and  note that $\overline{J _2}=\emptyset$ since $i\not\in I_2$.
Using the triangle inequality, we observe for every $i\in I$,
$$
|\langle R_i(M-\ww \idmat), x^\dagger \rangle|
\geq \Big|\sum _{j \in  \overline{J_1}}  x_j\Big| - \sum _{j \in \overline{J_3}} |x_j| - \sum _{j \in \overline{J_4}} |x_j|
  - |\ww x_i|.
$$
We estimate each of the terms on the right hand side separately. By the definition of $J_1$, we have
$$
\Big|\sum _{j \in  \overline{J_1}}  x_j\Big|\geq |\lam _0|\, |\overline{J _1}|- \sum _{j \in  \overline{J_1}} \Big| x_j-\lambda_0\Big|\geq |\overline{J _1}| \, ( |\lam _0|- \theta x^*_{\nn})
\geq (d - 32 \nn d/n)\, (1-2\theta)\, x^*_{\nn},
$$
where for the last inequality we used (\ref{rhon3}) and that for $i\not\in I_2\cup I_3\cup I_4$ one has
$$
   |\overline{J _1}|= d-   |\overline{J _2}|-|\overline{J _3}|- |\overline{J _4}|\geq d - 32 \nn d/n.
$$
Using (\ref{condonJ}),
$$
  \sum _{j \in  \overline{J_3}}  |x_j|+ \sum _{j \in  \overline{J_4}} | x_j|\leq |\overline{J _3}|\, x_{n_0}^* + |\overline{J _4}|\, x_{\nn}^*\leq 16 (1+t) \nn d x^*_{\nn}/n.
$$
Putting together the above estimates, we obtain for large enough $d$
\begin{align*}
    |\langle R_i(M-\ww \idmat), x^\dagger \rangle|
  &\geq \left((d - 32 \nn d/n)(1-2\theta) - 16 (1+t) \nn d /n   -  | \ww|\r) x^*_{\nn}
  \\
  &\geq \left(1-2\theta - 16 \aaa (3+t)    -  | \ww|/d\r) d x^*_{\nn}\geq d x^*_{\nn}/2,
\end{align*}
where we used  $\theta \leq 1/20$, $t+3\leq 5t/4$, $\aaa t\leq 1/100$, and $\nn/n\leq \aaa$,
and $|\ww|\leq d /5$. This implies
$$
   \| (M - \ww \idmat) x\|_2 \geq  \frac{d x^*_{\nn}}{2}\, |I|^{1/2} \geq \frac{d x^*_{\nn}}{2} \, \sqrt{\frac{n}{2}},
$$
and completes the proof.
\end{proof}

\section{Gradual vectors}
\label{s:k-vektors}

In this section we introduce the notion of $k$-vectors,
which provide a discretization of the set of gradual vectors, and discuss their properties.
We will use notations of Section~\ref{steep}, in particular, $\eps_0$, $p$, $r$,
$n_1$, $n_2$, and $\nn$.

\smallskip

We first define the set of gradual vectors as the set of all vectors which are not almost
constant and not steep. Note that any
gradual vector $x$ satisfies $x_{\nn}^{*}\neq 0$.
We will use the following normalization of
gradual vectors,
$$
  \mathcal S:=\bigl\{x\in\C^n\setminus (\st\cup \BB)\, : \, x_{\nn}^{*}=1\bigr\},
$$
where $\BB=\BB(\theta_0)$ with $\theta_0=10/d^3$ (the set $\BB(\theta)$
was introduced at the beginning of  Section~\ref{subs: almost constant}).
Note that, by the definition of the almost constant vectors, we have for any $x\in\mathcal S$
that
$$\forall\;\lambda\in\C\;\;\;|\{ i\leq n \, : \, |x_i - \lambda|\leq \theta_0\}| \leq n- \nn,$$
and by the definition of the steep vectors,
\begin{align*}
  &\forall \,0\le i\le r_0\, : \, \, \quad x^*_{p^i}\le (n/p^i)^3 (4d)^{r-r_0+1}d^3, \\
  &\forall \,r_0< i\le r\, : \, \,  \quad  x^*_{p^i}\le  (4d)^{r-i+1}d^3,\\
 & x^*_{\lceil n_1/p\rceil}\le 4d^4,\quad  x^*_{n_1}\le d^3, \quad \mbox{ and }
  \quad x^*_{n_2}\le d^{3/2}.
\end{align*}

\subsection{Gradual $k$-vectors}

For every positive integer $k$ we define  {\it $k$-vectors} as  vectors in $\C^n$ with
coordinates taking values in the set $\Z^2/k=\{\omega/k:\, \omega\in\Z^2\}$.
Let $x=(x_1,x_2,\dots,x_n)\in\C^n$ and $k\in\N$.
The {\it $k$-approximation} of $x$ is defined as the $k$-vector $y\in\C^n$
such that $\Re\, y_i=\lfloor k \Re\, x_i\rfloor/k$ and $\Im\, y_i=\lfloor k \Im\, x_i\rfloor/k$ for all $i\leq n$.
Clearly, $\|x-y\|_\infty\leq\sqrt{2}/k$.

\smallskip

Below we split gradual vectors into classes of vectors, such that every pair of vectors from a
given class has the same coordinates up to some permutation. We formalize it as follows.
Let $x=\{x_i\}_i \in\C^n$. By $x^\sharp=\{x_i^\sharp\}_i$ denote the
vector $(x_{\sigma(1)},x_{\sigma(2)},\dots,x_{\sigma(n)})$, where the
permutation $\sigma$ is chosen so that
$x_{\sigma(1)}\geq x_{\sigma(2)}\geq\ldots\geq  x_{\sigma(n)}$
in the sense of lexicographical order introduced in Section~\ref{preliminaries}.
Recall that $x^{*}=\{x_i^*\}_i$ denotes the non-increasing
rearrangement of $\{|x_i|\}_i$. Consider the following subset of ``normalized'' $k$-vectors,
$$
\kset_k :=  \big\{y\in\C^n:\,y\mbox{ is a $k$-approximation of a vector in $\mathcal S$}\big\}.
$$
Observe that for every $y\in \kset_k$,
\begin{equation} \label{norm-simp}
1-\sqrt{2}/k \leq y_{\nn}^{*}\leq 1+\sqrt{2}/{k}.
\end{equation}
 Next consider the equivalence relation on $\kset_k$ defined by
 $x\stackrel{\sharp}{\sim}y$ iff $x^\sharp=y^\sharp$ for two $k$-vectors $x$ and $y$.
This relation partitions the set $\kset_k$ into the equivalence classes.
We first estimate how many classes we have.

\begin{lemma}\label{l:equiv classes}
Let $d\leq n^{1/3}$ be large enough and $1\leq k\leq \sqrt{n}/d^{3/2}$. Then
the  number of the equivalence classes (with respect to the relation $\stackrel{\sharp}{\sim}$) in $\kset_k$ does not exceed $e^n$.
\end{lemma}

\begin{proof}
From every equivalence class choose exactly one representative
$x$, satisfying $|x_1|\geq |x_2|\geq \ldots \geq |x_n|$, multiply it by $k$ and consider
the set $\kset _k'$ of such elements. Note that by definitions every element of
$\kset _k'$ has integer coordinates and,  moreover,
$\kset _k' \subset k \kset _k$.

\smallskip

Define a partition of $[n]$ into following $r+4$ sets.
Let $I_0=[n]\setminus [\nn]$. Set
$I_{1}=[p]$.
Then for every $1< i\leq r$, set $I_i=[p^i]\setminus [p^{i-1}]$.
Finally, set
$$
  I_{r+1}=[n_1]\setminus [p^r], \quad \quad I_{r+2}=[n_2]\setminus [n_1], \quad \quad \mbox{ and }\quad \quad I_{r+3}=[\nn]\setminus [n_2].
$$
The cardinalities of $I_i$'s, $0\leq i\leq r+3$, we denote by $N_i$'s.
Clearly, $N_0\leq n$, $N_{r+j}\leq n_j$ for $j=1,2,3$, and $N_{i}\leq p^i$ for  $1\leq i\leq r$.

By the normalization of vectors in $\mathcal{S}$ and by (\ref{norm-simp}), for every $x\in k \kset _k$,
we have $x^*_{\nn}\leq k+\sqrt{2}\leq 2.5 k$. Therefore, by
the definition of gradual vectors we have that for every $x\in \kset _k'$ and every $r_0\leq i\leq r$,
\begin{equation}\label{besknorm}
  x^*_{\nn}\leq 2.5 k, \, \, \, \, \, \, x^*_{n_{2}}\leq 2.5 k d^{3/2},\, \, \, \,  \, \,
  x^*_{p^{r+1}} \leq x^*_{n_{1}}\leq 2.5 k d^3,\, \, \, \,  \, \,   x^*_{p^{i}} \leq 2.5 k d^3 \, (4d)^{r+1-i}
\end{equation}
  and, using that $n\leq n_1 d^{3/2}\leq p^{r+1} d^{3/2}$ and $p^2 \leq d$, for $0\leq i<r_0$,
\begin{equation}\label{besknorm-1}
 x^*_{p^{i}} \leq  2.5  k (n/p^i)^3 x^*_{p^{r_0}}\leq  2.5  k p^{3(r+1-i)} \, d^{4.5} \, (4d)^{r+1-r_0} \leq
 2.5  k d^{4.5} \, (4d)^{2.5(r+1-i)}.
\end{equation}

For $0\leq i\leq r+3$ let $\nu_i$ be the number of possible distinct coordinates of the projection of $y\in \kset_k'$ on $\C^{I_i}$.
Recall that every element of $\kset_k'$ has integer coordinates. Note that if a complex number $z=a+ \iii b$ with integer $a$ and $b$
satisfy $|z|\leq  A$ for some  $A\geq 2.5$ then $-A\leq a, b \leq A$, so there are at most $(2A+1)^2\leq 6 A^2 $ such numbers $z$.
Therefore, by (\ref{besknorm}) and (\ref{besknorm-1}), we have for $1\leq i \leq r+1$,
$$
 \nu_0\le 40 k^2  ,\quad \nu_{r+3} \le 40 k^2 d^3 , \quad
 \nu_{r+2} \le 40 k^2 d^6,
 \quad \mbox{ and }\quad
 \nu_i\le 40 k^2 d^9 \, (4d)^{5(r+2-i)}.
$$

The number of sequences $\{x_i\}_{i=1}^N$ in $\C^N$ taking values in a set of cardinality $\nu$,
where we don't distinguish between sequences which can be obtained one from another by a permutation,
equals ${N+\nu-1 \choose N}$ (indeed, after introducing an order, this corresponds
to the number of  non-increasing sequences $\{y_i\}_{i=1}^N \subset [\nu]$ and we can pass
to the strictly decreasing sequences $\{z_i\}_{i=1}^N \subset [\nu+N-1]$, where $z_i=y_i+N-i$,
hence this number is the same as the number of $N$ elements subsets of $[\nu+N-1]$).
 This leads to
$$
 |\kset'_k| \le
 \prod _{i=0}^{r+3}  {N_i+\nu_i -1 \choose N_i}.
$$

Using bounds for $\nu_i$ and $N_i$, the standard estimate ${m\choose \ell} \leq (e m/\ell)^\ell$,
and that $d^3 k^2\leq  n$, $p^{r_0}\leq n/d^{3/2}$,
we get
$$
  B_1:=  {N_0+\nu_0 -1\choose N_0} \leq {N_0+\nu_0 \choose \nu_0} \leq  {n+40 k^2\choose 40 k^2} \leq
  {n+ \lfloor 40 n/d^3\rfloor \choose \lfloor 40 n/d^3\rfloor} \leq
 \left(e d^3/20 \right)^{40 n/d^3};
$$
$$
  B_2:=  {{N_{r+2}}+\nu_{{r+2}} -1 \choose {N_{r+2}}}\leq  {n_2 + 40 k^2 d^{6} \choose n_2}\leq
  \left(\frac{41 e n d^3}{n_2}\right)^{n_2}
  \leq   d^{ 4 n/d^{2/3}};
$$
$$
  B_3:=  {{N_{r+3}}+\nu_{{r+3}} -1 \choose {N_{r+3}}-1}
  \leq  {\nn + 40 k^2 d^3\choose \nn} \leq (41 e /\aaa )^{\aaa n} ;
$$
and, for $1\leq  i\leq r+1$,
$$
 B_{4, i}:=   {N_i+\nu_i-1 \choose N_i}
 \leq
  {p^i +  40 k^2  d^9 \, (4d)^{5(r+2-i)}\choose p^i}.
$$
If $p^i>40 k^2 d^9 \, (4d)^{5(r+2-i)}$ then $B_{4, i}\leq 4^{p^i}$, otherwise, using
$$
   k^2 d^3  p^{-i} \leq   n p^{-i} \leq   n_1 d^{3/2} p^{-i} \leq   p^{r+1-i} d^{3/2},
$$
 we have
$$
 B_{4, i} \leq \left(\frac{80 e k^2  d^9 \, (4d)^{5(r+2-i)}}{p^i}\right)^{p^{i}}
   \leq\left(  d^8 \, (4d)^{6(r+2-i)} \right)^{p^{i}}.
$$

Denoting $B_4=\prod _{i=1}^{r+1} B_{4,i}$, using that $d$ is large enough, and passing to sums of logarithms,
we have
\begin{align*}
 \ln B_4 &\leq \sum _{i=1}^{r+1} p^i\,  \ln \left( d^8 \, (4d)^{6(r+2-i)} \right) \leq
  \sum _{\ell=1\atop (\ell=r+2-i)}^{r+1}  p^{r+2-\ell} \, \left(6 \ell\, \ln (4d) + 8 \ln d \right)\\
 &\leq 20 p^{r+1} \ln d  \leq 20 p n_1 \ln d \leq  n (\ln d )/d.
\end{align*}
Combining all bounds we obtain
$$
   |\kset _k ' |\leq B_1 B_2 B_3 B_4\leq
e^n,
$$
provided that $\aaa$ is small enough and  $d$ is large enough.
\end{proof}

\subsection{The $\ell$-decomposition with respect to $k$-vectors}\label{subs: ell decomposition}

In this subsection, we introduce one of the most important technical ingredients of the paper --
the $\ell$-decomposition with respect to $k$-vectors, which is a special way to structure
a $k$-vector $y$ as a collection of two-dimensional ``stairs'' or ``ladders'' which ultimately determine the anti-concentration properties of the product $My$ (with a random matrix $M$ uniformly distributed in $\MSet_{n,d}$).

Let $y=(y_i)_{i=1}^n \in \C^n$ be a $k$-vector. We will construct a partition of [n] into
two sequences of subsets of $[n]$, $(\ls_j(y))_{j=0}^\infty$ and $(\lr_j(y))_{j=0}^\infty$,
which we call {\it spread $\ell$-parts} and {\it regular $\ell$-parts}, respectively.
Note that all but a finite number of the subsets are empty. When the vector $y$ is clear from
the context, we will simply write $\ls_j$ and $\lr_j$ for the corresponding $\ell$-parts.

Our construction consists of a series of steps (indexed by $j$), and each step
comprises a sequence of substeps.
At $j$-th step (except $j=0$), we already have sets $(\ls_u)_{u=0}^{j-1}$ and $(\lr_u)_{u=0}^{j-1}$
constructed. If $j=0$ set $I_0:=[n]$ and  $\Lambda_0:=\{y_i:\,i\in [n]\}$, otherwise
set
$$
  I_j:=[n]\setminus \big(\bigcup_{u\leq j-1}\ls_u\,\cup\,\bigcup_{u\leq j-1}\lr_u\big)\quad\mbox{and}\quad
  \Lambda_j:=\{y_i:\,i\in I_j\}.
$$
Now, for each $\lambda\in \Lambda_j$ such that
$|\{i\in I_j:\,y_i=\lambda\}|< 2^{j+1}$ we let
$$
  L(j,\lambda):=\{i\in I_j:\,y_i=\lambda\},
$$
and for every $\lambda\in \Lambda_j$ with
$|\{i\in I_j:\,y_i=\lambda\}|\geq 2^{j+1}$
we let
$L(j,\lambda)$ be the subset of $\{i\in I_j:\,y_i=\lambda\}$ of cardinality $2^{j}$
such that
$$
  L(j,\lambda)= I_j\cap [1,\, \sup L(j,\lambda)]
$$
(that is, we choose $L(j,\lambda)$ as the ``leftmost''  subset of cardinality $2^{j}$).
Note that by construction for $j\geq 0$ we have
\begin{equation}\label{levset}
2^{j-1}\leq |L(j,\lambda)|< 2^{j+1}.
\end{equation}
We  refer to sets $(L(j,\lambda))_{\lambda\in \Lambda_j}$ as {\it level sets of order $j$}
(with respect to $y$).
The union of the level sets of order $j$ will form the spread and regular parts, $\ls_j$
and $\lr_j$, i.e., we define $\ls_j$ and $\lr_j$  so that
$$
  \ls_j\cup\lr_j = \bigcup\limits_{\lambda\in \Lambda_j}L(j,\lambda).
$$
To separate the spread part from the regular one of the same order, we apply an
embedded procedure consisting of substeps.
Our construction of spread vectors is based on extracting a maximal $(d/k)$-separated set
subset from $\Lambda_j$, consisting of at least $2$ elements, provided that such a set exists.
Note, that we need to have at least $2$ elements to be able to apply anti-concentration later.
We  construct a subset $\Lambda_j^S\subset \Lambda_j$ as follows.

\smallskip

\noindent {\it Substep $1$.}
If the diameter of $\Lambda_j$ is strictly less than $d/k$
then we set $\Lambda_j^S:=\emptyset$ and terminate.
Otherwise, note that there is at least one pair of numbers $\lambda,\lambda'\in \Lambda_j$ such that $|\lambda-\lambda'|\geq d/k$.
Define $\lambda_1$ as the {\it largest} (with respect to the lexicographical order,
see Section~\ref{preliminaries}) number in $\Lambda_j$ such that $|\lambda_1-\lambda'|\geq d/k$
for some $\lambda'\in \Lambda_j$ and pass to the next substep.

\smallskip

\noindent
{\it Substep $m$ ($m>1$).}
We have already chosen numbers
$\lambda_1,\lambda_2,\dots,\lambda_{m-1}$ in $\Lambda_j$.
If all $\lambda\in \Lambda_j$ are within a distance strictly less than $d/k$ to $\{\lambda_1,\lambda_2,\dots,\lambda_{m-1}\}$
then set $\Lambda_j^S:=\{\lambda_1,\lambda_2,\dots,\lambda_{m-1}\}$ and terminate (note,
by the construction, this cannot happen if $m=2$).
Otherwise, let $\lambda_m$ be the largest number in $\Lambda_j$ with the distance to
$\{\lambda_1,\lambda_2,\dots,\lambda_{m-1}\}$ greater or equal to $d/k$
and go to the next substep.

\smallskip

\noindent
Note that by construction we have that the sequence $(\lambda_m)_{m\geq 1}$ is decreasing (with respect to the lexicographical order) and, moreover,
$\vert \lambda_u-\lambda_{v}\vert\geq d/k$ for every admissible $u\neq v$.
Now, by {\it the spread $\ell$-part of order $j$} with respect to $y$, we call the union
$$\ls_j=\ls_j(y):=\bigcup_{\lambda\in \Lambda_j^S}L(j,\lambda)$$
and by {\it the regular $\ell$-part of order $j$} with respect to $y$, we call the union
$$
  \lr_j=\lr_j(y):=\bigcup_{\lambda\in \Lambda_j\setminus\Lambda_j^S}L(j,\lambda).
$$

The {\it height} $h(\cdot)$ of a regular (resp, spread)
$\ell$-part is the number of level sets it comprises (if the $\ell$-part is empty then $h=0$).
In particular, by (\ref{levset}), if $\lset_j$ is either $\ls_j$ or $\lr_j$, then
\begin{equation}\label{levsethei}
2^{j-1} h(\lset_j )\leq |\lset_j|\leq  2^{j+1} h(\lset_j).
\end{equation}
Note also that by the construction the height of a non-empty spread part is at least $2$.
We will often write $\lset$ to denote an $\ell$-part (of some order) with respect to $y$.
Note also that the maximal number of steps (starting with the step $j=0$) that we can have
is the smallest $j+1$ such that $n<2^{j+1}$, i.e. $j+1 = \lceil \log _2 n\rceil < 1.5 \ln n$
for large enough $n$. Therefore, the number of non-empty $\ell$-parts, denoted below
by $m(y)$ is at most $3\ln n$.

Finally we introduce the $\ell$-decomposition. Let $y$ be a $k$-vector with the corresponding
$\ell$-parts $\{\ls_j,\lr_j\}_{j\geq 0}$. We will re-enumerate the non-empty spread and
regular $\ell$-parts and will write $(\lset^{(q)})_{q=1}^m$ (i.e., suppressing the order
and spreadness/regularity), where $m=m(y)\leq 3\ln n$.
To make this representation  unique, we assume that within the sequence
$(\lset^{(q)})_{q=1}^m$, any spread $\ell$-part precedes (by the index) any regular $\ell$-part,
and that for any two spread (resp. regular) parts, the one of smaller order precedes the other.
In what follows, such a sequence will be called the {\it $\ell$-decomposition} with respect to $y$.

Below, given a level set $L\subset [n]$, i.e., a set of coordinates where $y_i$ preserves its value, we
denote this value by $y(L)$.

\medskip

To clarify our construction we would like to provide the following example.

\smallskip

\noindent {\it Example.} Let $n=7$, $d=2$, $k=6$. Consider $y=(1/2,1/3,1/2,1/6,1/2,1/3,-1/3)$.
Note that $y$ is a $k$-vector.
According to the above procedure, at step $j=0$ we have $\Lambda_0=\{1/2,1/3,1/6,-1/3\}$ and construct level sets
$L(0,1/2)=\{1\}$, $L(0,1/3)=\{2\}$, $L(0,1/6)=\{4\}$, $L(0,-1/3)=\{7\}$.
Since $d/k=1/3$, we get that $\Lambda_0^S=\{1/2, 1/6, -1/3\}$.
Thus $\{1, 4, 7\}$ is the spread $\ell$-part of order $0$,
and $\{2\}$ is the regular $\ell$-part of order $0$. At step $1$, we have $\Lambda_1=\{1/2,1/3\}$ and construct level sets
$L(1,1/2)=\{3,5\}$ and $L(1,1/3)=\{6\}$. Then $\Lambda_1^S=\emptyset$, therefore $\emptyset$
 is the spread $\ell$-part of order $1$, and
$\{3,5,6\}$ is the regular $\ell$-part of order $1$.
Altogether, we have $m(y)=3$ non-empty $\ell$-parts --
one spread $\ell$-part of order $0$ with the height $3$,
one regular $\ell$-part of order $0$ with the height $1$, and
one regular $\ell$-part of order $1$ with the height $2$.
The  $\ell$-decomposition with respect to $y$ is
$(\{1, 4, 7\}, \, \{2\},\,  \{3,5,6\})$.

\medskip

A quick analysis of the construction procedure for the $\ell$-parts gives the following properties,
which we summarize into three lemmas. We leave the (rather straightforward) proofs to the reader.

\begin{lemma}\label{l: level sets basic}
Let $y$ be a $k$-vector,  $\lambda\in\Z/k$ and set $I=\{i\leq n:\,y_i=\lambda\}$.
Assume that $I\ne \emptyset$  and denote $u:=\big\lfloor\log _2((|I|+1)/3)\big\rfloor$.
Then
$$
  I=\bigcup_{j=0}^{u+1} L(j,\lambda),
$$
$$
  2^u\leq |L(u+1,\lambda)|=|I|-2^{u+1}+1\leq 2^{u+2}-1,
 \quad
  \mbox{ and } \quad \forall\, 0\leq j\leq u\, : \, \,  |L(j,\lambda)|=2^j.
$$
\end{lemma}

\smallskip

\begin{lemma}\label{l: ell parts monotone}
Let $y$ be a $k$-vector, let $j\geq 1$ and assume that $\ls_j\cup\lr_j\neq \emptyset$.
Then for all $0\leq m<j$ we have $\ls_{m}\cup\lr_{m}\neq \emptyset$,
$$
  h(\ls_{m})+h(\lr_{m})\geq h(\ls_j)+h(\lr_j),
$$
and
$$
 \{y_i:\,i\in \ls_{j}\cup\lr_{j}\}\subset\{y_i:\,i\in \ls_{m}\cup\lr_{m}\}.
$$
\end{lemma}

\smallskip

\begin{lemma}\label{l: elem properties of ell}
Let $y$ be a $k$-vector and let $\ls_j$, $\lr_j$, $j\geq 0$,
be its $\ell$-parts. Then
\begin{itemize}
\item The height of every non-empty spread $\ell$-part is at least $2$.
\item For every non-empty spread $\ell$-part $\ls$ and any $i_1,i_2\in\ls$ with $y_{i_1}\neq y_{i_2}$
we have $|y_{i_1}-y_{i_2}|\geq d/k$.
\item If $\widetilde y$ is a permutation of the vector $y$ then necessarily the $\ell$-parts of $y$ and
$\widetilde y$ agree up to a permutation of $[n]$; in particular, the heights and cardinalities of spread or regular
$\ell$-parts of a given order with respect to $y$ and $\widetilde y$ are the same.
\end{itemize}
\end{lemma}
The last property implies that with every equivalence class $\mathcal C\subset\kset_k$ and every $j\geq 0$,
we may associate four integers by fixing (an arbitrary) $y\in \mathcal C$ and setting
$$
  \cards_j(\mathcal C) := |\ls_j(y)|, \;
  \cardr_j(\mathcal C) := |\lr_j(y)|, \;
  \heis_j(\mathcal C):=h(\ls_j(y)), \;
  \heir_j(\mathcal C):=h(\lr_j(y)).
$$

The following lemma allows to estimate cardinalities of equivalence classes in terms of these
quantities.

\begin{lemma}\label{equiv cardinality}
Let $k\geq 1$, $\mathcal C$ be an equivalence class in $\kset_k$ with respect to the relation
$\stackrel{\sharp}{\sim}$.
Then the cardinality of the class $\mathcal C$ can be estimated as
$$
 |\mathcal C|\leq n!
 \prod_{j=0}^\infty \frac{{\heis_j}^{\cards_j}\,{\heir_j}^{\cardr_j}}{\cards_j!\,\cardr_j!},
$$
where we adopt the notation $0^0=1$.
\end{lemma}
\begin{proof}
There are clearly $n!/\prod_{j=0}^\infty {\cards_j!\cardr_j!}$ ways to ``assign'' $\ell$-parts to
specific locations within $[n]$.
Fix for a moment $j\geq 0$ with $\cards_j\neq 0$ and let $\ls$ be a
fixed subset of $[n]$ of cardinality $\cards_j$. Recall that $P_{\ls} (y)\in\C^{\ls}$ denotes the
coordinate projection of $y$ onto $\C^{\ls}:=\spn\{e_i:\,i\in\ls\}$. Consider the set
$$
  W_{\ls}:=\bigl\{P_{\ls} (y):\,\mbox{$y\in\mathcal C\, $ is such that $\, \ls_j(y)=\ls$}\bigr\}
$$
Since all vectors within a given equivalence class share the same {\it levels},
the cardinality of $W_{\ls}$ can be estimated from above by ${\heis_j}^{\cards_j}$.
Similarly, we can estimate the number of realizations of regular $\ell$-parts.
Combining this with the estimate for ``location assignments," we obtain the desired bound.
\end{proof}

\subsection{Decomposition of the set of gradual vectors}
\label{s:decomposition of S}

In this subsection, we define a way to partition the set of gradual vectors $\mathcal S$
in terms of structure of their $k$-approximations.
Roughly speaking, we will observe the following dichotomy for a vector $x$ in $\mathcal S$:
either $x$ possesses a $k$-approximation $y$ (for a relatively small $k$) whose $\ell$-decomposition contains
many spread $\ell$-parts (that is, the distance between the ``stairs'' in a graphical representation of $y$
is often large), or, for an appropriately chosen $k$, the $k$-approximation of $x$ contains $\ell$-parts
with large heights.

\smallskip

Given integer $u\geq 0$ we introduce two subsets of $\mathcal S$,
\begin{align*}
\kapset_u:=\bigl\{&x\in \mathcal S:\,\mbox{in the $\ell$-decomposition with respect to the $d^u$-approximation of $x$,}\\
&\mbox{the total cardinality of the spread $\ell$-parts is at least }c_\kapset \nn\bigr\}
\end{align*}
and
\begin{align*}
\rhoset_u:=\bigl\{&x\in \mathcal S:\,\mbox{in the $\ell$-decomposition with respect to the $d^u$-approximation of $x$,}\\
&\mbox{the total cardinality of spread and regular $\ell$-parts with heights not smaller} \\
&\mbox{than } c_\rhoset 2^{c_\rhoset (u-4)\aaa} \aaa \mbox{ is at least }c_\rhoset \nn \bigr\}.
\end{align*}
Here, by ``total cardinality'' we mean the cardinality of the union of the respective $\ell$-parts,
and $c_\kapset$, $c_\rhoset\in (0, 1)$ are two universal constants whose values can be derived from the proofs.
Note that for small $u\geq 1$, we have $c_\rhoset 2^{c_\rhoset (u-4)\aaa} \aaa\leq 1$, so the set $\rhoset_u$ coincides with $\mathcal S$.

The next theorem is the main statement of the subsection, and one of the main technical ingredients of the paper.

\begin{theor}[Decomposition of $\mathcal S$]\label{kappa and rho}
Let $v\geq 5$ be an integer. Then
$$
 \mathcal S=\bigcup\limits_{u=4}^{v}\kapset_u\,\cup\,\rhoset_v.
$$
\end{theor}

Theorem~\ref{kappa and rho} says that for any vector $x$ in $\mathcal S$, either
$x$ belongs to $\kapset_u$ for some $u\leq v$ or $x\in\rhoset_v$.
To prove this theorem, we first consider more technical (yet more simple)
ways to partition $\mathcal S$, and then gradually ``replace'' them with the conditions we are interested in.

\smallskip

The following lemma  is a straightforward implication of Lemma~2.2 in \cite{LLTTY first part}.

\begin{lemma}\label{l: separation in k}
Let $\theta_0=10/d^3$, $x\in \mathcal S$, $k\geq 5/\theta_0$, and let $y$ be the $k$-approximation of $x$.
Then there exist disjoint subsets $I,J\subset [n]$ such that $|I|,|J|\geq   \nn/4$ and for
any $i\in I$ and $j\in J$ we have
$\vert y_i-y_j\vert \geq \theta_0/2$.
\end{lemma}

We now prove a dichotomy lemma dealing with cardinalities of $\ell$-parts.

\begin{lemma}\label{l: dichotomy ls lr}
Let $\theta_0=10/d^3$, $x\in \mathcal S$, $k\geq 2d/\theta_0$, and let $y$ be the $k$-approximation of $x$.
Then at least one of the following assertions holds.
\begin{itemize}
\item The cardinality of $\, \bigcup_j\ls_j\cup\lr_j$, where the union is taken over all $j\geq 0$ with
$h(\ls_j)+h(\lr_j)\geq 10$, is at least $\nn/8$.
\item The total cardinality of the spread $\ell$-parts in the $\ell$-decomposition with respect to $y$ is at least
$\nn/120$.
\end{itemize}
\end{lemma}
\begin{proof}
By Lemma~\ref{l: separation in k}, we can find disjoint sets $I, J\subset[n]$
of cardinality at least  $\nn/4$ such that for any $i\in I $ and any $j\in J$ one has
$
\vert y_i-y_j\vert \geq \theta_0/2\geq d/k.
$
Let $(\ls_j,\lr_j)_{j=0}^\infty$ be the $\ell$-parts of $y$,
and let $j_0$ be the largest integer $j$ such that
$$
  (\ls_j\cup\lr_j)\cap (I\cup J)\neq \emptyset.
$$
For concreteness,  assume that $(\ls_{j_0}\cup\lr_{j_0})\cap I\neq \emptyset$
(the other case is treated similarly).
By Lemma~\ref{l: ell parts monotone},
$
 \{y_i:\,i\in \ls_{j}\cup\lr_{j}\}\cap \{y_i:\,i\in I\}\neq \emptyset
$
for all $j\leq j_0$.

Consider two disjoint sets of indices,
$$
  U_1= \left\{ j\leq j_0 \, : \, (\ls_{j}\cup\lr_{j})\cap J\neq \emptyset \, \, \mbox{ and } \, \,
   h(\ls_j)+h(\lr_j)\geq 10\right\}
$$
and
$$
  U_2= \left\{ j\leq j_0 \, : \, (\ls_{j}\cup\lr_{j})\cap J\neq \emptyset \, \, \mbox{ and } \, \,
   h(\ls_j)+h(\lr_j)\leq 9\right\}.
$$
Clearly,
$$
 J\subset\bigcup_{j\in U_1\cup U_2}(\ls_{j}\cup\lr_{j}),
$$
hence either
$$
 \Big|\bigcup_{j\in U_1}(\ls_{j}\cup\lr_{j})\Big|\geq  \nn/8 \quad \mbox{ or } \quad
 \Big|\bigcup_{j\in U_2}(\ls_{j}\cup\lr_{j})\Big|\geq  \nn/8.
$$
If the first bound holds we get that the total cardinality of spread or regular $\ell$-parts of cumulative height
at least $10$ is at least $\nn/8$, i.e. the first alternative of the lemma holds. We now assume that
the second bound holds. Note that for every $j\in U_2$ we have
$$
 \{y_i:\,i\in \ls_{j}\cup\lr_{j}\}\cap \{y_i:\,i\in I\}\neq \emptyset
\, \mbox{ and }\, \{y_i:\,i\in \ls_{j}\cup\lr_{j}\}\cap \{y_i:\,i\in J\}\neq \emptyset.
$$
Using that $\vert y_a-y_b\vert \geq d/k$ for all $a\in I$, $b\in J$, by the definition of the spread $\ell$-part,
we necessarily have $h(\ls_j)\geq 2$, hence $h(\lr_j)\leq 7$. By (\ref{levsethei}) this implies
$|\ls_j|\geq |\lr_j|/14$ for every $j\in U_2$. Thus,
$$
  \Big|\bigcup_{j\in U_2}\ls_{j}\Big| \geq  \frac{1}{15} \,
  \Big|\bigcup_{j\in U_2}(\ls_{j}\cup\lr_{j})\Big|\geq  \nn/120,
$$
which implies the desired result.
\end{proof}

Lemma~\ref{l: dichotomy ls lr} allows us to prove a more elaborate  dichotomy  statement.

\begin{lemma}\label{l: dichotomy III}
Let $x\in \mathcal S$, $u\geq 4$, and let $y^{u}$ and $y^{u+1}$ be the $d^u$- and $d^{u+1}$-approximations
of $x$, respectively. For each $i\leq n$, set
$$
  J^u(i):=\{j\leq n:\,y^u_j=y^u_i\}\quad\mbox{ and }\quad J^{u+1}(i):=\{j\leq n:\,y^{u+1}_j=y^{u+1}_i\}.
$$
Then we have  the following dichotomy.
\begin{itemize}
\item Either $x\in\kapset_{u}\cup\kapset_{u+1}$,
\item or $\big|\big\{i\leq n:\,2|J^{u+1}(i)|\leq |J^{u}(i)|\big\}\big|\geq \nn/192$.
\end{itemize}
\end{lemma}

\begin{proof}
Let $x,y^u,y^{u+1}$ be as above and note that $d^u\geq 2d/\theta_0$ for any $u\geq 4$, in particular
we may apply Lemma~\ref{l: dichotomy ls lr} witk $k=d^u$ to the vector $y^u$. Denote
$$
 I:=\big\{i\leq n:\,2|J^{u+1}(i)|\leq |J^{u}(i)|\big\}.
$$
Assume that $x\notin \kapset_u$ and that $I<\nn/192$.
We show that $x\in \kapset_{u+1}$.

For $m=u, u+1$ denote
\begin{align*}
U^{m}&:=\{J^{m}(i):\,i\in I^c\}  \quad \mbox{ and } \\
V^m&:=\big\{j\geq 0:\,|\{J\in U^{m}:\,(\ls_j(y^{m})\cup\lr_j(y^{m}))\cap J\neq \emptyset\}|\geq 10\big\}.
\end{align*}
We first prove that
\begin{equation}\label{old-lemma}
\Big|\bigcup_{j\in V^{u+1}}(\ls_{j}(y^{u+1})\cup\lr_{j}(y^{u+1}))\cap \bigcup_{J\in U^{u+1}}J\Big|
\geq \nn/144.
\end{equation}

Note that by the definition of $k$-approximation, given $j$, $\Re\ y^u_j=\ell/d^u$ for some integer $\ell$ if and only if
$\Re\, y^{u+1}_j=\ell/d^u+ m/d^{u+1}$
for some $0\leq m < d$, and the same holds for the imaginary
parts of  $y^u_j$, $y^{u+1}_j$. This implies that $J^{u+1} (i) \subset J^{u} (i)$ for every $i$.
Thus, there exists a bijection $\rho: U^{u}\to U^{u+1}$  such that each set
$J\in U^u$ corresponds to $\rho(J)\in U^{u+1}$ with $\rho(J)\subset J$ and $2|\rho(J)|>|J|$.
Since every $J$ in $U^{u}$ is a level set, Lemma~\ref{l: ell parts monotone} implies
that the set $V^u$ is an interval in $\Z$, that is, either $V^u=\emptyset$ or
$V^u=\{0,\dots,\, \sup V^u\}$. Similarly,  $V^{u+1}$ is an interval. Moreover, if $V^u\ne\emptyset$ then
Lemma~\ref{l: level sets basic} together with the inequality $2|\rho(J)|>|J|$ implies
$$
  \sup V^{u+1}\geq \max(\sup V^{u}-1, 0)
$$

Consider the set
$$
 J_0:=\{ j\geq 0 \, : \, j\notin V^u \quad \mbox{ and } \quad h(\ls_j(y^{u}))+h(\lr_j(y^{u}))\geq 10\}.
$$
Observe that for every $j\in J_0$, the union $\ls_j(y^{u})\cup\lr_j(y^{u})$ has at least
$$
 h(\ls_j(y^{u}))+h(\lr_j(y^{u}))- 9\geq (h(\ls_j(y^{u}))+h(\lr_j(y^{u})))/10
$$
of its level sets contained entirely in $I$.
Hence, by (\ref{levset}) (see also Lemma~\ref{l: level sets basic}),
$$
 \forall j\in J_0 \quad \big|(\ls_j(y^{u})\cup\lr_j(y^{u}))\cap I\big|\geq \frac{1}{40}\, \big|\ls_j(y^{u})\cup\lr_j(y^{u})\big|.
$$
Since $x\not\in \kapset_u$, the total cardinality of the spread $\ell$-parts in the $\ell$-decomposition
with respect to $y^u$ is at most $c_\kapset\nn <  \nn/120$ provided that $c_\kapset<1/120$.
Therefore, by Lemma~\ref{l: dichotomy ls lr}, the total cardinality of spread and regular parts of cumulative height
$10$ or more, is at least
$\nn/8$. Then the last relation and the upper bound on the cardinality of $I$ yield that
$$
  \Big|\bigcup_{j\in V^u}\ls_{j}(y^u)\cup\lr_{j}(y^u)\Big|
  \geq \nn/8 -12\vert I\vert \geq \nn/16
$$
(in particular, $V^u\ne \emptyset$).
Using that $\bigcup_{J\in U^u}J\supset I^c$, we obtain
\begin{equation}\label{interm-ineq}
 \Big|\bigcup_{j\in V^u}(\ls_{j}(y^u)\cup\lr_{j}(y^u))\cap \bigcup_{J\in U^u}J\Big|
 \geq \nn/16- \vert I\vert \geq \nn/18.
\end{equation}

Next, consider a set $J\in U^u$ satisfying
$$
 L:=\bigcup_{j\in V^u}(\ls_{j}(y^u)\cup\lr_{j}(y^u))\cap J\neq\emptyset.
$$
Then $L$ is the union of level sets of $y^u$ of all orders $0, ...,\, j_0$ for some $0\leq j_0 \leq \sup V^u$.
Since the set $\rho(J)\in U^{u+1}$ has cardinality greater than  $|J|/2$, Lemma~\ref{l: level sets basic}
implies that $\rho(J)$ must contain level sets of $y^{u+1}$ of all orders $0, ...,\, \max(j_0-1,0)$
(note that necessarily $\max(j_0-1,0)\in V^{u+1}$).
Applying Lemma~\ref{l: level sets basic} again, we obtain
$$
  \Big| \bigcup_{j\in V^{u+1}}(\ls_{j}(y^{u+1})\cup\lr_{j}(y^{u+1}))
  \cap \rho(J)\Big| \geq
  \frac{1}{8}\, \Big| \bigcup_{j\in V^{u}}(\ls_{j}(y^{u})\cup\lr_{j}(y^{u}))\cap J\Big|.
$$
This together with (\ref{interm-ineq}) implies (\ref{old-lemma}).

\smallskip

Finally we show that (\ref{old-lemma}) implies that $x\in \kapset_{u+1}$.

\smallskip

Fix $j\in V^{u+1}$ and let $J^1,J^2,\dots,J^b$ ($b\geq 10$) be (distinct) elements of $U^{u+1}$,
which have a non-empty intersection with $\ls_{j}(y^{u+1})\cup\lr_{j}(y^{u+1})$.
Denote $z_a=y^{u+1}(J^{a})$ and $w_a=y^u(\rho^{-1}(J^{a})$, $a\leq b$. Since $\rho$ is a
bijection, $w_1$, ..., $w_b$ are also distinct. It will be convenient, to see elements
of those two sequences as elements of lattices $\Lambda _u:= (\Z/d^{u})^2$ and
$\Lambda _{u+1}:= (\Z/d^{u+1})^2$.  We also denote $D=[0, (d-1)/d^{u+1}]\times [0, (d-1)/d^{u+1}]$.
As we noticed above, by construction, we have $z_a \in w_a +D$ for every $a\leq b$. Now we split
$\Lambda _u$ into nine equivalence classes using the relation $(v_1, v_2) \sim (v_3, v_4)$ if and only if
$d^u(v_1-v_3)$ and $d^u(v_2-v_4)$ are divisible by 3. Let $\Lambda$ be an equivalence class such that
$|\Lambda \cup \{w_a\}_{a\leq b}|\geq b/9$. Note, if $w_a, w_{\ell}\in \Lambda$ then
$\|z_a-z_{\ell} \|_\infty\geq 2/d^u$, in particular, $\ls_{j}(y^{u+1})\ne \emptyset$.
Let $\lambda_1$, ..., $\lambda_m$, $m\leq b/9-1$, be as in the construction of $\ls_{j}(y^{u+1})$.
Then for each $i\leq m$, $\lambda _i\in \Lambda _{u+1}$ and $\lambda _i\in \mu_i +D$
for some $\mu _i\in \Lambda _{u+1}$. Let $\bar \mu_i$  be the closest
(in $\ell_\infty$-metric) to $\mu_i$ point of $\Lambda$. Since $m\leq b/9-1$, there exists
$w_a\in \Lambda \setminus \{\mu_i\}_{i\leq m}$. Then for each $i\leq m$ we have
$$
 \| w_a- \mu_i\|_\infty\geq \| w_a-\bar \mu_i\|_\infty -\| \bar \mu _i- \mu_i\|_\infty \geq 2/d^u.
$$
Since $z_a\in w_a + D$, $\lambda _i \in \mu_i +D$, we observe
$$
 | z_a- \lambda_i| \geq \|z_a- \lambda_i\|_\infty  \geq 1/d^u.
$$
This shows that the sequence $\{\lambda_i\}_{i\leq m}$ can be continued.
Thus, $\ls_{j}(y^{u+1})$, the spread $\ell$-part of order $j$ with respect to $y^{u+1}$,
must comprise at least $b/9$ levels (i.e., its height is at least $b/9$).
Then, applying estimates for cardinalities of individual level sets (\ref{levset}),
we obtain
$$
  |\ls_j(y^{u+1})|\geq \frac{1}{4}\cdot\frac{1}{9}\,
  \Big|(\ls_{j}(y^{u+1})\cup\lr_{j}(y^{u+1}))\cap \bigcup_{J\in U^{u+1}}J\Big|.
$$
Taking the union over all $j\in V^{u+1}$ and choosing small enough $c_\kapset$,
we obtain the desired result.
\end{proof}

We are now ready to prove Theorem~\ref{kappa and rho}.

\begin{proof}[Proof of Theorem~\ref{kappa and rho}]
Fix a vector $x\in\mathcal S$, and
assume that $x\notin \bigcup_{u=4}^v\kapset_u$. We  show that $x\in\rhoset_v$.
For every $u\geq 4$, let $y^u$ be the $d^u$-approximation of $x$.
For $i\leq n$ and  $u\geq 4$ let
$$
 J^{u}(i):=\{j\leq n:\,y^u_j=y^u_i\}
\quad \mbox{ and } \quad
 I^u:=\big\{j\leq n:\,2|J^{u+1}(j)|\leq |J^u(j)|\big\}.
$$
The assumption that $x\notin \bigcup_{u=4}^v\kapset_u$, together with Lemma~\ref{l: dichotomy III},
implies that $|I^u|\geq \nn/192$ for $4\leq u<v$.
Define an auxiliary integer vector $a=(a_i)_{i=1}^n$ by setting for $i\leq n$,
$$
 a_i:=\big|\big\{4\leq u<v:\,i\in I^u\big\}\big|.
$$
The lower bound on cardinalities of sets $I^u$ implies that
$$\sum_{i=1}^n a_i\geq (v-4)\nn/192.$$
On the other hand, clearly $a_i\leq v-4$ for all $i\leq n$.
Recall $\nn=\lfloor \aaa n\rfloor$. Let
$$
  J:=\{i\leq n \, : \, a_i\geq (v-4)\nn/(384 n)\}.
$$
Then
$$
 (v-4)\nn/192\leq \sum_{i=1}^n a_i\leq |J|  (v-4) + (n-|J|) (v-4)\aaa/384,
$$
which implies
$$
 |J| \geq (\nn/192 - \aaa n/384)/(1-\aaa/384)\geq \nn/400.
$$
By the definitions of $I^u$ and $a_i$'s, we have for every $i\in J$,
$
 |J^v(i)|\leq 2^{-(v-4)\aaa/384 }\, n,
$
 hence, by Lemma~\ref{l: ell parts monotone}, in the $\ell$-decomposition of $y^v$,
 any regular or spread $\ell$-part of order
$$
 j > j_0 := \lfloor \log _2(2^{-(v-4)\aaa/384}n)\rfloor + 1
$$
does not have a non-empty intersection with $J$. Thus, we obtain
$$
  \Big|\bigcup_{j\geq 0}\ls_j(y^v)\cup\lr_j(y^v)\Big| =
  \Big|\bigcup_{j=0}^{j_0}\ls_j(y^v)\cup\lr_j(y^v)\Big| \geq |J|\geq \nn/400.
$$
Finally, since by (\ref{levsethei}) any regular or spread $\ell$-part of order $j$ and of height at most $h$ has cardinality at most
$2^{j+1} h$,  the last relation yields  for every positive integer
$h$,
\begin{align*}
 \Big|&\bigcup_{j\geq 0:\,h(\ls_j(y^v))\geq h}\ls_j(y^v)\;\cup
 \;\bigcup_{j\geq 0:\,h(\lr_j(y^v))\geq h}\lr_j(y^v)\Big|\\
 &\geq \nn/400 - \Big|\bigcup_{\substack{j\leq j_0:\, h(\ls_j(y^v))< h}}\ls_j(y^v)\;\cup
 \;\bigcup_{\substack{j\leq j_0:\, h(\lr_j(y^v))< h}}\lr_j(y^v)\Big|\\
 &\geq \nn/400 - 2 \cdot 2^{j_0+2} (h-1)
 \geq \nn/400 - h\cdot 2^{4-(v-4)\aaa/384}n.
\end{align*}
Choosing $h=2^{(v-4)\aaa/384 } \aaa/(400\cdot2^5)$, we get the result with $c_\rhoset= 1/(400\cdot2^5)$.
\end{proof}

\section{A small ball probability theorem}
\label{s:small ball}

Let $\KK\subset [n]$ and $M$ be the random matrix uniformly distributed on $\MSet_{n,d}$.
The purpose of this section is to study anti-concentration properties of a random vector
of the form $M^{\KK}y+\xyzv$, where $y$ is a fixed $k$-vector and $\xyzv$ is a fixed vector in $\C^{|\KK|}$.
The high-level idea is to replace the random vector $M^{\KK} y$, whose distribution is
difficult to describe due to
dependencies within $M^{\KK}$, by a ``simpler'' random vector $Z=(Z_i)_{i\in \KK}$
whose anti-concentration properties can be studied with the help of standard tools. The construction of $Z$
will be done in such a way that we will be able to pass from estimates for $Z$ back to $M^{\KK} y$ by conditioning on a certain event
of not too small probability.
The actual proof is technical, and even stating the main result of the section requires some preparatory work.
Instead of working with the probability space $\MSet_{n,d}$, we will split it into certain equivalence classes
(the splitting will depend on the structure of the vector $y$, more precisely, on the partition of $[n]$ given
by the $\ell$-decomposition of $y$),
and study the conditional anti-concentration.
The probability estimate will be given as a function of the $\ell$-decomposition.
As we mentioned in the introduction, this argument is related to the LCD-based method of Rudelson
and Vershynin \cite{RV} which in turn was strongly influenced by earlier works on singularity of
discrete random matrices
\cite{KKS95, TV bernoulli, TV ann math}.
A principal difference of our approach is that the $\ell$-decomposition, being a ``multidimensional'' characteristic
of a vector, provides much more structural information than LCD. This structural information
is heavily used in this part of the paper.

We start by introducing a structure on $\MSet_{n,d}$.
For each $m\leq n$, let $\RSet_{n,m,d}$ be the set of $n\times m$ matrices with  integers coefficients
from the set $\{0, 1, \dots, d\}$ such that \\
{\bf 1.} the sum in each row is $d$, and \\
{\bf 2.} the sum in every column is a non-negative integer multiple of $d$.

\smallskip

Now, for every $k$-vector $y$ with the
$\ell$-decomposition $(\lset^{(q)})_{q=1}^m$, where
$$
 m=m(y)\leq 3 \ln n,
$$
we define the mapping $\stackrel{(y)}{\longrightarrow}$ from $\MSet_{n,d}$ into $\RSet_{n,m,d}$,
which assigns to each matrix $M\in \MSet_{n,d}$ an $n\times m$ matrix $\RR=(\RR_{iq})_{iq}\in \RSet_{n,m,d}$ defined by
\begin{equation*}
 \forall i\leq n \,\, \, \forall q\leq m \, \, \, \,\, \,  \RR_{iq}:=\sum\limits_{j\in \lset^{(q)}}M_{ij},
\end{equation*}
that is, the matrix $Q$ is obtained from $M$ by summing up respective columns.
This mapping defines an equivalence relation $\stackrel{y}{\sim}$ on $\MSet_{n,d}$,
where $M\stackrel{y}{\sim}M'$ whenever both $M$ and $M'$ are mapped to the same element of $\RSet_{n,m,d}$.
Further, a given matrix $\RR\in \RSet_{n,m,d}$, we  denote by $\MSet_{n,d}(\RR,y)$ the equivalence
class of matrices in $\MSet_{n,d}$, which are mapped to $\RR$ via the correspondence
$\stackrel{(y)}{\longrightarrow}$.
If $\MSet_{n,d}(\RR,y)\ne \emptyset$ then the uniform probability measure on $\MSet_{n,d}(\RR,y)$
will be denoted by $\Prob_{\RR,y}$.

\smallskip

For the rest of this section we fix integers $k\geq 1$, $m\geq 1$, and a vector $y\in\kset_k$ with the
$\ell$-decomposition $(\lset^{(q)})_{q=1}^m$.
Let $\RR\in\RSet_{n,m,d}$ be such that there exists $M\in \Mc$ which is mapped to $\RR$ by $\stackrel{(y)}{\longrightarrow}$,
in particular, for every $q\leq m$ one has
\begin{equation}\label{sumqiq}
 \sum_{i=1}^n \RR_{iq} = d\, |\lset^{(q)}|.
\end{equation}
In what follows, such matrices $\RR$ will be called {\it $y$-admissible}.
Denote $\hq:=h(\lset^{(q)})$.
For all $i\leq n$ and $q\leq m$, we define {\it the weight} $w_{iq}$ by
\begin{equation}\label{weights}
 w_{iq}=w_{iq}(y,k,\RR):=
 \begin{cases} \hq{\RR_{iq}/d},
 &\mbox{if $\lset^{(q)}$ is regular},\\
 \hq\sqrt{\RR_{iq}},&\mbox{if $\lset^{(q)}$ is spread}.
 \end{cases}
\end{equation}
Now, given $i\leq n$, {\it the small ball probability estimator} $\est_i$ is
$$
  \est_i=\est_i(y,k,\RR):=\min\big(1,\min\limits_{q\leq m}w_{iq}^{-1}\big),
$$
where we adopt the convention $0^{-1}=\infty$. The estimators $\est_i$ are designed to
measure anti-concentration of inner products $\langle\row_i(M),y^\dagger\rangle$, for $M$ distributed in $\MSet_{n,d}(\RR,y)$.
We prove the following theorem.

\begin{theor}[Small ball probability]\label{sbp th}
Let $d,n$ be large enough integers such that $d^3\leq n$.   Let
 $\KK\subset [n]$ be such that $\vert \KK^c\vert \leq n/(50 \ln d)$
and assume
$$
 1\le k\le \min\big( \sqrt{n}/(8d^{3/2} \sqrt{\ln d}),\, d^{-10} e^{n/(5\vert \KK^c\vert)}\big).
$$
Let $y$ and $Q$ be as above. Then
for any non-random vector $\xyzv\in \C^{\vert \KK\vert}$ and any $\gamma\geq 1$ one has
$$\Prob_{\RR,y}\bigl\{M\in \MSet_{n,d}(\RR,y):\,
\|M^{\KK} y+\xyzv\|_2\leq \gamma \sqrt{d\vert \KK\vert }/k\bigr\}\leq C^n
\gamma^{2\vert \KK\vert}\prod_{i=1}^n\est_i,$$
where $C>0$ is a universal constant.
\end{theor}

\smallskip

The main difficulty in proving Theorem~\ref{sbp th} lies in the fact that the
rows of the random matrix uniformly disctributed on $\Mc$ are dependent.
To deal with this issue, we construct a special random vector $Z$ in $\C^{\vert \KK\vert}$
with {\it independent} coordinates
having a property that, conditioned on a certain event of not too small probability, it
has the same distribution as $M^\KK y$.

Let $y$ and $K$ be as in the theorem.
Recall that by the definition each $\ell$-part $\lset^{(q)}$ is representable
as the union of level sets of $y$,
\begin{equation} \label{hoflq}
 \lset^{(q)}=\bigcup\limits_{p=1}^{\hq}L^q_p,
\end{equation}
 where we  assume for concreteness
that $y(L^q_{p+1})<y(L^q_{p})$ (in lexicographical order) for all $p<\hq$.
For each $q\leq m$, we define the set of pairs
$$
  \Delta_q:=\bigl\{(i,w)\,\, :\,\, 1\leq i\leq n,\, 1\leq \RR_{iq}, \, 1\leq w\leq \RR_{iq}\bigr\}.
$$
Then $|\Delta_q|=\sum\limits_{i=1}^n \RR_{iq}=d|\lset^{(q)}|$. Further, let
$$\bigl\{\xi_{\delta}^q\, :\, 1\leq q\leq m,\, \delta\in\Delta_q\bigr\}$$
be a collection of jointly independent random variables, where
each $\xi_{\delta}^{q}$ is distributed in the set $\{1,2,\dots,\hq\}$ in such a way that
 for all $p\leq \hq$,
$$
  \Prob\bigl\{\xi_{\delta}^{q}=p\bigr\}=\frac{|L^q_{p}|}{|\lset^{(q)}|}.
$$
Define random variables $Z_{i}$, $i\in \KK$,  as
\begin{equation}\label{eq-defi-Zi}
 Z_i:=\sum\limits_{q=1}^m\sum\limits_{w=1}^{\RR_{iq}}y\Bigl(L^q_{\xi^q_{(i,w)}}\Bigr)
\end{equation}
and set $Z:=(Z_{i})_{i\in \KK}$. Note that each variable $Z_i$ is a function of
$$
 \bigl\{\xi^q_{(i,w)}\, :\, 1\leq q\leq m,\, 1\leq w\leq \RR_{iq}\bigr\},
$$
and those sets of variables are clearly disjoint for distinct $i$'s, hence $(Z_{i})_{i\in \KK}$
are jointly independent. Since each $Z_i$ is a sum of discrete complex-valued random variables,
we can apply Proposition~\ref{p:complexLevy} to study its
anti-concentration properties.
As we  show below, the conditional distribution of $Z$ given a certain event of not too small probability,
coincides with the distribution of $A^Ky$, where $A^K$ is a ``multigraph'' version of $M^K$ in which
we allow entries greater than one (i.e., multiple edges). This correspondence will be made precise later,
as the first step we define and estimate the probability of the event to be conditioned on.

\begin{claim}\label{c:multinomial}
Let $h,\, N, N_1,\,\ldots,\,N_h$ be positive integers satisfying
$\sum_{p=1}^h N_p=N$. Let $\xi_1$, $\xi_2,\dots$, $\xi_N$ be i.i.d.\ random variables taking values in the set
$\{1,2,\dots,h\}$ with probabilities
$\Prob\{\xi_i=p\}=N_p/N$ for all $p\leq h$. Then
$$
  \Prob\big\{ \forall p\leq h: \, \, \, \,  |\{i\le N\,:\,\xi_i=p\}|=N_p \big\}\geq (h/(Ne^2))^{h/2}.
$$
\end{claim}

\begin{proof}
Denote the event $\big\{ \forall p\leq h: |\{i\le N:\xi_i=p\}|=N_p \big\}$ by $\Event$. Note that the random variables
$$
\eta_p :=|\{i\le N\,:\,\xi_i=p\}|, \, \,  p\leq h,
$$
have a multinomial distribution. Since
$(n/e)^n< n!\le e\sqrt{n}(n/e)^n$, we have
$$
\Prob(\Event)=\frac{N!}{N_1!\cdots N_h!}\prod_{p=1}^h\Big(\frac{N_p}{N}\Big)^{N_p}
> \frac{(N/e)^N\prod_{p=1}^h{N_p}^{N_p}}{N^{N}\prod_{p=1}^h e\sqrt{N_p}(N_p/e)^{N_p}}
=1/\prod_{p=1}^h e\sqrt{N_p}.
$$
The arithmetic-geometric mean inequality $\prod_{p=1}^h {N_p}\le (N/h)^h$ implies the bound.
\end{proof}

\begin{lemma}\label{lem-multinomial}
Let $d\leq n$ be large enough positive integers, $k\le  \sqrt{n}/(8d^{3/2} \sqrt{\ln d})$, $y\in\kset_k$, and  $\{\xi^q_\delta\}$ be as above.
Define the event
$$
 \Event_{\ref{lem-multinomial}}:=\bigcap_{q\leq m}\big\{  \forall p\leq \hq: \, \, \, \,  \vert\{\delta\in\Delta_q:\,
 \xi^q_\delta =p\}\vert =d\vert L^q_p\vert \big\}.
$$
Then $\,\Prob(\Event_{\ref{lem-multinomial}}) \geq e^{-n}$.
\end{lemma}
\begin{proof}
Let $H=\sum_{q=1}^m \hq$ be the total number of level sets of $y$.
Since $y$ is the $k$-approximation of a gradual vector, we have
$y_{n_2}^*\leq 2d^{\frac{3}{2}}$, with $n_2$
defined in Section~\ref{steep}.
Therefore, using the assumption on $k$,
$$
 H\leq n_2+ ((2k+1)\, 2d^{\frac{3}{2}})^2\leq n/(2 \ln d).
$$

By the independence of $\xi^q_\delta$, $q\leq m$, and by Claim~\ref{c:multinomial} applied
for every $q\leq m$ with  $N= d|\lset^{(q)}|$ and $N_p= d |L^q_{p}|$, $p\leq \hq$, we get
\begin{align*}
\Prob(\Event_{\ref{lem-multinomial}}) &\geq \prod_{q=1}^m\Big(\frac{\hq}{d|\lset^{(q)}| e^2}\Big)^{\hq/2}
=(e^2 d)^{-H/2}\prod_{q=1}^m\Big(\frac{\hq}{|\lset^{(q)}|}\Big)^{\hq/2}\\
&\ge (e^2 d)^{-n/(4\ln d)}\prod_{q=1}^m e^{-|\lset^{(q)}|/2e}
= \exp\left(-n/(2\ln d) -n/4 -n/(2e)\r),
\end{align*}
where in the last inequality we used the bound on $H$ and that $x^x\geq e^{-1/e}$ for all $x>0$.
This completes the proof for   large enough $d$.
\end{proof}

\smallskip

As the next step, we study anti-concentration properties of a single variable $Z_i$.

\begin{lemma}\label{single row lemma}
Given $k\ge 1$, let $y\in \kset_k$ and let vector $Z$ be defined as above.
Then for every $i\in \KK$ and every  $\tau \geq 1$ one has
$$
\cf(Z_i,\sqrt{d}\tau/k)\leq C\tau ^2\,\est_i,
$$
where $C\ge 1$ is a universal constant.
\end{lemma}

\begin{proof}
 A simple estimate $\cf(\eta_1+\eta_2,\gamma)\leq \min\bigl(\cf(\eta_1,\gamma),\cf(\eta_2,\gamma)\bigr)$,
which is valid for any pair $\eta_1,\eta_2$ of independent random variables and all $\gamma>0$, together
with the definitions of $Z_i$'s and  $\est_i$'s, implies that it is sufficient to prove the relations
\begin{equation*}\label{up b sin row}
\cf\Bigl(\sum\limits_{w=1}^{\RR_{iq}}
y\bigl(L^q_{\xi^q_{(i,w)}}\bigr), \frac{\sqrt{d}\tau }{k}\Bigr)\le \frac{C \tau ^2}{w_{iq}}.
\end{equation*}
for all $i\in K$ and $q\leq m$.

Fix $i\in K$ and  $q\leq m$ such that $w_{iq}\neq 0$, and
denote the variables $y\bigl(L^q_{\xi^q_{(i,w)}}\bigr)$ by $\psi_{w}$, $1\leq w\leq \RR_{iq}$.
Note that each $\psi_w$ is a discrete random variable taking values in the set
$$
  B:=\bigl\{y(L^q_p)\, : \, p\leq \hq\bigr\}.
$$
By (\ref{hoflq}) and by (\ref{levset}) one has
$$
  |\lset^{(q)}|= \sum_{p=1}^{\hq} |L^q_p| \quad  \quad \mbox{ and } \quad  \quad
  \max_{p \leq \hq} |L_{p}^q|\leq 4 \min_{p \leq \hq} |L_{p}^q|.
$$
Hence, for any $b\in B$
$$
  \Prob\{\psi_w=b\}\leq \max\limits_{p\leq \hq}\frac{|L^q_p|}{|\lset^{(q)}|}
  \leq \frac{4}{\hq+3}.
$$

 If the part $\lset^{(q)}$ is regular then the set $B$ is a $(1/k)$-separated subset of the complex plane.
Applying Proposition~\ref{p:complexLevy} and using that $\tau \geq 1$ and $\RR_{iq}\leq d$,
we obtain
$$
\cf\Bigl(\sum\limits_{w=1}^{\RR_{iq}}\psi_w,\frac{\sqrt{d}\tau }{k}\Bigr)\leq
C_0\, \max\Big(\frac{d\tau ^2}{\hq\RR_{iq}},\frac{1}{\hq}\Big)= \frac{C_0\, d\tau ^2}{\hq\RR_{iq}}
= \frac{C_0\, \tau ^2}{w_{iq}} ,
$$
where $C_0>0$ is a universal constant.

If the part  $\lset^{(q)}$ is spread then the set $B$ is a
$(d/k)$-separated subset of the complex plane. Note that
$w_{iq}= \hq\sqrt{Q_{iq}}\le \hq\sqrt{d}$.
Without loss of generality, we can also assume that
$1\le \tau^2\le w_{iq}$ (otherwise the probability estimate is trivial).
Using that the number of $(d/k)$-separated points in a ball of radius $\lambda$
is smaller than $\big(1+2\lambda k/d\big)^2$, we obtain for all
 $\lambda>0$ and  $w\leq \RR_{iq}$,
$$
\cf(\psi_w,\lambda)\leq \frac{4}{\hq+3}\, \big(1+2\lambda k/d\big)^2.
$$
Assume first that $\hq\le 32$. Using that the heights of non-empty spread parts
$\hq$ are at least $2$, we have
$$
\cf(\psi_w,\sqrt{d} /k)\leq \frac{4}{\hq+3}\, \big(1+2/\sqrt{d}\big)^2\leq \frac56,
$$
provided $d$ is large enough. Therefore, applying Proposition~\ref{prop: esseen}
with $t=\sqrt{d} \tau/k$ and $t_0=\sqrt{d}/k$, we get
$$
 \cf\Bigl(\sum\limits_{w=1}^{\RR_{iq}}\psi_w,\frac{\sqrt{d}\tau}{k}\Bigr)\leq
 \frac{C_1 \, \tau^2}{\sqrt{\RR_{iq}}}\leq \frac{32 C_1 \, \tau^2}{\hq\sqrt{\RR_{iq}}}
 =\frac{32 C_1 \, \tau^2}{w_{iq}},
$$
where $C_1>0$ is an absolute constant.
Let now $\hq>32$. Then,  using $\tau^2\le \hq\sqrt{d}$ and that $d$ is large enough,
$$
  \cf(\psi_w, \sqrt{ d}\, \tau/k)\leq \cf(\psi_w, \sqrt{2 d}\, \tau/k)\leq \frac{4}{\hq+3}\, \big(1+2\sqrt{2}\tau/\sqrt{d}\big)^2\leq
  \frac{8 + 64 \tau^2/d}{\hq}
  \leq 1/2.
$$
Therefore, applying Proposition~\ref{prop: complex-kesten} with $t=t_0=\sqrt{d}\tau/k$, we obtain
$$
  \cf\Bigl(\sum\limits_{w=1}^{\RR_{iq}}\psi_w,\frac{\sqrt{d}\tau}{k}\Bigr)\leq C_2
   \frac{8 + 64 \tau^2/d}{\hq \sqrt{\RR_{iq}}} \leq 72  C_2
   \frac{\tau^2}{\hq \sqrt{\RR_{iq}}} = 72 C_2  \frac{\tau^2}{w_{iq}},
$$
where $C_2>0$ is an absolute constant.
This completes the proof.
\end{proof}

\smallskip

To complete the proof of Theorem~\ref{sbp th}
we  tensorize the last lemma,
i.e., we pass from the anti-concentration estimates for individual $Z_i$'s to the vector $Z$, and
then we tie the obtained estimates for $Z$ with anti-concentration properties of $M^\KK y$.
At this point, it will be convenient to introduce a new random object -- a multigraph on $[n]$
which, in a certain sense, will correspond to the vector $Z$. This way, a direct relation between
$Z$ and $M^\KK y$ can be defined by conditioning on the event that the multigraph is simple, i.e.,
does not contain multiple edges.

\smallskip

Let $y$, $(\lset^{(q)})_{q=1}^m$, and $\RR$ be as above.
We construct the  multigraph $\widehat{G}_{\RR}$ on $[n]$ as the union  of certain independent
bipartite multigraphs $\widehat{G}_q$ on $([n], \lset^{(q)})$, that is
$$
\widehat{G}_{\RR}=\bigcup_{q=1}^m\widehat{G}_q,
$$
where to form $\widehat{G}_q$ we adapt the configuration model in the following way.
For every $q\le m$, define
$$
  \Delta_q':=\lset^{(q)} \times [d] = \bigl\{(j,w'):\,j\in \lset^{(q)}, 1\leq w'\leq d\bigr\}.
$$
Clearly, $\vert \Delta_q'\vert =d \vert \lset^{(q)}\vert= \vert \Delta_q\vert$.
Let $g_q$ be a (fixed) bijection from $\Delta_q$ to $\Delta_q'$, and let $\sigma_q$ be
a random uniform permutation on $\Delta_q'$
($\sigma_q$ does not respect the two-dimensional structure of $\Delta_{q}'$ and  can be viewed as a uniform random element of the permutation group $\Pi_{\vert \Delta_q'\vert}$, also we suppose that $\sigma_1$, ..., $\sigma_m$ are jointly independent).
We  define $\widehat{G}_q$ as a bipartite multigraph
on $([n], \lset^{(q)})$ with the edge multiset
\begin{align*}
E_q:= \Big\{(i,j)\in [n]\times\lset^{(q)} \, :&\,\,  \exists\, 1\leq w\leq \RR_{iq}\;,1\leq w'\leq d \text{ such that } \\
&(i,w)\in \Delta_q,\,
(j, w')\in \Delta_q' \text{ and } g_q(i,w)=\sigma_q(j, w')\Big\},
\end{align*}
where the multiplicity $r_q(i,j)$ of each edge $(i,j)$ in $E_q$ is equal to the cardinality of the set
$$g_q(\{i\}\times [\RR_{iq}])\cap \sigma_q(\{j\}\times [d]).$$
Note that by construction, $\widehat{G}_q$
has degree sequence $(\RR_{iq})_{i\leq n}$ for vertices in $[n]$ and a constant degree
$d$ for vertices in $\lset^{(q)}$.
We define $\widehat{G}_{\RR}$ as the union of $\widehat{G}_{q}$'s, $q\leq m$, in particular,
the edge multisets $E$ of $\widehat{G}_{\RR}$ is
$$
  E=\bigcup_{q=1}^m E_q.
$$
 Denote
$$
  p_q=p_q(i,j): = \sum_{(w,w')\in [\RR_{iq}]\times [d]} \Prob\{ g_q(i,w)=\sigma_q(j,w')\}\le \RR_{iq}/\vert \lset^{(q)}\vert.
$$
Then  $ \Exp\, \{r_q(i,j)\mid (i,j)\in E_q\}= p_q$, and  by the union bound $\Prob\{(i,j)\in E_q\} \leq  p_q$. Therefore,
using that each $\RR_{iq}$ is at most $d$, we observe
\begin{align*}
\Exp\, r_q(i,j) = \Exp\, \{r_q(i,j)\mid (i,j)\in E_q\}\,\, \,   \Prob\{(i,j)\in E_q\} \leq  p_q^2\le \RR_{iq}^2/\vert \lset^{(q)}\vert^2 \le d \, \RR_{iq}/\vert \lset^{(q)}\vert^2.
\end{align*}
Let $N_q$ be the total number of multiple edges produced in the random bipartite multigraph
$\widehat{G}_q$. Using Markov's inequality and  \eqref{sumqiq}, we obtain
\begin{align}\label{eq: multiple-edges}
     \Prob\big\{ N_q\geq 2d^2\big\} &\leq \frac{1}{2d^2}\, \Exp\, N_q\leq
     \frac{1}{2d^2}\,  \sum_{i\leq n}\, \sum_{ j\leq \vert \lset^{(q)}\vert} \Exp\, r_q(i,j) \leq
     \frac{1}{2d}\,  \sum_{i\leq n}\frac{\RR_{iq}}{\vert \lset^{(q)}\vert } \leq
     \frac12.
\end{align}
Thus for every $q\le m$ at least half of realizations of the random bipartite multigraph $\widehat{G}_q$ have the number of multiple edges at most $2d^2$.
 In the sequel we will see that  for every $q\le m$ a non-negligible part of realizations of $\widehat{G}_q$  have no multiple edges,
 so that a non-negligible proportion of realizations of  $\widehat{G}_{\RR}$ are simple.

One can check that any realization of $\widehat{G}_{\RR}$
occurs with probability
$$
\prod_{q=1}^m \frac{(d!)^{\vert\lset^{(q)}\vert}\prod_{i=1}^n \RR_{iq}!}{(d\vert \lset^{(q)}\vert)! }.
$$
Moreover, the realizations of $\widehat{G}_{\RR}$ which are simple precisely correspond
to the graphs whose adjacency matrices belong to
$\MSet_{n,d}(\RR,y)$. Therefore,
\begin{equation}\label{eq: card-simple}
\Prob\{\widehat{G}_{\RR} \text{ is simple}\}= \vert\MSet_{n,d}(\RR,y)\vert \prod_{q=1}^m \frac{(d!)^{\vert\lset^{(q)}\vert}\prod_{i=1}^n \RR_{iq}!}{(d\vert \lset^{(q)}\vert)! }.
\end{equation}
Below we denote the adjacency matrix of $\widehat{G}_{\RR}$ by $A$
(with the entries of $A$ respecting multiplicities).
Then for any $M\in \MSet_{n,d}(\RR,y)$ one has
\begin{equation}\label{eq: conf-model}
  \Prob\{ A=M\mid \widehat{G}_{\RR} \text{ is simple}\}=
  \Big(\Prob\{\widehat{G}_{\RR} \text{ is simple}\}\Big)^{-1}\,\,
  \displaystyle\prod_{q=1}^m \frac{(d!)^{\vert\lset^{(q)}\vert}\prod_{i=1}^n
  \RR_{iq}!}{(d\vert \lset^{(q)}\vert)! }= \frac{1}{\vert\MSet_{n,d}(\RR,y)\vert},
\end{equation}
which means that conditioned on the event that $\widehat{G}_{\RR}$ is simple, the matrix $A$
is uniformly distributed on $\MSet_{n,d}(\RR,y)$.

In the next proposition, we provide a lower bound on the  probability that $\widehat{G}_{\RR}$ is simple.
Note that by our construction, this probability is equal to  the product of the probabilities that each of the
bipartite graphs $\widehat{G}_q$ is simple.

\begin{prop}\label{prop-simple}
Let $d\leq n$ be large enough and
denote by $\Event_{\ref{prop-simple}}$ the event that $\widehat{G}_{\RR}$ is simple. Then
$$
\Prob\big(\Event_{\ref{prop-simple}}\big)\geq \exp(-33 d^2 \ln^2 n).
$$
\end{prop}

\begin{proof}
For every  $q\leq m$  such that $\vert \lset^{(q)}\vert\leq 5 d$, we bound from below
the probability that $\widehat{G}_q$ is simple by one over the number of realizations
of such multigraphs, that is, by
$$
  \big(d\vert\lset^{(q)}\vert)^{-d\vert\lset^{(q)}\vert}\geq \exp(-5 d^2 \ln (5d)^2)
  \geq \exp(-11 d^2 \ln n).
$$

Now we treat $q\leq m$ such that $\vert \lset^{(q)}\vert> d\ln n$. For such $q$  we could use
precise asymptotics obtained in \cite{McKay-simple} (see also \cite[Theorem~1.1]{GMckayW}),
however, for the readers' convenience, we prefer to provide a simple self-contained argument
(which leads also to a better bound).
For every $q\leq m$, denote by $\MSet_{n,d}(\RR,y,\lset^{(q)})$ the set of $n\times \vert \lset^{(q)}\vert$
matrices corresponding to blocks of columns indexed by $\lset^{(q)}$
of matrices from the equivalence class $\MSet_{n,d}(\RR,y)$. With this notation, we have
$$
 \vert\MSet_{n,d}(\RR,y)\vert= \prod_{q=1}^m \vert \MSet_{n,d}(\RR,y,\lset^{(q)})\vert.
$$
Similarly to \eqref{eq: card-simple}, the probability that the random multigraph $\widehat{G}_q$
on $([n],\lset^{(q)})$ is simple is given by
$$
\vert\MSet_{n,d}(\RR,y, \lset^{(q)})\vert \, \frac{(d!)^{\vert\lset^{(q)}\vert}\prod_{i=1}^n \RR_{iq}!}{(d\vert \lset^{(q)}\vert)! }.
$$
Therefore it is sufficient to estimate the cardinality of $\MSet_{n,d}(\RR,y, \lset^{(q)})$ for each $q\leq m$.
Let $\mathcal{\widehat{M}}_q$ be the set of all adjacency matrices corresponding to realizations of
$\widehat{G}_q$ (with the entries respecting multiplicities). Moreover, let $\mathcal{\widehat{M''}}$
be the subset of $\mathcal{\widehat{M}}_q$ given by matrices such that the sum over entries exceeding
$1$ is bounded above  by $2d^2$.
The latter corresponds to multigraphs having at most $2d^2$ multiple edges. By  \eqref{eq: multiple-edges}, we have
\begin{equation}\label{eq: size-multiple-edges}
\vert \mathcal{\widehat{M''}}\vert \geq \frac12 \vert\mathcal{\widehat{M}}_q \vert.
\end{equation}
To estimate the cardinality of $\MSet_{n,d}(\RR,y, \lset^{(q)})$, we define a relation $R\in \mathcal{\widehat{M''}}\times \MSet_{n,d}(\RR,y, \lset^{(q)})$ as follows.
We let a pair $(M,M')$ belong to $R$ if $M'$ can be obtained from $M$ by a sequence (of maximal length $2d^2$)
of simple switching operations in the following way: for every $(i,j)$ such that $M_{ij}>1$, choose first $(i',j')$
such that $M_{ij'}=M_{i'j}=0$ and $M_{i'j'}\geq 1$, then operate the simple switching on $i, j, i', j'$.
By regularity of our matrices and by the definition of $\mathcal{\widehat{M''}}$, the number of pairs $(i',j')$
with $M_{i'j'}\geq 1$ is at least
$$
  \sum _{s, t} M_{s,t} - \sum _{s, t: \,M_{s,t}\geq 1 } (M_{s,t}-1) \geq d\vert \lset^{(q)}\vert  - 2d^2.
$$
Moreover, the number of $s$ and $t$ such that either $M_{is}\ne 0$ or $M_{tj}\ne 0$ is at most $2 d (d-1)$.
Therefore, using that $\vert \lset^{(q)}\vert \geq 5 d$, we observe that there are at least
$d\vert \lset^{(q)}\vert - 4d^2\geq d\vert\lset^{(q)}\vert/5 $ choices for a ``good"  pair $(i',j')$, hence
$$
 \vert R(M)\vert \geq d\vert\lset^{(q)}\vert/5.
$$
Note that after such a switching, the sum over entries exceeding $1$ must decrease.
Then we reiterate this procedure until the all entries becomes less than or equal to $1$.
Note that we do not need more than $2d^2$ steps (since the sum over entries exceeding $1$ is bounded above
by $2d^2$).
Now we revert the procedure and start with $M'\in R(\mathcal{\widehat{M''}})$.
Since at each step the number of non-zero elements is at most $d\vert \lset^{(q)}\vert $,
the number of possible switching operations is smaller than
$d^2\vert \lset^{(q)}\vert^2 /2$. Since the number of steps is at most $2d^2$, we have
$$
\vert R^{-1}(M')\vert\leq  \big(d^2\vert \lset^{(q)}\vert^2 /2\big)^{2d^2}.
$$
 Claim~\ref{multi-al} and the  bound $d\vert\lset^{(q)}\vert\leq n^2$ imply  that
$$
\vert\mathcal{\widehat{M''}}\vert \leq
\big(5/d\vert \lset^{(q)}\vert\big)\, \, \big(d^2\vert \lset^{(q)}\vert^2 /2\big)^{2d^2}\, \,
\vert \MSet_{n,d}(\RR,y, \lset^{(q)})\vert
\leq \big(1/2\big)\, \, \exp \big( 8 d^2 \ln n\big)\, \,
\vert \MSet_{n,d}(\RR,y, \lset^{(q)})\vert.
$$
By  $\eqref{eq: size-multiple-edges}$ this yields
$$
 \vert \MSet_{n,d}(\RR,y, \lset^{(q)})\vert \geq
 \big(d\vert \lset^{(q)}\vert/4\big)\, \, \big(\sqrt{2} /(d\vert \lset^{(q)}\vert) \big)^{4 d^2}\, \,
 \geq
 \exp \big(- 8 d^2 \ln n\big)\, \, \vert\mathcal{\widehat{M}}_q \vert,
$$
hence, the probability that the $\widehat{G}_q$ is simple is at least $\exp(- 8 d^2 \ln n)$.

Finally, as we mentioned above, the probability that $\widehat{G}_{\RR}$ is simple is equal to the
product of the probabilities that each $\widehat{G}_q$, $q\leq m$, is simple. Thus, the probability
that $\widehat{G}_{\RR}$ is simple is at least $\exp\big(- 11 d^2\, m  \ln n\big)$. Since by the
construction of the $\ell$-decomposition, $m\leq 3\ln n$, we obtain
the desired estimate.
\end{proof}

We now verify that the adjacency matrix $A$ of  $\widehat{G}_{\RR}$
 satisfies a condition similar to that of Theorem~\ref{sbp th}.

\begin{lemma}\label{lem-norm of Ay}
Let $d\leq n$ be large enough,  $m\geq 1$, and $k\le  \sqrt{n}/(8d^{3/2} \sqrt{\ln d})$.
Let  $y\in\kset_k$, $(\lset^{(q)})_{q=1}^m$ be its $\ell$-decomposition,
$\RR$ be a $y$-admissible matrix, and  $\widehat{G}_{\RR}$ be the random multigraph constructed above with the
 adjacency matrix $A$. Then for any $\KK\subset [n]$, any non-random vector
 $\xyzv\in \C^{\vert \KK\vert}$, and any $\gamma\geq 1$ one has
$$
\Prob\Big\{\Vert A^\KK y+\xyzv\Vert_2\leq \gamma\, \frac{\sqrt{d\vert \KK\vert}}{k}\Big\}\leq C_{\ref{lem-norm of Ay}}^n \gamma^{2\vert \KK\vert}\prod_{i\in \KK} \est_i.
$$
\end{lemma}

\begin{proof}
As before, for every $q\leq m$
we represent the $\ell$-set $\lset^{(q)}$ as the union $\bigcup_{p\leq h_q} L^q_p$.
For every $j\in\lset^{(q)}$ let $f^q(j)$ denote the index
of the level set in $\lset^{(q)}$ containing $j$, i.e., $f^q(j)=p$ whenever $j\in L^q_{p}$.
It is convenient  to have a representation for the multiset $E_q$ in the form
$$
E_q=\big\{e_\delta:\,\delta\in \Delta_q\big\},
$$
where for every $\delta=(i,j)\in \Delta_q$ we have
$e_\delta=(i,j_\delta^q)$ with $j_\delta^q\in\lset^{(q)}$
being equal to the first component of the pair
$$
 \sigma_q^{-1}(g_q(\delta))\in\Delta_q'.
$$
Recall that random variables $\{\xi^q_\delta\}$, $q\leq m$, $\delta\in\Delta_q$,  were introduced in
the first part of this section.
Observe that for any fixed $q\leq m$, the joint distribution of the variables
$\{\xi^q_\delta\}$,  $\delta\in\Delta_q$,
conditioned on the event $\Event_{\ref{lem-multinomial}}$,
coincides with the joint distribution of
$$
 \{f^q(j_\delta^q):\,\delta\in\Delta_q\}.
$$
Indeed, by the construction of the multigraph $\widehat{G}_{\RR}$, the variables
$\big(f^q(j_\delta^q)\big)_{\delta\in\Delta_q}$
take values in the set of sequences
\begin{equation}\label{eq: aux325y}
\left\{(a_\delta)_{\delta\in \Delta_q}\in\N^{\Delta_q}:\quad\quad \quad \forall p\leq \hq \quad \quad
\vert \{\delta:\, a_\delta =p\}\vert =d\vert L_p^q\vert \r\} .
\end{equation}
Note that the set of permutations of $\Delta_q$ acts transitively on the set in \eqref{eq: aux325y}.
Hence, taking into account that the distribution of $\big(f^q(j_\delta^q)\big)_{\delta\in\Delta_q}$
is invariant under permutations of $\Delta_q$, we get that all realizations of the sequence are equi-probable.
On the other hand, conditioned on $\Event_{\ref{lem-multinomial}}$, the sequence
$(\xi^q_\delta)_{\delta\in\Delta_q}$ is distributed over \eqref{eq: aux325y}, and it is not difficult to see
that the conditional distribution is uniform. Thus, the distribution of $\big(f^q(j_\delta^q)\big)_{\delta\in\Delta_q}$
and $(\xi^q_\delta)_{\delta\in\Delta_q}$ given $\Event_{\ref{lem-multinomial}}$ must coincide.

We now relate the coordinates of the vector $A^\KK y$ to the variables $Z_i$.
Denoting the entries of $A$ by $a_{ij}$, note that for every pair $(i, j)$
the entry $a_{ij}$ is
the multiplicity of the edge $(i,j)$ in $E_q$ (which can be zero if the edge
does not belong to $E_q$). Therefore for all $i\in \KK$, we have
\begin{align*}
  (Ay)_i= \sum_{j=1}^n a_{ij}y_j =\sum_{q=1}^m\sum_{j\in \lset^{(q)}} a_{ij}y_j
  &=\sum_{q=1}^m\sum_{j\in \lset^{(q)}}  y\big(L^q_{f^q(j)}\big)\,a_{ij}
  =\sum_{q=1}^m\sum_{w=1}^{Q_{iq}} y\bigl(L^q_{f^q(j_{(i,w)}^q)}\big).
\end{align*}
Hence,
$$A^\KK y\stackrel{d}{\sim} Z \mbox{ conditioned on $\Event_{\ref{lem-multinomial}}$}.$$
Applying Lemma~\ref{lem-multinomial}, we get
\begin{align*}
\Prob\Big\{\Vert A^\KK y+\xyzv\Vert_2\leq \gamma\, \frac{\sqrt{d\vert \KK\vert}}{k}\Big\}
&=\Prob\Big\{\Vert Z+\xyzv\Vert_2\leq \gamma\, \frac{\sqrt{d\vert \KK\vert}}{k}\,\, \Big|\, \,  \Event_{\ref{lem-multinomial}}\Big\}\nonumber\\
&\leq e^n{\Prob\Big\{\Vert Z+\xyzv\Vert_2\leq \gamma\, \frac{\sqrt{d\vert \KK\vert}}{k}\Big\}}.
\label{eq-from A to Z}
\end{align*}
By Lemma~\ref{single row lemma} there is an absolute constant $C\geq 1$ such that for all $\gamma\geq 1$,
$$
 \cf( Z_i, \gamma\sqrt{d}/k) \leq C \gamma^2\,\est_i.
$$
Therefore, the random vector $Z+v$ satisfies the assumptions
of Lemma~\ref{lem-tensorization} with $\varepsilon_0= \sqrt{d}/k$ and
$p_i=C k^2\,\est_i/d$.
Applying this lemma and taking
$\varepsilon=\gamma\varepsilon_0$, we obtain
$$
\Prob\Big\{\Vert Z+\xyzv\Vert_2\leq \gamma\, \frac{\sqrt{d\vert \KK\vert}}{k}  \Big\}
\leq
C_1^{|K|}\gamma^{2|K|} \prod_{i\in \KK}  \est_i,$$
where $C_1\geq 1$ is an absolute constant.
This completes the proof.
\end{proof}

Before we complete the proof of Theorem~\ref{sbp th}, we show that
$\prod_{i=1}^n \est_i$ and $\prod_{i\in \KK} \est_i$ are comparable
whenever $\KK$ is not too small.

\begin{lemma}\label{lem: est-K^c}
Let $d,n$ be large enough integers with $d^3\leq n$. Let $k\geq 1$  and $\KK\subset [n]$ be such that
$$
 \vert \KK^c\vert \leq  n/(50 \ln d) \quad \mbox{ and } \quad  k\leq d^{-10} e^{n/(5\vert \KK^c\vert)}
$$
and let  $y\in\kset_k$.  Then
$$
\prod_{i\in \KK^c} \est_i^{-1} \leq e^n.
$$
\end{lemma}

\begin{proof}
Since $\RR_{iq}\leq d$, we have $w_{iq}\leq \hq\, \sqrt{d}$ for all $i\leq n$ and  $q\leq m$.
For each $b\geq 1$, denote
$$
 I_b:=\{q\le m:\, \hq\in[2^{b-1},\,2^b)\}    \quad \quad \mbox{ and } \quad \quad
 \hset_b:= \bigcup_{q\in I_b} \lset^{(q)}.
$$
 Let $b_0$ be such that
$$
  2^{b_0-1}\sqrt{d}\leq e^{n/(2\vert \KK^c\vert)}<2^{b_0}\sqrt{d}.
$$
Note that the assumption on the cardinality of $K$ implies that $2^{b_0}\geq d^{24}$.
Since for every
$$
   q\in I_0 :=\bigcup_{b<b_0}  I_b
$$
 we have
$w_{iq}\leq \hq \sqrt{d} \le e^{n/(2\vert \KK^c\vert)}$,
then
\begin{equation}\label{eq1: lem-est-K^c}
 S:=\prod_{i\in \KK^c} \est_i^{-1}
 \leq  \prod_{i\in \KK^c}e^{n/(2\vert \KK^c\vert)}   \max(1,\max_{q\not\in I_0} w_{iq})
 \leq e^{n/2} \prod_{b\geq b_0}  \prod_{i\in \KK^c} \max(1,\max_{q\in I_b} w_{iq}).
\end{equation}
Fix $b\geq b_0$. By the construction of the matrix $Q$ (from a $d$-regular matrix $M$),
there are at most $d\vert \hset_b\vert$ indices $i$ for which $\RR_{iq}\neq 0$ for some
$q\in I_b$. Therefore,
\begin{equation}\label{eq2: lem-est-K^c}
\prod_{i\in \KK^c} \max(1,\max_{q\in I_b} w_{iq})\leq (2^b\sqrt{d})^{d\vert \hset_b\vert}.
\end{equation}
Let $j$ be the maximal order of $\ell$-parts  in the definition of $\hset_b$
and $\lset$ be a corresponding $\ell$-part (if there are two of them, spread and regular, we choose and fix one).
Denote $h:=h(\lset)$.
By the definition of $\hset_b$, the construction of the $\ell$-decomposition, (\ref{levset}), and
(\ref{levsethei}), we have
$$
  2^{b_0-1}\leq 2^{b-1}\leq  h< 2^b
 \quad \mbox{ and } \quad \vert \hset_b\vert \leq
  2\max_{q\in I_b} h_q\,  \sum _{i=0}^{j} 2^{i+1} \leq 2^{j+b+3}.
$$
Moreover, the size of each level set in $\lset$ is in the interval $[2^{j-1},2^{j+1}]$,
in particular, $|\lset|\in [2^{j-1} h,\,  2^{j+1} h]$.
Since $y$ is a $k$-vector, its levels are $1/k$-separated. Thus, using that
for any $p>0$ there are at most $(2p+1)^2$ integer complex numbers of absolute value less or equal
$p$, we obtain for $s= \lceil 2^{j-3}h\rceil$:
$$
  y_{s}^*> \frac{ \sqrt{h/8} }{ k }\geq \frac{ 2^{b/2}}{4k} \geq \frac{ 2^{b_0/2}}{4k} \geq
   \frac{ e^{n/(4\vert \KK^c\vert)}}{4 k d^{1/4}} \geq 2,
$$
provided that $k\leq e^{n/(4\vert \KK^c\vert)}/(8 k d^{1/4})$. Now we use that $y$ is the
$k$-approximation of a vector $x\in \mathcal S$. Since $x_{\nn}^{*}=1$, we observe
$s< n_3$. On the other hand,  applying Lemma~\ref{l:decay} to $x$,
$$
  y_{s}^*\leq 2 x_{s}^*\leq d^3 (n/s)^6 ,
$$
which implies
$$
  2^{j+b}\leq 2^{j+1} h \leq 16s \le  16n \sqrt{d}\, (4k)^{1/6}\, 2^{-b/12}.
$$
 Therefore,
$$
  \vert \hset_b\vert \leq 2^{j+b+3} \leq C' n \sqrt{d}\, k^{1/6}\, 2^{-b/12}
$$
for a universal constant $C'>0$.
This, together with \eqref{eq1: lem-est-K^c}, \eqref{eq2: lem-est-K^c}, and
$2^b\geq 2^{b_0}\geq  e^{n/(2\vert \KK^c\vert)}/\sqrt{d}\geq d$, implies
for an appropriate absolute positive constant $C$,
\begin{align*}
 S &\leq e^{n/2}\, \exp\Big(
\sum_{b\geq b_0} C' n d^{3/2} k^{1/6}2^{-b/12}\ln (2^b\sqrt{d})\Big)
\leq e^{n/2}\, \exp\Big(C n d^{3/2}  k^{1/6}  b_0 \, 2^{-b_0/12}\Big)
\\&
\leq  e^{n/2}\, \exp\big((n/2)  d^{3/2}  k^{1/6}
\big(\sqrt{d} e^{-n/(2\vert \KK^c\vert)}\big)^{1/13}\big)\leq e^{n},
\end{align*}
provided that $k\leq d^{-10} e^{n/5\vert \KK^c\vert}$ and that $d$ is large enough.
\end{proof}

\begin{proof}[Proof of Theorem~\ref{sbp th}]
Following our configuration-type model construction, the law of the adjacency matrix $A$ of the
multigraph $\widehat{G}_{\RR}$, conditioned on the event that $\widehat{G}_{\RR}$ is simple,
coincides with the uniform distribution on $\MSet_{n,d}(\RR,y)$ (see \eqref{eq: conf-model}).
Using this and applying Proposition~\ref{prop-simple} and Lemma~\ref{lem-norm of Ay}, we obtain
\begin{align*}
\Prob_{\RR,y}\bigl\{&M\in \MSet_{n,d}(\RR,y):\,
\|M^\KK y+\xyzv\|_2\leq \gamma\sqrt{d\vert \KK\vert}/k\bigr\}\\
&=\Prob\Big\{\Vert A^\KK y+\xyzv\Vert_2\leq \gamma\, \sqrt{d\vert \KK\vert}/k \, \big|\,  \Event_{\ref{prop-simple}}\Big\}
\\
&\leq {\Prob\Big\{\Vert A^\KK y+\xyzv\Vert_2\leq\gamma \, \sqrt{d\vert \KK\vert}/k\Big\}}/
{\Prob(\Event_{\ref{prop-simple}})}\\
&\leq \exp(33 d^2 \ln^2 n)\, C_{\ref{lem-norm of Ay}}^n\gamma^{2\vert \KK\vert}\prod_{i\in \KK} \est_i.
\end{align*}
Lemma~\ref{lem: est-K^c} implies the desired result.
\end{proof}

\section{Proof of Theorem \ref{ker th}}
\label{s:proof}

This section is devoted to the proof of the main result of the paper, obtained by a combination of the estimates for
steep and almost constant vectors from Section~\ref{steep}, the structural information on the set of gradual vectors
from Section~\ref{s:k-vektors} and the small ball probability theorem of Section~\ref{s:small ball}.
Setting aside technical details, the principal idea of the proof is to define a discrete structure on the set of gradual
vectors with ``not small'' subsets of almost equal coordinates and, by a combination of small ball probability estimates
for individual vectors (Lemmas~\ref{kapset lemma} and~\ref{rhoset lemma}), the union bound and an approximation argument (Proposition~\ref{prop: approx}),
to eliminate those vectors from the set of ``potential'' null vectors
of our matrix. Construction of the discrete subset (which can be viewed as a collection of nets with respect to
$\ell_\infty$-metric in $\C^n$) is quite involved -- it uses rather complex information about the
structure of a gradual vector (the $\ell$-decompositions of its $k$-approximations) which affects both cardinalities
of the nets and the probability estimates. As for the latter, to simplify analysis of the product $\prod_{i=1}^n\est_i$,
we introduce another set of estimators $\{\triv_i\}_{i=1}^n$, which we call {\it trivial estimators}
(see Subsection~\ref{s:rouh estimators} for definitions). The product $\prod_{i=1}^n\est_i$ is then estimated in terms of
$\prod_{i=1}^n\triv_i$ and an auxiliary functional $\eta$ (also defined in Subsection~\ref{s:rouh estimators}).
Note that, by the definition, the probability estimators $\est_i$ depend both on the structure of the underlying $k$-vector
and on statistics of the corresponding matrix $\RR$.
By introducing the estimators $\triv_i$ and the functional $\eta$, we ``separate'' these dependencies:
the trivial estimators are entirely determined by the $\ell$-decomposition of the related $k$-vector
while $\eta$ carries information about the matrix $\RR$ and the $\ell$-decomposition in a much more convenient
form compared to $\est_i$'s.

The next informal argument,
following the universality paradigm of the random matrix theory,
may be useful as an illustration of our approach.
Given the random matrix $M$, we may think that null vectors of $(M-z\,\idmat)^K$
behave essentially like Gaussian vectors (up to rescaling). In particular, this would imply that
the $\ell$-decompositions of $k$-approximations
of the null vectors are comprised of $\ell$-parts which are ``mostly'' regular and, moreover,
there are very few $\ell$-sets of comparable (up to a constant multiple) cardinalities.
Accordingly, vectors whose $\ell$-decompositions contain ``many'' spread $\ell$-parts
or many $\ell$-parts of approximately equal cardinalities, should be typically in the complement of the matrix kernel.
This imprecise observation can in fact be rigorously verified, in particular, we show that
vectors with large spread $\ell$-parts (from the sets $\kapset_u$) are not in the kernel with high probability.

\smallskip

Before we pass to the probability estimators and computing the union bound over a discrete subset
of $\mathcal S$, we reformulate the main statement of Section~\ref{s:small ball}. We first construct the following subset $\RSet^{ST}_{n,m,d}$ (``ST'' stands for ``standard'') of  $\RSet_{n,m,d}$.  First consider
the subset of matrices $\RR=(\RR_{iq})\in \RSet_{n,m,d}$ such that

\medskip

\noindent
{\bf 1. }
For every $q\leq m$ with $\|\col_q(\RR)\|_1=\sum_{i=1}^n \RR_{iq}\geq \sqrt{d}n$ one has
$$\bigl|\bigl\{i\leq n:\,\RR_{iq}<c_{\ref{graph th to prove}} \|\col_q(\RR)\|_1/n\bigr\}\bigr|\leq n/\sqrt{d};$$

\smallskip

\noindent
{\bf 2. }
For every non-empty subset $J\subset [m]$
and $\kappa:=\sum\limits_{q\in J}\|\col_q(\RR)\|_1$ one has
\begin{align*}
\Bigl|\Bigl\{&i\leq n:\,\sum\limits_{q\in J}\RR_{iq}\geq c_{\ref{graph prop}}\frac{\kappa}{n}\mbox{ and }
\sum\limits_{q\in J^c}\RR_{iq}\geq c_{\ref{graph prop}} \frac{dn-\kappa}{n}\Bigr\}\Bigr|
\geq c_{\ref{graph prop}}\min\bigl(\kappa,(dn-\kappa),n\bigr).
\end{align*}

\smallskip

\noindent
We denote this subset by $\RSet^{ST_0}_{n,m,d}$.
Note that in Corollary~\ref{graph th to prove} and Proposition~\ref{graph prop},
we showed that the event
\begin{align*}
\Event=\Bigl\{&M\in\MSet_{n,d}:\,\forall J\subset[n]\mbox{ one has}\\
&\Bigl|\Bigl\{i\leq n:\,|\supp\,\row_i(M)\cap J|\geq c_{\ref{graph prop}}\frac{d|J|}{n}\mbox{ and }
|\supp\,\row_i(M)\cap J^c|\geq c_{\ref{graph prop}}\frac{d|J^c|}{n}\Bigr\}\Bigr|\\
&\geq c_{\ref{graph prop}}\min(d|J|,d(|J^c|),n)\,\,\,\mbox{ AND}
\\
&\mbox{if} \,\, \vert J\vert\geq n/\sqrt{d} \text{ one has }
\Bigl|\Bigl\{i\le n :\,|\supp\,\row_i(M)\cap J|<\frac{c_{\ref{graph th to prove}}d|J|}{n}\Bigl\}\Bigl|
\leq n/\sqrt{d}\Bigl\}
\end{align*}
has probability very close to one.
Now, if $M\in\Event$ and  $y$ is a $k$-vector with $m$ non-empty
$\ell$-parts in its $\ell$-decomposition
then the correspondence $\stackrel{(y)}{\longrightarrow}$ necessarily maps $M$
into $\RSet^{ST_0}_{n,m,d}$. The image of $\Event$ we denote by $\RSet^{ST}_{n,m,d}$.
Thus, the preimage of $\RSet^{ST}_{n,m,d}$
with respect to $\stackrel{(y)}{\longrightarrow}$ is almost the entire set $\MSet_{n,d}$.
This fact combined with Theorem~\ref{sbp th} gives the following proposition.

\begin{prop}\label{prop sbp}
Let $d, n$ be large enough integers such that $d^3\leq n$.  Let
 $\KK\subset [n]$ be such that $\vert \KK^c\vert \leq n/(50 \ln d)$
and assume
$$
 1\le k\le \min\big( \sqrt{n}/(8d^{3/2} \sqrt{\ln d}),\, d^{-10} e^{n/(5\vert \KK^c\vert)}\big).
$$
Let  $y\in \kset_k$ and  $A>0$ be such that
$
\prod\limits_{i=1}^n\est_i\leq A
$
for every $Q\in \RSet^{ST}_{n,m,d}$. Then for any non-random vector
$\xyzv\in \C^{\vert \KK\vert}$ and any $\gamma \geq 1$ we have
$$
  \Prob\bigl\{M\in \Mc:\, \| M^\KK y+\xyzv\|_2\leq \gamma \sqrt{d \vert \KK\vert }/k\,\,\big|\,\,
  \Event \bigr\}\leq 2C^n \gamma^{2\vert \KK\vert} A,
$$
where $C>0$ is an absolute constant.
\end{prop}
\begin{proof}
As we already noted, the event $\Event$ is contained inside
$$
 \Event_0:=\{M:\mbox{ $M$ is mapped into $\RSet_{n,m,d}^{ST}$ via
$\stackrel{(y)}{\longrightarrow}$}\}.
$$
By Corollary~\ref{graph th to prove} and Proposition~\ref{graph prop}, we have
$\Prob(\Event_0)\geq \Prob(\Event)\geq 1/2$, in particular,
$\Prob(\Event_0)/\Prob(\Event)\leq 2$.
This and Theorem~\ref{sbp th} imply
\begin{align*}
\Prob&\bigl\{M\in \Mc:\,
\|M^\KK y+\xyzv\|_2\leq \gamma\sqrt{d\vert \KK\vert}/k\,\,\big|\,\,\Event\bigr\}\\
&\leq 2\Prob\bigl\{M\in \Mc:\,
\|M^\KK y+\xyzv\|_2\leq \gamma\sqrt{d\vert \KK\vert}/k\,\,\big|\,\,\Event_0 \bigr\}
\leq 2C^n \gamma^{2\vert \KK\vert} A.
\end{align*}
\end{proof}

\subsection{Rough estimators for small ball probability}
\label{s:rouh estimators}

We now introduce a ``rougher'' estimator than $\est_i$, which is easier to study. Note that conditioned on the event $\Event$ defined above, if
$|\lset^{(q)}|\ge n/\sqrt{d}$ then for most rows, $\RR_{iq}$ is of the same order of magnitude as
$d|\lset^{(q)}|/n$. Having this in mind, we
replace non-zero $\RR_{iq}$ in the definition of weights $w_{iq}$ in (\ref{weights}) with $d|\lset^{(q)}|/n$ if $|\lset^{(q)}|\ge
d^{-1/3}n$ and with $1$
otherwise and come to the following definition. Given a vector $y\in\kset_k$ with the corresponding
$\ell$-decomposition $(\lset^{(q)})_{q=1}^m$, let
$$
 \widetilde w_{q}=\widetilde w_{q}(y,k):=
\begin{cases}\hq |\lset^{(q)}|/n,
&\mbox{if $|\lset^{(q)}|\geq d^{-1/3}n$ and $\lset^{(q)}$ is regular},\\
\hq/d,
&\mbox{if $|\lset^{(q)}|< d^{-1/3}n$ and $\lset^{(q)}$ is regular},\\
\hq\, \sqrt{d\, |\lset^{(q)}|/n} ,
&\mbox{if $|\lset^{(q)}|\geq d^{-1/3}n$ and $\lset^{(q)}$ is spread},\\
\hq
&\mbox{if $|\lset^{(q)}|< d^{-1/3}n$ and $\lset^{(q)}$ is spread}
\end{cases}
$$
be {\it the truncated weight} (``truncated" because for small non-zero $\RR_{iq}$ we replace it with its minimal possible value $1$).
The reason why we truncate at level $d^{-1/3}n$ instead of $n/\sqrt{d}$ is that
we want the weights $\widetilde w_{q}$ to be significantly weaker than $w_{iq}$.
Specifically, this definition of the truncated weights will allow us to  estimate the product $\prod_{i=1}^n\est_i$
from above in terms of {\it trivial estimators} defined with respect to the truncated weights (see below).
Note that for any $q\leq m$, we have
\begin{equation}\label{eq: trivial-truncated-weight}
\widetilde w_{q}\geq {\hq}/{d}\quad  \text{ if $\lset^{(q)}$ is regular} \quad \quad
\text{ and }\quad \quad  \quad \widetilde w_{q}\geq \hq \quad  \text{ if $\lset^{(q)}$ is spread}.
\end{equation}
The principal difference between the weights $\widetilde w_{q}$ and $w_{iq}$
is that $\widetilde w_{q}$ does not depend  on $\RR_{iq}$, which makes its analysis easier.
We define also {\it the trivial estimator} $\triv_i$, $i\leq n$, as a weighted geometric mean of $\widetilde w_q$, $q\le m$,
$$
\triv_i:=\triv_i(y,k,\RR):=\prod\limits_{q\leq m} ({\widetilde w_q})^{-\RR_{iq}/d}.
$$
Note that by (\ref{sumqiq}) and (\ref{eq: trivial-truncated-weight}),
\begin{equation}\label{eq: trivial-prod}
 \prod_{i=1}^n\triv_i=\prod\limits_{q\leq m} ({\widetilde w_q})^{-|\lset^{(q)}|}\le d^n
 \prod\limits_{q\leq m} {\hq}^{-|\lset^{(q)}|}.
\end{equation}
In what follows, we usually do not mention explicitly dependency of the weights and the estimators
on the vector $y$, which is assumed to be fixed throughout most of the subsection.
To compare $ \est_i$ and $\triv_i$, we introduce  more notations.
For each $b\in\Z$, let $\wset_b$ be the union of all spread and regular $\ell$-parts $\lset^{(q)}$,
whose truncated weights
$\widetilde w_q$ lie in the interval $[2^b,2^{b+1})$, that is, we set
$$
 I(b):=\{ q\, : \, 2^b\le \widetilde w_q<2^{b+1}\}\quad \quad \mbox{ and } \quad \quad
 \wset_b:=\bigcup_{I(b)} \lset^{(q)}
$$
 (we will call $\wset_b$ {\it the $w$-set of order $b$}).
    For a fixed $i\le n$, define
   $$
   b(i):=\max\{b\in \Z\, :\,\, \exists q\in I(b) \, \, \mbox{ such that }\, \,  \RR_{iq}\neq 0\} .
$$
Put $b_{\min}:=\lfloor\log_2 1/d\rfloor$, $b_{\max}:=\max\limits_{i\leq n} b(i)$
and define for all $b_{\min}\leq b\leq b_{\max}$,
 \begin{align*}
    \wbl&:=\bigcup_{s=b_{\min}}^{b}\wset_s=\bigcup_{I_{\min}(b)}\lset^{(q)},\quad
    I_{\min}(b)=\{q\, : \,  2^{b_{\min}}\le \widetilde w_q<2^{b+1} \}
   \\
  \wbg&:=\bigcup_{s=b+1}^{b_{\max}}\wset_s=\bigcup_{I^c_{\min}(b)}\lset^{(q)},\quad
    I^c_{\min}(b)=\{q\, : \,  2^{b+1}\le \widetilde w_q<2^{b_{\max}+1} \}.
 \end{align*}
 Note that for every $b_{\min}\leq b\leq b_{\max}$ one has
 $$
   I_{\min}(b)  = \bigcup_{s=b_{\min}}^{b} I(b), \quad
    I^c_{\min}(b)=\bigcup_{s=b+1}^{b_{\max}} I(b), \quad \mbox{ and } \quad \wbl \cup\wbg =[n].
 $$
The following quantities will play an important role below,
\begin{equation*}\label{of}
 \of_i:= \frac{1}{d} \sum_{b=b_{\min}}^{b_{\max}}\, \, \min\Big(\sum\limits_{q\in I_{\min}(b)}\RR_{iq}, \,\sum\limits_{q\in I^c_{\min}(b)}\RR_{iq}\Big)
 \quad \mbox{ and }  \quad \of=\sum_{i=1}^n\of_i.
\end{equation*}

We start with a useful bound on cardinalities of $\ell$-parts $\lset^{(q)}$ inside $I(b)$ for a given $b$,
in which we also use that for all positive integers $N$, $N_i$, $i\leq \ell$, with $N=N_1+\ldots + N_\ell$
one has
\begin{equation} \label{factorial}
 N!/\prod_{i=1}^{\ell} N_i! \leq \prod_{i=1}^{\ell} (eN/N_i)^{N_i}
\leq (eN)^N\prod_{i=1}^{\ell}1/(N_i)^{N_i}
\end{equation}
(this follows by the standard inequality ${N \choose \ell} \leq (eN/\ell)^\ell$).

\begin{lemma}\label{cardlqpart}
Let $b_{\min}\leq b\leq b_{\max}$. Then the multiset
$\{|\lset^{(q)}|:\,q\in I(b)\}$, when arranged in the  non-increasing order,
can be majorized by the geometric sequence $\bigl(C|\wset_b|\exp(- s/C)\bigr)_{s\geq 0}$
for a sufficiently large absolute constant $C>0$. In particular,
$$
 \prod\limits_{I(b)}|\lset^{(q)}|!\geq |\wset_b|!\, \exp(-C' |\wset_b|),
$$
where $C'>0$ is another absolute constant.
\end{lemma}

\begin{proof}
We apply (\ref{levsethei}), which roughly speaking says that an $\ell$-part obtained at step $j$ satisfies
$|\lset_q|\approx 2^j h_q$. Note also that at most two $\ell$-parts can be obtained on a given step $j$.

Split the set $\{\lset^{(q)}:\,q\in I(b)\}$ into four subsets
$U_1,U_2,U_3,U_4$ determined by whether an $\ell$-part is spread or regular and whether its cardinality
is greater than $d^{-1/3}n$ or not. For concreteness, assume that $U_1,U_2$ contain regular $\ell$-parts,
with larger $\ell$-parts in $U_1$, and $U_3,U_4$ include spread $\ell$-parts, with larger ones in $U_3$.
Within the set $U_2$, the heights of the respective $\ell$-parts are equivalent up to multiple $2$. Therefore
their cardinalities, when arranged in non-increasing order, are majorized by an appropriate geometric sequence.
The same argument works for $\ell$-parts in $U_4$. For the $\ell$-parts in $U_1$, the quantities
$h_q |\lset_q|$ are equivalent to each other, say to a number $a$.  This means that for an $\ell$-part
obtained at the step $j$ we have $h_q^2\approx a 2^{-j}$. This in turn implies $|\lset_q|\approx 2^{j/2} \sqrt{a}$.
Thus $|\lset_q|$ (after a rearrangement) are geometrically decreasing.
Finally, for set $U_3$, the quantities $h_q\sqrt{|\lset_q|}$ are  stable and a similar argument works.
Combining the four decreasing sequences into one, we obtain a sequence that can be also majorized by a
geometric series (with worse constants).

To prove the ``in particular" part, let $N_s$, $s\geq 0$, corresponds to cardinalities of $\lset^{(q)}$, $q\in I(b)$.
Then $N= \sum _s N_s = |\wset_b|$ and, by the first part, $N_s\leq C N \exp(-s/C)$. Therefore, by (\ref{factorial}),
$$
 \ln \bigg(|\wset_b|!/\prod\limits_{I(b)}|\lset^{(q)}|!\bigg) \leq
 \sum _{s\geq 0} N_s \ln (e N/N_s) \leq  N  \sum _{s\geq 0} (s+1) e^{-s/C} \leq C'N,
$$
where $C'>0$ is an absolute constant. This completes the proof.
\end{proof}

In the next lemma we relate $\est_i$ and $\triv_i$ estimators using the parameter $\of$ introduced above.

\begin{lemma}\label{lem: precise-bound-triv-est}
Let $y\in\kset_k$, $(\lset^{(q)})_{q=1}^m$ be its  $\ell$-decomposition, and  $Q$ be a $y$-admissible matrix in $\RSet_{n,m,d}^{ST}$.
Then
$$
 \prod_{i=1}^n\est_i\leq C^n_{\ref{lem: precise-bound-triv-est}}\, 2^{-\of}\prod_{i=1}^n\triv_i,
$$
where $C_{\ref{lem: precise-bound-triv-est}}\ge 1$ is a universal constant.
\end{lemma}

\begin{proof}
For every $i\le n$, set
$$
  \widetilde\est_i = \widetilde\est_i (y, k, Q) :=\min\bigl\{{\widetilde w_{q}}^{-1}:\,q\le m\mbox{ and }\RR_{iq}\neq 0\bigr\}.
$$
Since $\RR_{iq}\geq 1$ in the above minimum, by the definition of the weights, we get
for all $i\leq n$,
\begin{equation}\label{simpleb}
 \widetilde\est_i\geq \est_i/d.
\end{equation}

We first show that
\begin{equation}\label{sti}
\prod\limits_{i=1}^n \widetilde\est_i\geq \exp(-C n)\prod\limits_{i=1}^n \est_i,
\end{equation}
where $C$ is a positive universal constant.
Let $\lset^{(q)}$ be one of $\ell$-parts from the  $\ell$-decomposition (no matter whether spread or regular),
having cardinality at least $d^{-1/3}n$.
Then, by the definition of $\RSet_{n,m,d}^{ST}$ (part $1$), $d$-regularity of matrices in $\Mc$, and
(\ref{sumqiq}),
there are at most $n/\sqrt{d}$ indices $i\leq n$ such that
$\RR_{iq}<c_{\ref{graph th to prove}} d|\lset^{(q)}|/n$.
For all other $i$'s, we have $\RR_{iq}\geq c_{\ref{graph th to prove}} d|\lset^{(q)}|/n$,
implying together with the definitions of $w_{iq}$ and $\widetilde w_q$ that
$w_{iq}\geq c_{\ref{graph th to prove}}\, \widetilde w_q$.
On the other hand, if the cardinality of $\lset^{(q)}$ is less than $d^{-1/3}n$
and $\RR_{iq}\neq 0$ (hence greater or equal to $1$) then necessarily $w_{iq}\geq\widetilde w_q$.
Denote
$$
 I:=\{ i\in [n]:\, w_{iq}< c_{\ref{graph th to prove}}\,\widetilde w_q\, \, \, \, \text{ for some $\, \, \, \,
 q\leq m\, \, \, \, $ with $\, \, \, \, \RR_{iq}\neq 0$}\}.
$$
By definitions of $\est_i$ and $\widetilde\est_i$, we have
 $\widetilde\est_i\geq c_{\ref{graph th to prove}}\, \est_i$ for all $i\in I^c$.
Since there are at most $d^{1/3}$
$\ell$-parts of cardinality at least $d^{-1/3}n$ each, then from the above we also have
 that $\vert I\vert \leq  d^{1/3} d^{-1/2}n= d^{-1/6}n$.
Therefore, using (\ref{simpleb}) for $i\in I$, we obtain
$$
\prod\limits_{i=1}^n \widetilde\est_i\geq (c_{\ref{graph th to prove}})^{\vert I^c\vert}\, d^{-\vert I\vert}\, \prod\limits_{i=1}^n \est_i
\geq c_{\ref{graph th to prove}}^{n}d^{-d^{-1/6}n}\prod\limits_{i=1}^n \est_i,
$$
which leads to (\ref{sti}).

To complete the proof, it is sufficient to show that for every $i\le n$
$$
\widetilde\est_i\leq 2^{-\of_i+1}\, \triv_i.
$$
Fix $i\leq n$.  By the definition of $\widetilde\est_i$ and $b(i)$ we have
$\widetilde\est_i\leq 2^{-b(i)}.$
Since
\begin{equation}\label{idenone}
 \sum_{b=b_{\min}}^{b(i)} \, \, \sum\limits_{q\in I(b)}
 \RR_{iq}/d=\sum_{q=1}^m \RR_{iq}/d=1,
\end{equation}
the definition of $\triv_i$  implies
\begin{align*}
   \triv_i&=\prod_{b=b_{\min}}^{b(i)}\, \prod\limits_{q\in I(b)}
   (\widetilde{w}_q)^{-{\RR_{iq}}/{d}}
   >\prod_{b=b_{\min}}^{b(i)}\, \prod\limits_{q\in I(b)}
 2^{-{(b+1)\RR_{iq}}/{d}}
\\
&=\frac{1}{2}\, \exp\Bigl(-\ln 2\,\sum_{b=b_{\min}}^{b(i)} b
\sum\limits_{q\in I(b)}
\frac{\RR_{iq}}{d}\Bigr).
\end{align*}
Using (\ref{idenone}) again and applying the simple identity
$$
 \sum_{j=j_0}^{j_1}(j_1-j)a_j=\sum_{j=j_0}^{j_1-1}\sum_{k=j_0}^ja_k,
$$
 valid for any
integers $j_0<j_1$ and any
numbers $a_{j_0}$,..., $a_{j_1}$,   we get
\begin{align*}
    d b(i)-\sum_{b=b_{\min}}^{b(i)} &b
      \sum\limits_{q\in I(b)}
        {\RR_{iq}} =
     \sum_{b=b_{\min}}^{b(i)} (b(i)-b) \sum\limits_{q\in I(b)}
           {\RR_{iq}}
 =
        \sum_{b=b_{\min}}^{b(i)-1}\, \, \sum_{a=b_{\min}}^b\, \,
         \sum\limits_{q\in I(a)}                {\RR_{iq}}
 \\&=
        \sum_{b=b_{\min}}^{b(i)-1}\, \,
         \sum\limits_{q\in I_{\min}(b)}                {\RR_{iq}}
  \ge
   \sum_{b=b_{\min}}^{b_{\max}}\, \, \min\Big(\sum\limits_{q\in I_{\min}(b)}\RR_{iq}, \,\sum\limits_{q\in I^c_{\min}(b)}\RR_{iq}\Big)
       = d\of_i.
\end{align*}
where we used that if $b>b(i)$ then for every $q\in I(b)$ one has $\RR_{iq}=0$.
Therefore,
\begin{align*}
    \widetilde\est_i\leq 2^{-b(i)} &\leq
       2\, \triv_i\exp\Bigl(-\ln 2\,\Big(b(i)-\sum_{b=b_{\min}}^{b(i)} b
      \sum\limits_{q\in I(b)}        {\RR_{iq}}/{d}\Big)\Bigr)
\leq 2^{-\of_i+1}\triv_i,
\end{align*}
This completes the proof.
\end{proof}

In the next two lemmas, we estimate $\of$  in the case when $\RR$ belongs to $\RSet_{n,m,d}^{ST}$.

\begin{lemma}\label{simple offset}
Let $y\in\kset_k$,  $(\lset^{(q)})_{q=1}^m$ be its $\ell$-decomposition
and  $\wset_b$, $b\in\Z$, be its corresponding $w$-sets.
Further, let $\RR=(\RR_{iq})$ be a $y$-admissible matrix in $\RSet_{n,m,d}^{ST}$.
Then
$$
 \of\geq c_{\ref{simple offset}}\, \sum_{b=b_{\min}}^{b_{\max}} \min \big(|\wbl|,\,|\wbg|\big),
$$
where $c_{\ref{simple offset}}>0$ is a universal constant.
\end{lemma}

\begin{proof}
By the definition of $\of$  we have
$$
 \of =\frac{1}{d}\sum_{b=b_{\min}}^{b_{\max}} \sum_{i=1}^n
  \min\Big(\sum\limits_{q\in I_{\min}(b)}\RR_{iq}, \,\sum\limits_{q\in I^c_{\min}(b)}\RR_{iq}\Big) .
$$
To prove the lemma we prove the corresponding inequality for each summand in the first sum.
To this end, for every (fixed) $b_{\min}\le b<b_{\max}$ we apply the definition of $\RSet_{n,m,d}^{ST}$ (more precisely, part~{\bf 2} of the definition of $\RSet_{n,m,d}^{ST_0}$),
with
$$
 J=J(b):=I_{\min}(b) \quad \mbox{ and } \quad
\kappa=\kappa(b) :=d|\wbl|.
$$
 Note that by (\ref{sumqiq}) and $d$-regularity we have
$$
 \kappa=d|\wbl|= \sum\limits_{q\in J}\|\col_q(\RR)\|_1 \quad \mbox{ and } \quad
  dn - \kappa = d|\wbg|.
$$
We distinguish two cases.

\smallskip

\noindent
{\it Case 1.}   $\min\bigl(|\wbl|,|\wbg|\bigr)\geq n/d$.
In this case $\min\bigl(\kappa, dn-\kappa, n )=n$. Thus,
the definition of $\RSet_{n,m,d}^{ST}$ yields that
 the cardinality of the set
\begin{align*}
 I:=\Bigl\{i\leq n:\,&\sum\limits_{q\in I_{\min}(b)} \RR_{iq}\geq c_{\ref{graph prop}}\frac{d|\wbl|}{n}\;
\mbox{ and }
\sum\limits_{q\in I^c_{\min}(b)} \RR_{iq}\geq c_{\ref{graph prop}}\frac{d|\wbg|}{n}\Bigr\}
\end{align*}
is at least $c_{\ref{graph prop}} n$. Therefore,
\begin{align*}
\sum_{i=1}^n
\min\Bigl(\sum\limits_{q\in I_{\min}(b)} {\RR_{iq}}
,\sum\limits_{q\in I^c_{\min}(b)} {\RR_{iq}}\Bigr)
&\geq\sum_{i\in I}
\min\Bigl(\sum\limits_{q\in I_{\min}(b)} {\RR_{iq}}
,\sum\limits_{q\in I^c_{\min}(b)} {\RR_{iq}}\Bigr)\\
&\geq
c_{\ref{graph prop}}^2d\min\bigl(|\wbl|,|\wbg|\bigr).
\end{align*}

\smallskip

\noindent
{\it Case 2.}  $1\leq \min\bigl(|\wbl|,|\wbg|\bigr)< n/d$. In this case  $\kappa/n<1$  or $(dn-\kappa)/n<1$.
Using that $\RR_{iq}$ are non-negative integers, that $c_{\ref{graph prop}}<1$, and the definition
of $\RSet_{n,m,d}^{ST}$, we observe
\begin{align*}
\Bigl|\Bigl\{i&\leq n:\,\sum\limits_{q\in I_{\min}(b)} \RR_{iq}\geq 1\;\mbox{ and }
\sum\limits_{q\in I^c_{\min}(b)} \RR_{iq}\geq 1\Bigr\}\Bigr|
\\&\geq
 \Bigl|\Bigl\{i\leq n:\, \sum\limits_{q\in I_{\min}(b)} \RR_{iq}\geq c_{\ref{graph prop}}\,\frac{\kappa}{n}\;
\mbox{ and }
\sum\limits_{q\in I^c_{\min}(b)} \RR_{iq}\geq c_{\ref{graph prop}}\, \frac{dn-\kappa}{n}\Bigr\}\Bigl|
\\&
\geq c_{\ref{graph prop}}\, d\min\bigl(|\wbl|,|\wbg|\bigr).
\end{align*}
Therefore,
\begin{align*}
\sum_{i=1}^n &
\min\Bigl(\sum\limits_{q\in I_{\min}(b)} {\RR_{iq}}
,\sum\limits_{q\in I^c_{\min}(b)} {\RR_{iq}}\Bigr)\geq  c_{\ref{graph prop}}d\min\bigl(|\wbl|,|\wbg|\bigr).
\end{align*}
Since the case $\min\bigl(|\wbl|,|\wbg|\bigr)=0$ is trivial,
this completes the proof.
\end{proof}

\begin{lemma}\label{offset lemma}
Let $y\in\kset_k$, $(\lset^{(q)})_{q=1}^m$ and $\wset_b$, $b\in\Z$,
be as above, and $\RR=(\RR_{iq})$ be a $y$-admissible matrix
from $\RSet_{n,m,d}^{ST}$.
Then
$$
 n!\prod_{q\leq m} \frac{1}{|\lset^{(q)}|!}\le C^n\, 2^{\of/2},
$$
where $C$
is a positive universal constant.
\end{lemma}
\begin{proof}
Denote $c:=(c_{\ref{simple offset}}\ln 2)/2$. By Lemma~\ref{simple offset}, it is enough to show that
$$
  n!\prod_{q=1}^m \frac{1}{|\lset^{(q)}|!}\exp\Big(-c
 \sum_{b=b_{\min}}^{b_{\max}} \min \big(|\wbl|,\,|\wbg|\big)
 \Big)\leq C^n
$$
for a universal constant $C>0$.
Denote
$$
  I:=\Bigl\{b\in\Z\, :\, \, \wset_b\neq\emptyset\, \,  \, \mbox{ and } \, \, \,
  |\wset_b| \ln (n/|\wset_b|)\geq c \min \big(|\wbl|,\,|\wbg|\big)\Bigr\}
$$
and for every  integer $p\geq 0$,
$$
  I_p:=\bigl\{b\in\Z:\,|\wset_b|\in (n2^{-p-1},n2^{-p}]\bigr\}\cap I.
$$
Note that if $b\in I_p$ then $\min\bigl(|\wbl|,|\wbg|\bigr)
\leq n(p+1) 2^{-p}/c$. Hence,
\begin{align*}
\vert I_p\vert &\leq \Big\vert \Bigl\{b\in\Z:\,|\wset_b|\in (n2^{-p-1},n2^{-p}]
\, \, \, \text{ and }\, \, \,
\min\bigl(|\wbl|,|\wbg|\bigr)
\leq n(p+1) 2^{-p}/c\Bigr\}\Big\vert
\\
&\leq \Big\vert \Bigl\{b\in\Z:\,|\wset_b|\in (n2^{-p-1},n2^{-p}] \, \, \, \text{ and }\, \, \, |\wbl|
\leq n(p+1) 2^{-p}/c\Bigr\}\Big\vert \\
&\quad +\Big\vert \Bigl\{b\in\Z:\,|\wset_b|\in (n2^{-p-1},n2^{-p}] \, \, \, \text{ and }\, \, \, |\wbg|
\leq n(p+1) 2^{-p}/c\Bigr\}\Big\vert.
\end{align*}
Denote the cardinalities of the sets in the last inequality by $\alpha=\alpha(p)$ and
$\beta=\beta(p)$ correspondingly. Let $b_1< \ldots <b_{\alpha}$ and $b_1'< \ldots <b_{\beta}'$
be the elements of those set.
Then
$$
 \alpha\, n2^{-p-1}\leq \Big |\bigcup_{i=1}^{\alpha} \wset_{b_i}\Big|\leq \Big| \wset^1_{b_{\alpha}}
 \Big|\leq  n(p+1) 2^{-p}/c
$$
and
$$
 \beta\, n2^{-p-1}\leq \Big |\bigcup_{i=1}^{\beta} \wset_{b_i'}\Big|\leq \Big| \wset_{b_{1}'} \Big| +
 \Big| \wset^2_{b_{1}'} \Big|\leq n 2^{-p} +  n(p+1) 2^{-p}/c.
$$
This implies that
$$\vert I_p\vert \leq \alpha+\beta \leq  6(p+1)/c.$$
Therefore using that $\sum _b |\wset_b| =n$ and (\ref{factorial}) with $N=n$ and $N_b= |\wset_b|$, we obtain
\begin{align*}
&n!\, \bigg(\prod_{b=b_{\min}}^{b_{\max}}  \frac{1}{|\wset_b|!}\bigg)\,
\exp\Bigl(-c\sum_{b=b_{\min}}^{b_{\max}}
\min\bigl(|\wbl|,|\wbg|\bigr)\Bigr)\\
&\leq e^n\, \prod_{b:\wset_b\neq\emptyset}\bigg( \Bigl(\frac{n}{|\wset_b|}\Bigr)^{|\wset_b|}
\exp\Bigl(-c
\min\bigl(|\wbl|,|\wbg|\bigr)\Bigr)\bigg)\\
&\leq  e^n\, \prod_{b\in I}\Bigl(\frac{n}{|\wset_b|}\Bigr)^{|\wset_b|}=  e^n\, \prod_{p=0}^{\log_2 n} \prod_{b\in I_p} \Bigl(\frac{n}{|\wset_b|}\Bigr)^{|\wset_b|}\\
&\leq   e^n\, \prod_{p=0}^{\log_2 n} 2^{(p+1)n2^{-p}\,\vert I_p\vert}\leq   e^n\, \prod_{p=0}^{\log_2 n} 2^{6(p+1)^2n 2^{-p}/c}\leq \exp(\widetilde C n),
\end{align*}
where $\widetilde C>0$ is a sufficiently large absolute constant. Applying Lemma~\ref{cardlqpart}
and using $\sum _b |\wset_b|=n$ again, we get
$$
 \biggl(n!\prod_{q\leq m} \frac{1}{|\lset^{(q)}|!}\biggr)\exp\Bigl(-c\sum_{b=b_{\min}}^{b_{\max}}
\min\bigl(|\wbl|,|\wbg|\bigr)\Bigr)
\leq \exp\bigl((\widetilde C+C'') n\bigr),
$$
which completes the proof.
\end{proof}

\subsection{Completion of the proof}

  In the proof of Theorem~\ref{ker th} we use the results established in the previous sections. Recall our decomposition of $\C^n$ into three types of vectors: almost constant vectors, steep vectors, and gradual vectors. We treat each of these types separately. The former two
   types are treated in Section~\ref{steep}, where a lower bound on $\Vert (M-\ww\idmat)^\KK x\Vert_2$ is given. This leaves us with the case of gradual vectors which we approximate by $k$-vectors. First, one needs to check that it is sufficient to establish a lower bound for the action of $M^\KK$ on the $k$-approximation of gradual vectors in order to deduce a similar bound for all such vectors.
   The next proposition provides such an approximation argument.

\begin{prop}\label{prop: approx}
Let $d\geq 1$ be large enough,  $n\geq d^3$, and $1\leq \LL\leq n/d^3$.
Let $\KK\subset [n]$ be such that $\vert \KK^c\vert \leq L$
and $\ww$ be such that $\vert \ww\vert \leq r \sqrt{d}$ for some $r\geq 1$.
Let
$X$ be a subset of the set of normalized  gradual vectors
$X\subset \mathcal S$ and $ \kset_k(X)$
be the set of $k$-approximations of vectors from $X$. Then
\begin{align*}
\Prob&\Big\{M\in \Mc:\, \exists x\in X,\, \|(M-\ww \idmat)^\KK x\|_2\leq \frac{L^3}{k n^{5.5}\, d}\Vert x\Vert_2\Big\}\\
&\leq 2 d^2\,  \sum_{ y\in \kset_k(X)}\, \,\sup_{\xyzv\in \C^{\vert \KK\vert}}\Prob\Big\{M\in \Mc:\, \, \| M^\KK y+\xyzv\|_2\leq C_{\ref{prop: approx}}\,  r \frac{\sqrt{dn}}{k}\Big\}+ n^{-100},
\end{align*}
where $C_{\ref{prop: approx}}\geq 1$ is a universal constant.
\end{prop}

\begin{proof}
By definitions we have $x_{n_1}^*\leq d^3$ for $x\in \mathcal S$. Therefore,
 Lemma~\ref{l:norma} implies
\begin{align*}
\Prob\Big\{M\in \Mc:\, &\exists \, x\in X,\,\, \|(M-\ww \idmat)^\KK x\|_2\leq \frac{ L^3}{k n^{5.5}\, d}\, \Vert x\Vert_2\Big\}\\
&\leq \Prob\Big\{M\in \Mc:\, \exists \, x\in X,\,\, \|(M-\ww \idmat)^\KK x\|_2\leq \frac{\sqrt{dn}}{k}\Big\}.
\end{align*}
Suppose that $M\in \Event_{\ref{norm lemma}}$.
Let $x\in X$ and let  $y\in\kset_k(X)$ be its $k$-approximation. Define
$a=(a_1, a_2), \, \widetilde a=(\widetilde a_1, \widetilde a_2) \in \C$ by
\begin{align*}
 a =\frac{1}{n}\sum_{i=1}^n  (x_i-y_i) \quad  \quad \mbox{ and } \quad  \quad
 \widetilde a_1= \Re\,\widetilde a= \frac{\lceil kda_1\rceil}{kd},\quad
 \widetilde a_2=\Im\,\widetilde a= \frac{\lceil kda_2\rceil}{kd}.
\end{align*}
   Below we use the same notation ${\bf 1}$ for the vectors in $\C^\ell$ for every $\ell\geq 1$.
   Note that since $M\in \Event_{\ref{norm lemma}}$ we have that $\|M w \|_2 \leq C_{\ref{norm lemma}} \sqrt{d}\|v\|_2$
   for every $v\in C^n$ which is orthogonal to ${\bf 1}$.
Denote
$$
 \xyzv(y,\widetilde a):=-\ww y^{\KK}+\widetilde a (d-\ww){\bf 1}\in\C^{|\KK|}.
$$
Since $y$ is the $k$-approximation of $x$,  then for $i=1,2$ we have $0\leq ka_i\leq 1$ and thus $\widetilde a_i\in {\{0,\ldots,d\}}/{kd}$.
Using the triangle inequality, $d$-regularity, $\vert a-\widetilde a\vert \leq{\sqrt{2}}/(kd)$, and that $x-y-a\, {\bf 1}$ is orthogonal to ${\bf 1}$, we observe
\begin{align*}
\| &M^\KK y+\xyzv(y,\widetilde a)\|_2
\\
&\leq \|(M-\ww \idmat)^\KK x\|_2 + \|(M-\ww \idmat)^\KK (x-y-a\, {\bf 1})\|_2+ \vert a-\widetilde a\vert\, \|(M-\ww \idmat)^\KK {\bf 1}\|_2\\
&\leq \|(M-\ww \idmat)^\KK x\|_2 + C_{\ref{norm lemma}} \Big(\sqrt{d}+\vert \ww\vert\Big) \Vert x-y-a\, {\bf 1}\|_2+ \frac{\sqrt{2} \vert d-\ww\vert}{kd}\sqrt{\vert \KK\vert}.
\end{align*}
Using that
$$
  \Vert x-y-a\, {\bf 1}\|_2 \leq  \Vert x-y\|_2+ \vert a\vert \sqrt{n}\leq {2\sqrt{2n}}/{k}
$$
 together with $\vert \ww\vert \leq r\sqrt{d}$ we get
$$
\|M^\KK y+\xyzv(y,\widetilde a)\|_2 \leq \|(M-\ww \idmat)^\KK x\|_2+ \left(2C_{\ref{norm lemma}}+ 1\right)(r+1){\sqrt{2dn}}/{k}.
$$
 Theorem~\ref{norm lemma} and the union bound imply
\begin{align*}
\Prob&\Big\{M\in \Mc:\, \exists\, x\in X,\,\, \|(M-\ww \idmat)^\KK x\|_2\leq  {\sqrt{dn}}/{k}\Big\}\\
&\leq \Prob\Big\{M\in \Mc:\,  \Event_{\ref{norm lemma}} \quad \mbox{ and }\quad \exists\, x\in X,\,\, \|(M-\ww \idmat)^\KK x\|_2\leq  {\sqrt{dn}}/{k}\Big\}+ n^{-100}
\\
&\leq \Prob\Big\{M\in \Mc:\, \Event_{\ref{norm lemma}} \quad \mbox{ and }\quad  \exists \, y\in \kset_k(X), \exists\, \widetilde a \in {\{0,\ldots,d\}^2}/{(kd)},
\\
&\quad\quad\quad\quad\quad\quad\quad\quad\quad\quad\quad\quad\quad\quad\quad\quad\,\,
\|M^\KK y+\xyzv(y,\widetilde a)\|_2\leq C\, r{\sqrt{dn}}/{k}\Big\}+ n^{-100}
\\
&\leq 2 d^2\, \sum_{ y\in \kset_k(X)}\, \, \sup_{\xyzv\in \C^{\vert \KK\vert}}\Prob\Big\{M\in \Mc:\,
\| M^\KK y+\xyzv\|_2\leq C\, r{\sqrt{dn}}/{k}\Big\}+ n^{-100},
\end{align*}
where $C>0$ is an appropriate large universal constant.
\end{proof}

Applying Theorem~\ref{kappa and rho}, we further decompose gradual vectors into two types,
$\kapset_u$ and $\rhoset_u$, depending on some properties satisfied by their $k$-approximation.
Recall that the set of $k$-vectors is partitioned into equivalence classes and bounds on
the cardinality
of each class and on their total number are established in Lemmas \ref{l:equiv classes} and
\ref{equiv cardinality}. Therefore, in view of the previous proposition, we can concentrate
our effort on bounding the probability that $\Vert M^\KK y+\xyzv\Vert_2$ is small for a
fixed $k$-vector $y$ satisfying the properties given in $\kapset_u$ or $\rhoset_u$ and any
vector $\xyzv\in \C^{\vert \KK\vert}$. Theorem~\ref{sbp th} and
Proposition~\ref{prop sbp} establish such bound in terms of the small ball estimators
$\est_i$ of the vector $y$. Therefore, it remains to estimate $\est_i$ for these two types
of vectors using all the tools developed in Section~\ref{s:rouh estimators}. We start
with  vectors in $\kapset_u$. Recall that for $x\in\kapset_u$, the total
cardinality of the spread $\ell$-parts in the $\ell$-decomposition with respect to
the $d^u$-approximation of $x$  is at least $c_\kapset \nn$, where $c_\kapset \in (0,1)$
is an absolute constant.

\begin{lemma}\label{kapset lemma}
Let  $u\geq 2$
be an integer, $y$ be the $d^u$-approximation with respect to a vector in $\kapset_u$,
 $(\lset^{(q)})_{q=1}^m$ be its $\ell$-decomposition, and $\RR=(\RR_{iq})$
be a $y$-admissible matrix in $\RSet_{n,m,d}^{ST}$. Then there exists a universal constant $c_{\ref{kapset lemma}}>0$ such that
$$\prod\limits_{i=1}^n\est_i(y,k,\RR)
\leq d^{-c_{\ref{kapset lemma}}n}
(n!)^{-1}\prod_{q \leq m} \frac{|\lset^{(q)}|!}{\hq^{|\lset^{(q)}|}}.
$$
\end{lemma}
\begin{proof}
To prove the lemma, it is enough to show that
\begin{equation}\label{beta}
\beta:=n!\prod_{q\le m} \frac{1}{|\lset^{(q)}|!}\prod\limits_{i=1}^n \est_i
\leq d^{-c_{\ref{kapset lemma}}n}
\prod_{q \leq m} \hq^{-|\lset^{(q)}|}.
\end{equation}
Applying Lemmas~\ref{lem: precise-bound-triv-est} and  \ref{offset lemma} we get
\begin{equation*}\label{eq1: est-triv}
\beta\leq C^n 2^{-\of/2}\, \prod\limits_{i=1}^n \triv_i\le C^n  \prod\limits_{i=1}^n \triv_i,
\end{equation*}
   where  $C$ is a positive absolute constant.
   Let $c:= c_\kapset\aaa/16<1/3$, where $\aaa$ comes from the definition of $\nn$ (recall $\nn =\lfloor \aaa n\rfloor$), and
   $I:=\{q\, : \, |\lset^{(q)}|\geq  d^{-c}n\}$.
   We consider two cases.

\medskip

\noindent
{\it Case 1.} $\big|\bigcup_{q\in I}\lset^{(q)}\big|\geq n-c_\kapset\nn/4$.
Denote by $I_1$ the set of all indices $q$ corresponding to spread $\ell$-parts
of cardinality at least $d^{-c}n$ and by $I_2$ be the
set of indices corresponding to regular $\ell$-parts of cardinality at least $d^{-c}n$.
Let $I_3$ be all the remaining indices, that is $I_3=[m]\setminus (I_1\cup I_2)$. Note that
$\bigl|\bigcup_{q\in I_3}\lset^{(q)}\bigr|\le
c_\kapset\nn/4$.
By (\ref{eq: trivial-prod}) we have
\begin{equation*}
\beta\leq
C^n \prod\limits_{q\in I_1} {\widetilde w_q}^{-|\lset^{(q)}|}
\prod\limits_{q\in I_2} {\widetilde w_q}^{-|\lset^{(q)}|}
\prod\limits_{q\in I_3} {\widetilde w_q}^{-|\lset^{(q)}|}.
\end{equation*}
By the definition of $\kapset_u$,
the total cardinality of $\ell$-parts with indices from $I_1$ is at least
$$
 c_\kapset\nn -c_\kapset \nn/4=3c_\kapset \nn/4.
$$
This together with the definition of the truncated weights for
$$
 |\lset^{(q)}|\geq d^{-c}n\geq d^{-1/3}n,
$$
 implies that
$$
 \prod\limits_{q\in I_1} ({\widetilde w_q})^{-|\lset^{(q)}|}\leq \prod\limits_{q\in I_1} \Big(\frac{n}{d\vert \lset^{(q)}\vert}\Big)^{|\lset^{(q)}|/2}\, \hq^{-\vert \lset^{(q)}\vert}
 \leq d^{-3c_\kapset \nn/8}\prod\limits_{q\in I_1} \Big(\frac{n}{\vert \lset^{(q)}\vert}\Big)^{|\lset^{(q)}|}\,
 \hq^{-\vert \lset^{(q)}\vert}.
$$
Hence, using the definition of the truncated weights for $q\in I_2$ and the bounds
\eqref{eq: trivial-truncated-weight} for $q\in I_3$, we get
\begin{align*}
\beta&\leq C^n d^{-3c_\kapset \nn/8}\prod\limits_{q\in I_1\cup I_2}(n/|\lset^{(q)}|)^{|\lset^{(q)}|}\, \prod\limits_{q\in I_3} d^{\vert \lset^{(q)}\vert}\, \prod\limits_{q\leq m} \hq^{-|\lset^{(q)}|}\\
&\leq C^n d^{-3c_\kapset \nn/8}d^{cn}d^{c_\kapset \nn/4}
\prod\limits_{q\le m} \hq^{-|\lset^{(q)}|} \le d^{-c_\kapset \nn/16}\prod_{q \leq m} \hq^{-|\lset^{(q)}|},
\end{align*}
which leads to (\ref{beta}).

\medskip

\noindent
{\it Case 2.} $\big|\bigcup_{q\in I}\lset^{(q)}\big|< n-c_\kapset\nn/4$. In this case
$\big|\bigcup_{q\in I^c}\lset^{(q)}\big|\geq c_\kapset\nn/4$.
Using (\ref{eq: trivial-prod}), we have
$$
  \beta\leq C^n d^{n} 2^{-\of/2}\prod\limits_{q\le m} \hq^{-|\lset^{(q)}|}.
$$
Denote
$$
 J(b):=I^c\cap I(b)=\{q\, : \, |\lset^{(q)}|< d^{-c}n\, \,  \mbox{ and }\, \,  2^b\le \widetilde w_q<2^{b+1}\}.
$$
Arguing as in the proof of Lemma~\ref{cardlqpart}, we have
$$
  \bigg|\bigcup_{q\in J(b)} \lset^{(q)} \bigg|\leq C'd^{-c}n,
$$
for a universal constant $C'>0$. Define two integer numbers $b_1$ and $b_2$ by
\begin{align*}
b_1:=\min\Big\{b\in \Z:
|\wbl|\geq \frac{c_\kapset \nn}{16}\Big\} \,   \mbox{ and } \,
b_2:=\max\Big\{b\in\Z:|\wset_b\cup\wbg|\geq \frac{c_\kapset \nn}{16}\Big\}.
\end{align*}
Clearly, $b_2\geq b_1$. Denoting
$J:=I^c\cap \bigcup_{b_1\leq b\leq b_2}I(b)$,
we observe
$$
  \bigg|\bigcup_{q\in J}\lset^{(q)}\bigg| \leq (b_2-b_1+1)\, C'd^{-c}n.
$$
On the other hand, using the definition of $b_1$, $b_2$ together with  the condition of this case, we have
\begin{align*}
   \bigg|\bigcup_{q\in J}\lset^{(q)}\bigg|
   &\geq \big|\bigcup_{q\in I^c}\lset^{(q)}\big|
   -    \big| \wset^1_{b_1-1} \big| - \big| \wset^2_{b_2} \big|
  \geq c_\kapset \nn/4- c_\kapset \nn/8 =c_\kapset \nn/8.
\end{align*}
Thus  $b_2-b_1+1\geq {c_\kapset \aaa d^c}/{8C'}$.
Now applying Lemma~\ref{simple offset}, we get
\begin{align*}
\of\geq
c_{\ref{simple offset}}\sum\limits_{b=b_1}^{b_2-1}
\min\bigl(|\wbl|,|\wbg|\bigr)
\geq c_{\ref{simple offset}}(b_2-b_1) c_\kapset \nn/16
\geq  c'\, d^{c}n,
\end{align*}
where $c'>0$ is an absolute constant.
Since $d$ is large enough, we obtain that
$C^n d^{n}2^{-\of/2} \leq d^{-n},$
which completes the proof.
\end{proof}

We turn now to a particular class of vectors in $\rhoset_v$ for some $v\geq 5$.
Recall that for $x\in\rhoset_u$, the total cardinality of spread and regular
$\ell$-parts in the $\ell$-decomposition with respect to the $d^u$-approximation
of $x$ with heights not smaller than  $c_\rhoset 2^{c_\rhoset (u-4)\aaa} \aaa$
is at least $c_\rhoset \nn$, where  $c_\rhoset<1$ is a universal constant.
For every integer $v\geq 5$ and every positive numbers $\delta, \rho$ we define
$$
\rhoset_{v,\rho,\delta}:=\bigl\{x\in \rhoset_v:\,
\exists \lambda\in\C\mbox{ such that }|\{i\leq n:\,\vert x_i-\lambda\vert \leq \rho\}|\geq \delta n\bigr\}.
$$
We start with a useful property of vectors in this set.

\begin{lemma}\label{lem: basic-rhoset}
Let $v\geq 5$, $\rho\leq d^{-v}$, and $\delta \in (0,1)$.
Let $x\in \rhoset_{v,\rho,\delta}$ and $y$ be its $d^v$-approximation.
Then in the $\ell$-decomposition of $y$,
there exists a $w$-set $\wset_{b}$ of order  $b\leq \log_2 ({72\sqrt{d}}/{\delta})$
and of cardinality at least $\delta n/36$.
\end{lemma}

\begin{proof}
Let $x\in \rhoset_{v,\rho,\delta}$ and $y$ be its $d^v$-approximation.
Let $\lambda=\lambda(x)\in \C$ be such that
$$
 |\{i\leq n:\,\vert x_i-\lambda\vert \leq \rho\}|\geq \delta n.
$$
Note that if $\vert x_i-\lambda\vert \leq\rho$, then since $d^v\rho<1$ we have
$$
 \Re (d^v y_i)\in \{\lfloor \Re(d^v\lambda)\rfloor-1, \lfloor \Re(d^v\lambda)\rfloor, \lfloor \Re(d^v\lambda)\rfloor+1\}
$$
and
$$
 \Im (d^v y_i)\in \{\lfloor \Im(d^v\lambda)\rceil-1, \lfloor \Im(d^v\lambda)\rfloor, \lfloor \Im(d^v\lambda)\rfloor+1\},
$$
which means that $d^vy_i$  can take at most $9$ possible values. This implies the existence of a set
$I\subset [n]$ of size at least $\delta n/9$ such that $y_i=y_j$ for all $i,j\in I$.
Let $a=\lfloor\log_2 (1+\delta n/9)\rfloor-1$ so that
$$\sum_{i=0}^a2^{i}\leq \vert I\vert. $$
From the construction of the $\ell$-decomposition of $y$, at the step $j=a$
the level set $L(j, y(I))$ is of size at least $2^a\geq \delta n/36$.
This implies  the existence of an $\ell$-part of size at least $\delta n/36$ and, by (\ref{levsethei}),
of height
at most $n/2^{a-1} \leq 72/\delta$. Since for every $q$, $\widetilde w_q\leq \hq\sqrt{d}$,
there exists a $w$-set of order at most $\log_2 (72\sqrt{d}/{\delta})$ and of cardinality
at least $\delta n/36$.
\end{proof}

Next we estimate  the product of $\est_i$ for
approximations of vectors from $\rhoset_{v,\rho,\delta}$.
\begin{lemma}\label{rhoset lemma}
Let $v\geq 5$ be an integer, $0\leq \rho\leq d^{-v}$ and $0\leq \delta\leq 36c_\rhoset \aaa$ be such that
$$
c_{\rhoset}\aaa (v-4) \geq 2\log_2  d+ 2\log_2({72\sqrt{d}}/{\delta})+ 2-\log_2 ({c_{\rhoset}\aaa}).
$$
Further, let $y$ be the $d^v$-approximation of a vector in $\rhoset_{v,\rho,\delta}$ and
$(\lset^{(q)})_{q=1}^m$ be its $\ell$-decomposition.
Finally, let $\RR=(\RR_{iq})$ be a $y$-admissible matrix in $\RSet_{n,m,d}^{ST}$. Then
$$\prod\limits_{i=1}^n\est_i(y,k,\RR)
\leq (C d)^n 2^{-c \delta n v}
(n!)^{-1}\prod_{q\leq m} \frac{|\lset^{(q)}|!}{\hq^{|\lset^{(q)}|}},$$
where $C>c>0$ are universal constants.
\end{lemma}

\begin{proof}
Applying Lemmas~\ref{lem: precise-bound-triv-est} and~\ref{offset lemma} together with \eqref{eq: trivial-prod}, we get
$$
\prod\limits_{i=1}^n\est_i(y,k,\RR)\leq (C\, d)^n 2^{-\of/2}(n!)^{-1}\prod_{q\le m} \frac{|\lset^{(q)}|!}{\hq^{|\lset^{(q)}|}},
$$
for a positive absolute constant $C$. Let
$$
b_1:=\log_2 ({72\sqrt{d}}/{\delta})\quad \text{ and } \quad b_2:=\log_2\left(c_{\rhoset}\aaa2^{(v-4)c_{\rhoset}\aaa}\right)-\log_2 d.
$$
By the assumptions of the lemma,  $b_2-b_1\geq {c_{\rhoset}\aaa (v-4)}/{2}$.

By Lemma~\ref{lem: basic-rhoset}, there exists a $w$-set of order at most $b_1$ and
of cardinality at least $\delta n/36$. On the other hand,
using the definition of $\rhoset_v$ and
 \eqref{eq: trivial-truncated-weight}, the total cardinality
of $w$-sets of order at least $b_2$, is at least $c_\rhoset \aaa n$.
Therefore, for every integer $b$ in the range $[b_1,b_2)$ we have
$$
  \min\bigl(|\wbl|,|\wbg|\bigr)\geq \delta n/36.
$$
Now, we apply Lemma~\ref{simple offset} to deduce that
\begin{align*}
\of&\geq c_{\ref{simple offset}}\sum\limits_{b_1\leq b< b_2}
\min\bigl(|\wbl|,|\wbg|\bigr)
\geq  c_{\ref{simple offset}}\delta n  (b_2-b_1)/36 \geq (c_{\rhoset} c_{\ref{simple offset}}\aaa/72) \, \delta\, n\, (v-4),
\end{align*}
which implies the desired bound.
\end{proof}

We are now ready to state and complete the proof of a generalization of Theorem~\ref{ker th}.
\begin{theor}[Structural theorem]\label{ker th gen}
There exist absolute positive constants $c$, $c'$, and  $C$ such that the following holds.
Let $d, n$ be a large enough integers satisfying $d\leq \exp(\sqrt{c'\ln n})$.
Let  $z\in\C$ be such that $\vert z\vert \leq \sqrt{d}\, \ln d$.
Let $1\leq L\leq n/d^3$ and let $r_0$ be the smallest integer such that
$p^{r_0}\geq 20L/d$, where $p=\lfloor (1/5)\sqrt{d/\ln d}\rfloor$.
Let $\KK\subset[n]$ satisfy $\vert \KK^c\vert\leq L$ and assume that
$$
  \max( n^{-c}, e^{-c n/\vert \KK^c\vert})\leq \rho\leq e^{-C \ln^2 d}\quad
  \text{ and }\quad \delta=C \frac{\ln^2 d}{\ln (1/\rho)}.
$$
Then with probability at least $1-1/n$  any non-zero vector $x\in\C^n$
with the property that
$$
 \Vert (M-\ww \idmat)^\KK x\Vert_2 \leq  {L^3}n^{-6} \Vert x\Vert_2
$$
satisfies one of the two conditions:
\begin{itemize}

\item (Gradual with many levels) One has 
$$
 x_i^*\leq \left\{
\begin{array}{ll}
\big(n/i\big)^{3}\,\, x_{p^{r_0}}^* & \mbox{ if }\,\, i\leq p^{r_0},
\\
 d\big(n/i\big)^3 x_{\nn}^* & \mbox{ if }\,\,  p^{r_0}\leq i\leq n_1,
\\
  d^3 x_{\nn}^* & \mbox{ if }\,\, n_1\leq i\leq \nn,
\end{array}
\right.
$$
and
$$
\Big|\Big\{i\leq n:\,|x_i-\lambda |\leq \rho x_{\nn}^* \Big\}\Big|\leq \delta n \quad \quad
\mbox{ for all } \quad \lambda\in\C.
$$
\item (Very steep) $x_i^* > 0.9 (n/i)^3 x_{p^{r_0}}^*$ for some  $i\le p^{r_0}$.
\end{itemize}
\end{theor}

\begin{proof}
Choose $v$ from $\rho=d^{-v}$. Without loss of generality we assume that $v$ is
an integer. Then $\delta = Cv^{-1}\log_2  d$ and
$$
 C\ln d \leq v\leq  c\,  \min\big(\log _d n, \frac{n}{\vert \KK^c\vert \ln d}\big),
$$
where $C$ is a large enough absolute constant and $c\in(0, 1/2)$
is a small enough absolute constant ($c=1/6$ works). Note that the left hand side of this
inequality is always smaller than the right hand side, provided that $d\leq \exp(\sqrt{c'\ln n})$
with $c'=c/C$. Then, using that $d$ is large enough we have
$$
  d^v\leq \min\big( \sqrt{n}/(8d^{3/2} \sqrt{\ln d}),
  d^{-10} e^{n/5\vert \KK^c\vert}\big),
$$
in particular, we may apply
Lemma~\ref{l:equiv classes} and Proposition~\ref{prop sbp} with $k\leq d^v$.
Note also that the assumptions of
Lemma~\ref{rhoset lemma} are satisfied as well.

Let $\Gamma_{\rho,\delta}$ be the set of non-zero vectors  satisfying none of the two
conditions in Theorem~\ref{ker th gen}. For every $x\in \Gamma_{\rho,\delta}$, define
$$
\Event_x:=\Big\{M\in\Mc:\,
\, \Vert (M-\ww \idmat)^\KK x\Vert_2\leq \frac{L^3}{n^6}\Vert x\Vert_2\Big\}.
$$
Since the event $\Event_x$ is homogeneous in $x$, we may restrict $\Gamma_{\rho,\delta}$
to vectors satisfying $x_{\nn}^*=1$ (if $x_{\nn}^*=0$, we consider a slight perturbation of $x$).

Our goal is to show that $\Prob(\bigcup_{x\in \Gamma_{\rho,\delta}} \Event_x)\leq 1/n$.
Recall that we decomposed  $\C^n$ into the set of almost constant vectors,
denoted by $\BB:=\BB(\theta_0)$ with
$\theta _0 = 10/d^3$, the set of steep vectors,
denoted by $\st$ (note, an almost constant vector can be also steep), and the set of gradual vectors,
denoted by $\mathcal{S}= \C^n\setminus(\BB \cup \st)$.
If $x$ doesn't satisfy the second condition of the theorem then, by (\ref{t3shift}),
$x\not\in \st_{3}^\CC$, that is $\Gamma_{\rho,\delta}\cap \st_{3}^\CC=\emptyset$.
Note also that
$$
 \mathcal{S}^c \setminus \st_{3}^\CC \subset \BB_0 := (\BB\setminus \st_{3}^\CC) \cup \st _\CC.
$$
Therefore, applying Theorem~\ref{th: steep and ac} we obtain
\begin{align*}
  \Prob\Big(\bigcup_{x\in \Gamma_{\rho,\delta}} \Event_x\Big)
  & \leq \Prob\Big(\bigcup_{x\in \BB_0\cap \Gamma_{\rho,\delta}} \Event_x\Big) +
  \Prob\Big(\bigcup_{x\in  \mathcal S \cap \Gamma_{\rho,\delta}} \Event_x\Big)
\\&
   \leq \exp\big(-(\ln{d})\, (\ln n)/20\big)+ \Prob\Big(\bigcup_{x\in
           \mathcal S\cap \Gamma_{\rho,\delta}} \Event_x\Big).
\end{align*}

We now show that $\Gamma_{\rho,\delta}\cap \rhoset_v\subset\rhoset_{v,\rho,\delta}$.
Indeed, a vector in $\Gamma_{\rho,\delta}$ does not belong in particular to $\st_{3}$.
Moreover, since $\rhoset_v \subset \mathcal S\subset \st^c$,  by Lemma~\ref{l:decay} every $x\in\rhoset_v$
satisfies $x_i^*\leq d \big(n/i)^3 x_{\nn}^*$ for all $p^{r_0}\leq i\leq n_1$ and
$x_{i}^*\leq d^3 x_{\nn}^*$ for all $n_1\leq i\leq \nn$.
Therefore, if $x\in\Gamma_{\rho,\delta}\cap \rhoset_v$ then it cannot satisfy the last condition
in ``gradual with many levels," which means that $x\in \rhoset_{v,\rho,\delta}$

This together with Theorem~\ref{kappa and rho} and the union bound gives
\begin{equation*}\label{eq2: ker th}
\Prob\Big(\bigcup_{x\in  \mathcal S\cap \Gamma_{\rho,\delta}} \Event_x\Big)
\leq \sum_{u=4}^v \Prob\Big(\bigcup_{x\in \kapset_u} \Event_x\Big) + \Prob\Big(\bigcup_{x\in \rhoset_{v,\rho,\delta}} \Event_x\Big).
\end{equation*}
Applying Proposition~\ref{prop: approx} with  $r=\ln d$ and $k=d^u$ (or $k=d^v$) and using
$d^{1+u}\leq d^{1+v}\leq \sqrt{n}$,  we get that there exists an
absolute constant $C'>0$ such that for any $5\leq u\leq v$,
\begin{align*}
\Prob\Big(\bigcup_{x\in \kapset_u} \Event_x\Big)\leq  2 d^2 \, \sum_{ y\in \kset_{d^u}(\kapset_u)}\, \,
\sup_{\xyzv\in \C^{\vert \KK\vert}}\Prob\Big\{&M\in \Mc:\,
\| M^\KK y+\xyzv\|_2\leq C'\, \ln d \, \frac{\sqrt{dn}}{d^u}\Big\}+n^{-100},
\end{align*}
and
\begin{align*}
 \Prob\Big(\bigcup_{x\in \rhoset_{v,\rho,\delta}} \Event_x\Big)
 \leq  2 d^2  \, \sum_{y\in \kset_{d^v}(\rhoset_{v,\rho,\delta})}\, \,
 \sup_{\omega\in \C^{\vert \KK\vert }}\Prob\Big\{&M\in \Mc:\,
\| M^\KK y+\xyzv\|_2\leq C'\, \ln d\,  \frac{\sqrt{dn}}{d^v}\Big\}+n^{-100}.
\end{align*}
Take any $4\leq u\leq v$ and
fix for a moment $y\in \kset_{d^u}(\kapset_u)$, the set of $k$-approximations of vectors in $\kapset_u$.
Assume that its $\ell$-decomposition consists of $m$ sets $(\lset^{(q)})_{q=1}^m$.
 Proposition~\ref{prop sbp} applied with $\gamma = C \ln d \sqrt{n/|K|}$ and  Lemma~\ref{kapset lemma}
 imply
$$
\sup_{\xyzv\in \C^{\vert \KK\vert}}\Prob\bigl\{M\in \Mc:\,
\| M^\KK y+\xyzv\|_2\leq C'\, (\ln d)\,  \sqrt{dn}/d^u\,|\,\Event_{\ref{graph prop}}\bigr\}
\leq  \frac{e^{-2n}}{n!}\prod_{q\le m} \frac{|\lset^{(q)}|!}{\hq^{|\lset^{(q)}|}},
$$
provided that $d$ is large enough.

Let $\mathcal C$ be the equivalence class in $\kset_{d^u}$ generated by $y$.
By Lemma~\ref{equiv cardinality} we have
\begin{align*}
 \sum_{ \widetilde y\in \mathcal C}\,\, \sup_{\xyzv\in \C^{\vert \KK\vert}}\Prob\bigl\{M\in \Mc:\, \,
 \|M^\KK\widetilde y+\xyzv\|_2\leq C'\, (\ln d)\, \sqrt{dn}/d^u
 \,|\,\Event_{\ref{graph prop}}\bigr\}&\leq \exp(-2n).
\end{align*}
Finally, Lemma~\ref{l:equiv classes} implies
\begin{align*}
\sum_{ y\in \kset_{d^u}(\kapset_u)}\, \,
\sup_{\xyzv\in \C^{\vert \KK\vert}}\Prob\bigl\{M\in \Mc:\, \,
\| M^\KK y+\xyzv\|_2&\leq C'\, (\ln d)\, \sqrt{dn}/d^u \,|\,\Event_{\ref{graph prop}}\bigr\}\leq e^{-n}.
\end{align*}
Repeating the above argument for vectors in ${\kset}_{d^v}(\rhoset_{r,v,\rho})$ with
Lemma~\ref{rhoset lemma} instead of  Lemma~\ref{kapset lemma} and using that $\delta =C (\log _2 d)/v$
with large enough $C$,  we get
\begin{align*}
  \sum_{y\in \kset_{d^v}(\rhoset_{v,\rho,\delta})}\, \,
\sup_{\xyzv\in \C^{\vert \KK\vert}}\Prob\bigl\{M\in \Mc:\,  \,
\| M^\KK y+\xyzv\|_2&\leq  C\, (\ln d)\, \sqrt{dn}/d^v
 \,|\,\Event_{\ref{graph prop}}\bigr\}\leq e^{-n}.
\end{align*}
Applying Proposition~\ref{graph prop} to remove the conditioning from the
two previous estimates and combining bounds, we obtain
$$
\Prob\Big(\bigcup_{x\in  \mathcal S\cap \Gamma_{\rho,\delta}} \Event_x\Big)
\leq 2 v d^2e^{-n}+\frac{v}{n^{100}}.
$$
The proof is finished by the choice of $v$.
\end{proof}
Note once more that the probability bound can be made $1-n^{-\kappa}$ for any fixed $\kappa\geq 1$ at the expense of having worse constants.

\bigskip

Finally we prove  Theorem~\ref{ker th}   and  Corollary~\ref{cor: deloc}.

\begin{proof}[{Proof of Theorem~\ref{ker th}}]
Set $L=\max(1, \vert \KK^c\vert)$ and
 $$
 \rho= \max\Big(n^{-c}, \exp\Big(-\big(n/(1+\vert \KK^c\vert)\big)^{c\ln\ln d/\ln d}\Big)\Big)
 $$
 for an appropriate positive constant $c$. Given $d^{-1/2}\leq a\leq 1$, let $q=\max(1, a\vert \KK^c\vert)$.
 Then $p^{r_0}\leq q$ and  $ a\vert \KK^c\vert \leq n_1$. Let $x$ satisfy the dichotomy from Theorem~\ref{ker th gen}.
 If $x$ is very steep in the sense of Theorem~\ref{ker th gen} then, using $p^{r_0}\leq q$, we get that
 $x$ is very steep in the sense of Theorem~\ref{ker th}. Assume now that $x$ is not very steep in the sense
 of Theorem~\ref{ker th}, i.e., assume that $x_i^*\leq (n/i)^{3} x_{q}^*$ for all $i\leq q$. If in addition,
 $x$ is gradual with many levels in
  the sense of Theorem~\ref{ker th gen} then it is not difficult to see that is gradual with many levels in
  the sense of Theorem~\ref{ker th} provided that $c'=\aaa$. This proves the desired result.
\end{proof}

\bigskip

\begin{proof}[{Proof of Corollary~\ref{cor: deloc}}]
Let $n$, $d$ be as in Theorem~\ref{ker th gen}.
By Theorem~\ref{norm lemma}, there is a universal constant $C>0$ such that
the event
$$\Event:=\Big\{M\in\Mc:\;\Big\|M- \frac{d}{n} {\bf 1}{\bf 1}^t \Big\|\geq C\sqrt{d}\Big\}$$
has probability less than $n^{-2}$. Denote by $V$ the set of all vectors in $\C^n$ having
sum of coordinates equal $0$. Clearly, $V$ is an invariant subspace of $M$.
Let $\lambda_i$, $i\leq n$, be eigenvalues of $M$ arranged so that
$|\lambda_1|\geq|\lambda_2|\geq\dots\geq|\lambda_n|$.
Since $\lambda _1 (M)=d$ corresponds to the eigenvector  ${\bf 1}$, we observe that
all eigenvectors corresponding to $\lambda _i$, $i\geq 2$, belong to $V$. This implies that,
conditioned on $\Event^c$,
we have $|\lambda_i(M)|<C\sqrt{d}$, $i\geq 2$.
Now, let $\Net$ be a fixed $(1/(C\sqrt{d}n^6)$-net in the disk of radius $C\sqrt{d}$
of the complex plane (we assume the usual Euclidean metric on $\C$).
Clearly, $\Net$ can be chosen so that $|\Net|\leq n^{13}$.
For any point $z'\in\Net$, applying Theorem~\ref{ker th gen} with
$K:=[n]$, $L=1$, $\rho:=n^{-c}$, $\delta:=\ln^2 d/\ln n$, we get
that with probability at least $1-n^{-15}$ every unit complex vector $x$ satisfying $\|(M-z'\,\idmat)x\|_{2}<n^{-6}$,
is ``gradual with many levels'' (since $q=1$, there are no ``very steep'' vectors).  Now, observe that for any matrix $M\in\Mc$, any eigenvalue $\lambda$ of $M$
satisfying $|\lambda|\leq C\sqrt{d}$ and a corresponding normalized eigenvector $x$, we necessarily have
$\|(M-z'\,\idmat)x\|_{2}<n^{-6}$ for some $z'=z'(\lambda)\in\Net$. Combining
this with the last remark, we get that the event
\begin{align*}
\Event':=\Big\{&M\in\Mc: \mbox{any normalized eigenvector of $M$ with eigenvalue}\\
&\mbox{in the disk of radius $C\sqrt{d}$ is ``gradual with many levels''}
\Big\}
\end{align*}
has probability at least $1-n^{-2}$. Here, ``gradual with many levels''
means that the vector satisfies the conditions listed in the corollary.
Finally note, that conditioned on $\Event^c$ all eigenvectors except for $(1/\sqrt{n},\dots,1/\sqrt{n})$,
have corresponding eigenvalues in the disk of radius $C\sqrt{d}$.
Thus, all matrices in $\Event^c\cap \Event'$ satisfy the assertion of the statement. The result follows.
\end{proof}

\subsection*{Acknowledgments}

We are grateful for  referees for careful reading and many suggestions, which
help us to improve presentation. In particular, for showing us relatively short
direct proofs of Propositions~\ref{p:complexLevy} and \ref{prop-simple}.
A significant part of this work was completed while the last three named authors
were in residence at the Mathematical Sciences Research Institute in Berkeley, California,
supported by NSF grant DMS-1440140, and the first two named authors visited the institute.
The hospitality of MSRI and of the organizers of the program on Geometric Functional Analysis
and Applications is gratefully acknowledged. The research of the last named author was
partially  supported by  grant  ANR-16-CE40-0024-01.

\address

\end{document}